\documentclass{siamltex}
\usepackage{amsfonts}
\usepackage{amsmath, amssymb}
\usepackage[all]{xy}
\usepackage[usenames,dvipsnames]{xcolor}
\usepackage{graphicx, tikz, moresize}
\usetikzlibrary{trees}
\usetikzlibrary{arrows,shapes,snakes,automata,backgrounds,petri}
\usepackage[normalem]{ulem}
\definecolor{vugold}{RGB}{237,160,0}
\definecolor{wvublue}{RGB}{0,69,156}
\definecolor{ulightgray}{gray}{.9}
\usepackage{enumitem}

\makeatletter
\newcommand*{\textoverline}[1]
{$\overline{\hbox{#1}}\m@th$}
\makeatother

\makeatletter
\newcommand*{\apost}[1]
{$\hbox{#1}'\m@th$}
\makeatother

\makeatletter
\newcommand*{\apostt}[1]
{$\hbox{#1}''\m@th$}
\makeatother

\newcommand*{\simneq}{%
\mathrel{\vcenter{\offinterlineskip
\hbox{$\sim$}\vskip0.15ex\hbox{$=$}\vskip-1.15ex\hbox{\hspace{0.35ex}$\scriptscriptstyle{/}$}}}}

\newcommand*{\rst}{%
\ensuremath{\circ^{\scriptstyle{r}}}}

\newcommand*{\twost}{%
\ensuremath{\circ^{\scriptstyle{2}}}}
\newcommand*{\nst}{%
\ensuremath{\circ^{\scriptstyle{n}}}}

\newcommand{\rstt}{%
\mathrel{\ooalign{$\bigcirc$\cr\kern3pt$\scriptstyle{r}$}}}

\newtheorem{thm}{Theorem}
\newtheorem{thm1}{Theorem}[section]
\newtheorem{remark}[thm1]{Remark}
\newtheorem{lemma1}[thm1]{Lemma}

\newtheorem{cor}[thm1]{Corollary}
\newtheorem{def1}[thm1]{Definition}
\newtheorem{example}{Example}
\newtheorem{notation}[thm1]{Notation}

\title{Some results on injectivity and multistationarity in chemical reaction networks}

\author{Murad Banaji\footnotemark[1] \and
Casian Pantea\footnotemark[2]}

\begin{document}

\maketitle

\renewcommand{\thefootnote}{\fnsymbol{footnote}}

\footnotetext[1]{Middlesex University, Department of Design Engineering and Mathematics: {\tt m.banaji@mdx.ac.uk}.}
\footnotetext[2]{West Virginia University, Department of Mathematics.}
\renewcommand{\thefootnote}{\arabic{footnote}}

\begin{abstract}
The goal of this paper is to gather and develop some necessary and sufficient criteria for injectivity and multistationarity in vector fields associated with a chemical reaction network under a variety of more or less general assumptions on the nature of the network and the reaction rates. The results are primarily linear algebraic or matrix-theoretic, with some graph-theoretic results also mentioned. Several results appear in, or are close to, results in the literature. Here, we emphasise the connections between the results, and where possible, present elementary proofs which rely solely on basic linear algebra and calculus. A number of examples are provided to illustrate the variety of subtly different conclusions which can be reached via different computations. In addition, many of the computations are implemented in a web-based open source platform, allowing the reader to test examples including and beyond those analysed in the paper. 
\end{abstract}

\section{Introduction}

In this paper, the term {\bf chemical reaction network (CRN)} will refer to a set of chemical reactions, and also to its description via a system of ordinary differential equations (ODEs). The study of how network structure/topology affects network dynamics, often termed ``chemical reaction network theory'', has a considerable history frequently traced to the pioneering works of Horn and Jackson \cite{hornjackson} and Feinberg \cite{feinberg0}. This area has, however, also seen a recent resurgence of interest; and perhaps the most active strand of recent work involves examining the capacity of CRNs for multiple equilibria. In this context, variants on the following question have been intensively studied: 
\begin{enumerate}
\item[Q1.] Which CRNs forbid multiple equilibria?
\end{enumerate}
In other words, for which CRNs do the vector fields derived from the network forbid more than one equilibrium on some set? Complicating any review of this and related questions is that the set examined may vary, conclusions may be phrased in terms of matrices or graphs associated with the network, and results may be derived under formally similar, but nevertheless subtly different, assumptions. For example, the reactions may be assumed to be occurring in a so-called continuous flow stirred tank reactor (CFSTR) or in a closed chamber; the kinetics may be assumed to be mass action or to belong to some other more general class; the domain examined may be the whole nonnegative orthant, its interior, or individual stoichiometry classes (to be defined later); and so forth. In some cases the question may be not about the possibility of multiple equilibria {\em per se}, but rather of multiple {\em nondegenerate} equilibria (defined later). Closely related to (Q1) is the question: 
\begin{enumerate}
\item[Q2.] Which CRNs are injective?
\end{enumerate}
Namely, when do the vector fields derived from the network necessarily take different values at different locations on some set? In the special case where the value is $0$, (Q2) reduces to (Q1). In other words, where noninjectivity of a vector field amounts to it taking the same value at two distinct locations in its domain, multiple equilibria occur in the case where it takes the particular value $0$ at two distinct locations. That noninjective reaction networks may forbid multiple equilibria under certain assumptions on the reaction rates is shown by example in \cite{craciun1}, and in some of the examples in Section~\ref{secexamples} of this paper. (Q2), like (Q1), becomes precise only once we specify the domain we are examining, the assumptions on the kinetics, etc. Some recent papers which have studied (Q1) and/or (Q2), sometimes alongside other questions, include \cite{craciun, craciun1, CraciunHeltonWilliams,banajiSIAM,banajicraciun,banajicraciun2,Craciun.2010ac,feliuwiufAMC2012,conradiflockerziJMB2012,shinarfeinbergconcord1,shinarfeinbergconcord2,joshishiu,feliuwiufSIADS2013, gnacadja} to cite but a few.

The goal here is to discuss (Q2) and (Q1) and to present known results, developments of existing results, and improvements on existing results. In some cases we point out relationships between results appearing in different papers, where these are obscured by differences in terminology, or minor variations in assumptions. For brevity, the focus is on matrix-theoretic approaches, although graph-theoretic corollaries are touched on at several points. Both general networks and certain special cases are treated in some detail: the latter include so-called simply reversible networks, namely networks of reversible reactions where no chemical species ever occurs on both sides of the same reaction. Similarly, general kinetics, power-law kinetics, and mass action kinetics are treated (defined formally later). One of our main conclusions is that many results in this area can be seen in a common framework -- for example results on CRNs with mass action kinetics often appear with very different proofs to those on CRNs with more general kinetics. In a sense to be made precise, we show that collective nonsingularity of vector fields associated with a CRN and some choice of kinetics is equivalent to injectivity of these vector fields, which in turn is sometimes equivalent to the absence of multiple equilibria. On the other hand collective nonsingularity also has elegant combinatorial characterisations. In the spirit of \cite{gunawardenaCRNT}, we find that the machinery of linear algebra, calculus, and a little convex analysis suffices for many of the results, and often results in shorter and/or more general proofs than previously available. Algorithmic forms of several of the results are implemented in the open-source web-based CRN analysis tool {\tt CoNtRol} \cite{control}, and a variety of examples are presented based on analysis carried out in {\tt CoNtRol}. 

The paper is structured as follows: the next two sections are set in a general context, developing background material from linear algebra and matrix theory (Section 2), and examining injectivity of functions of the kind arising as vector fields in CRNs, but in a non-CRN specific setting (Section 3). In Section 4, the results of the previous sections are applied to CRNs in a wide range of settings (e.g. under different choices of kinetics, for general or fully reversible networks, for open or closed systems, etc.); schematics summarising some key results are given in Figures \ref{fig:implications} and \ref{fig:implications1}. A series of examples illustrating the subtly different conclusions that are allowed by the results, as well as the limitations of our approach, are given in Section 5. Finally, Section 6 contains concluding remarks and discussion of future work. We have left outside the main body of the paper a selection of definitions, results, and proofs, relevant but not central to the development of the theory given here; these are given in appendices A--F. Some of these results are known, and we only present proofs where they are new and/or considerably simpler than previous proofs.

\section{Background material}
Before treating chemical reactions it is helpful to set out some background material from linear algebra and matrix-theory, and some results on the injectivity of functions. This material is developed in this section and the next, and much of it can be skipped by the reader interested primarily in the later applications to CRNs. However, we remark that it is re-usable in contexts which go beyond the study of chemical reaction networks, and hopefully demonstrates the more general point that work on CRNs throws up questions of theoretical and practical interest going beyond the application itself. For example, proofs of the so-called first ``Thomas conjecture'' \cite{thomas81} on multistationarity both inspired some of the material here and can be derived as easy corollaries of some results presented here. We sometimes preview in these sections the application of various lemmas to results on CRNs, although the precise statements may be deferred.

\subsection{Notation and basic definitions} Some basic matrix-related definitions are introduced. In particular, it is conceptually helpful and notationally elegant to express several of the results to come using (multiplicative) compound matrices and Hadamard products. 

\begin{notation}
Given an undetermined natural number $n$, a boldface $\mathbf{n}$ will refer to the set $\{1, \ldots, n\}$. However $\mathbf{1}$ will refer to the vector of ones, with size determined by context. 
\end{notation}

\begin{notation}[Submatrices and minors of a matrix]
Given a matrix $A \in \mathbb{R}^{n \times m}$ and (nonempty) sets $\alpha \subseteq \mathbf{n}$ and $\beta \subseteq \mathbf{m}$, define $A(\alpha|\beta)$ to be the submatrix of $A$ with rows from $\alpha$ and columns from $\beta$. If $|\alpha| = |\beta|$, then $A[\alpha|\beta] \stackrel{\text{\tiny def}}{=} \mathrm{det}(A(\alpha|\beta))$. $A(\alpha)$ is shorthand for $A(\alpha|\alpha)$, and $A[\alpha]$ means the principal minor $A[\alpha|\alpha]$.
\end{notation}

\begin{def1}[Nonnegative orthant in $\mathbb{R}^n$, facets]
\label{deffacet}
Define $\mathbb{R}^n_{\geq 0}$ to be the nonnegative orthant in $\mathbb{R}^n$ with boundary $\partial \mathbb{R}^n_{\geq 0}$ and interior $\mathbb{R}^n_{\gg 0}$. The closed, codimension $1$, faces of $\mathbb{R}^n_{\geq 0}$ are its {\em facets}. $x, y \in \mathbb{R}^n_{\geq 0}$ will be said to {\em share a facet} if there exists $i \in \mathbf{n}$ such that $x_i=y_i=0$. Observe that the line segment joining $x, y \in \mathbb{R}^n_{\geq 0}$ lies entirely in $\partial \mathbb{R}^n_{\geq 0}$ if and only if $x,y$ share a facet. Sometimes it turns out that a function on some $U \subseteq \mathbb{R}^n_{\geq 0}$ can take the same value at two points $x,y \in U$ only if they share a facet (see Theorem~\ref{thminj} later).
\end{def1}

\begin{notation}[Image of a matrix $A$ and $A$-equivalent points in $\mathbb{R}^n_{\geq 0}$]
The image of $A \in \mathbb{R}^{n \times m}$, a linear subspace of $\mathbb{R}^n$, will be denoted $\mathrm{im}\,A$. Given $x,y \in \mathbb{R}^n_{\geq 0}$ we will write $x \sim^A y$ for $x - y \in \mathrm{im}\,A$ and $x \simneq^A\! y$ for $x - y \in \mathrm{im}\,A\backslash\{0\}$. Clearly $\sim^A$ is an equivalence relation on $\mathbb{R}^n_{\geq 0}$. In the context of CRNs, where $\Gamma$ is the ``stoichiometric matrix'' of the system (to be defined later), each equivalence class of $\sim^\Gamma$ in $\mathbb{R}^n_{\geq 0}$ is a polyhedron termed a ``stoichiometry class''. In the study of many questions related to CRNs we restrict attention to these classes.
\end{notation}

\begin{def1}[Nonnegative/positive matrices and vectors]
Given a real matrix or vector $A$, $A \geq 0$ will mean that each entry of $A$ is nonnegative, and $A > 0$ will mean that $A \geq 0$ and $A \neq 0$. $A \gg 0$ will mean that each entry of $A$ is positive. The inequalities $<$, $\leq$ and $\ll$ will also have their natural meanings. Any $A \geq 0$ will be referred to as {\em nonnegative}, while $A \gg 0$ will be referred to as {\em positive}. 
\end{def1}

We next introduce ``compound matrices'' because these greatly simplify the statement of several definitions and results to follow. We only exploit the notational simplicity they allow, and do not apply any of the extensive theoretical machinery associated with compound matrices in the study of linear algebra and differential equations (e.g., \cite{allenbridges,muldowney}).

\begin{def1}[Multiplicative compound matrices]
Given $A \in \mathbb{R}^{n \times m}$ and $k \in \{1, \ldots, \min\{n, m\}\}$, define $A^{(k)}$ as the {\em $k$th multiplicative compound matrix of $A$} (see \cite{muldowney} for example), namely, choosing and fixing some orderings (say, lexicographic) on subsets of $\mathbf{n}$ and $\mathbf{m}$ of size $k$, $A^{(k)}$ is the ${n\choose k} \times {m\choose k}$ matrix of $k \times k$ minors of $A$. 
\end{def1}

\begin{def1}[Hadamard product]
Given $A, B \in \mathbb{R}^{n \times m}$, define $A \circ B \in \mathbb{R}^{n \times m}$ to be the {\em Hadamard product} (or entrywise product) of $A$ and $B$, namely $(A \circ B)_{ij} = A_{ij}B_{ij}$.
\end{def1}

\begin{notation}[$A \rst B$: Hadamard product of compound matrices]
 We introduce the abbreviation $A \rst B$ for $A^{(r)} \circ B^{(r)}$. This notation will be used extensively and is illustrated in the Example in Remark~\ref{excompHad} to follow.
\end{notation}

\begin{notation}[$\mathcal{D}_n$: positive diagonal matrices]
Define $\mathcal{D}_n \subseteq \mathbb{R}^{n \times n}$ to be the $n \times n$ diagonal matrices with positive diagonal entries, namely $A \in \mathcal{D}_n$ if and only if $A_{ii} > 0$ for $i \in \mathbf{n}$ and $A_{ij} = 0$ for $i,j \in \mathbf{n}$, $i \neq j$. Given $A \in \mathbb{R}^{n \times n}$ (resp., $\mathcal{A} \subseteq \mathbb{R}^{n \times n}$), we write $A + \mathcal{D}_n$ for $\{A + D\colon D \in \mathcal{D}_n\}$ (resp., $\mathcal{A} + \mathcal{D}_n$ for $\{A+D\colon A \in \mathcal{A}, D \in \mathcal{D}_n\}$).
\end{notation}

\begin{def1}[$P$-matrix, $P_0$-matrix]
\label{defP0}
$A \in \mathbb{R}^{n \times n}$ is a {\em $P$-matrix} (resp., {\em $P_0$-matrix}) if all its principal minors are positive (resp., nonnegative), namely if diagonal elements of $A^{(k)}$ are all positive (resp., nonnegative) for each $k = 1, \ldots, n$. 
\end{def1}
\begin{remark}[Characterisation of $P_0$-matrices via collective nonsingularity]
\label{remP0}
$P_0$-matrices can also be characterised as follows: $A \in \mathbb{R}^{n \times n}$ is a $P_0$-matrix if and only if $A+\mathcal{D}_n$ consists of nonsingular matrices (see Remark~3.4 in \cite{banajicraciun2}). 
\end{remark}

\begin{lemma1}[The Cauchy-Binet formula] 
\label{defCauchyBinet}
Given $A \in \mathbb{R}^{n \times m}$ and $B \in \mathbb{R}^{m \times n}$, and any nonempty $\alpha \subseteq \mathbf{n}$, $\beta \subseteq \mathbf{m}$ with $|\alpha| = |\beta|$:
\[
(AB)[\alpha|\beta] = \sum_{\substack{\gamma \subseteq \mathbf{m}\\ |\gamma| = |\alpha|}}A[\alpha|\gamma]B[\gamma|\beta]. 
\]
\end{lemma1}
\begin{proof}See \cite{gantmacher}, for example. \hfill
\end{proof} 

In terms of multiplicative compound matrices, the Cauchy-Binet formula is simply $(AB)^{(k)} = A^{(k)}B^{(k)}$ which is immediate from elementary properties of compound matrices.

\begin{def1}[Qualitative class $\mathcal{Q}(A)$]
$A \in \mathbb{R}^{n \times m}$ determines the {\em qualitative class} $\mathcal{Q}(A) \subseteq \mathbb{R}^{n \times m}$ consisting of all matrices or vectors with the same sign pattern as $A$, i.e., $X \in \mathcal{Q}(A)$ if and only if $(A_{ij} > 0) \Rightarrow (X_{ij} > 0)$; $(A_{ij} < 0) \Rightarrow (X_{ij} < 0)$; and $(A_{ij} = 0) \Rightarrow (X_{ij} = 0)$. The closure of $\mathcal{Q}(A)$ will be written $\mathcal{Q}_0(A)$. Given $A,B \in \mathbb{R}^{n \times m}$, we write $\mathcal{Q}(A) - \mathcal{Q}(B)$ for $\{A'-B'\colon\,A' \in \mathcal{Q}(A), B' \in \mathcal{Q}(B)\}$, $[\mathcal{Q}(A)|\mathcal{Q}(B)]$ for $\{[A'|B']\,\colon\,A'\in\mathcal{Q}(A), B'\in\mathcal{Q}(B)\}$, and so forth. If $\mathcal{A}$ is a set of matrices or vectors, we may write $\mathcal{Q}(\mathcal{A})$ for $\cup_{A \in \mathcal{A}}\mathcal{Q}(A)$.
\end{def1}

\begin{def1}[Semiclass $\mathcal{Q}'(A)$]
\label{defQdash}
Given $A \in \mathbb{R}^{n \times m}$, define the {\em semiclass of $A$}, $\mathcal{Q}'(A) \stackrel{\text{\tiny def}}{=} \{D_1AD_2\colon D_1 \in \mathcal{D}_n, D_2\in \mathcal{D}_m\}$.
\end{def1}

\begin{remark}[Qualitative classes, semiclasses, and when they coincide]
\label{remsemiclass}
Note that $\mathcal{Q}'(A) \subseteq \mathcal{Q}(A)$ and it can be shown that $\mathcal{Q}'(A) = \mathcal{Q}(A)$ if and only if the bipartite graph of $A$ (described in Section~\ref{secgraph}) is a forest, i.e., has no cycles \cite{banajirutherfordtreeprod}. For example, if $A$ is a $2 \times 2$ positive matrix with positive determinant, then $\mathcal{Q}(A)$ includes all $2 \times 2$ positive matrices, whereas all matrices in $\mathcal{Q}'(A)$ have positive determinant, demonstrating that $\mathcal{Q}'(A)$ is a proper subset of $\mathcal{Q}(A)$. In fact, the proofs in \cite{banajirutherfordtreeprod} make it clear that when $\mathcal{Q}'(A) \neq \mathcal{Q}(A)$, $\mathcal{Q}'(A)$ is of lower dimension than $\mathcal{Q}(A)$.
\end{remark}

\begin{def1}[Matrix-pattern]
\label{defmatrixpattern}
A {\em matrix-pattern} $\mathcal{A}$ is a set of matrices defined by equalities or inequalities on the entries of each $A \in \mathcal{A}$ taking one of the forms $A_{ij} = 0$, $A_{ij} > 0$ or $A_{ij} < 0$. Some entries may have no defining equality or inequality, and so we may think of each entry as a ``$+$'' (positive), a ``$-$'' (negative), $0$, or a ``?'' (any real number). A qualitative class is the special case of a matrix-pattern where there are no entries of unknown sign. Given a matrix-pattern $\mathcal{A} \subseteq \mathbb{R}^{n \times m}$ and nonempty $\alpha \subseteq \mathbf{n}, \beta \subseteq \mathbf{m}$, the set $\mathcal{A}(\alpha|\beta) = \{A(\alpha|\beta)\,\colon\,A \in \mathcal{A}\}$ is clearly a matrix-pattern.
\end{def1}

\begin{notation}[$A_{\alpha,\beta}$]
\label{notredmat}
Given $A \in \mathbb{R}^{n \times m}$ and (nonempty) sets $\alpha \subseteq \mathbf{n}$ and $\beta \subseteq \mathbf{m}$, define $A_{\alpha,\beta} \in \mathcal{Q}_0(A)$ via $(A_{\alpha,\beta})_{ij} = A_{ij}$ if $i \in \alpha, j \in \beta$; and $(A_{\alpha,\beta})_{ij} = 0$ otherwise, namely $A_{\alpha,\beta}$ is the matrix $A$ with all entries not in $A(\alpha|\beta)$ set to zero.
\end{notation}

\begin{def1}[Sign nonsingular, sign singular]
$A \in \mathbb{R}^{n \times n}$ is {\em sign nonsingular} if every matrix in $\mathcal{Q}(A)$ is nonsingular. It is {\em sign singular} if every matrix in $\mathcal{Q}(A)$ is singular.
\end{def1}

Characterising sign nonsingular matrices has led to a rich combinatorial literature (\cite{Thomassen1,RobertsonSeymourThomas} for example) and the more general question of understanding when properties of a matrix are invariant over a qualitative class has close connections with the study of CRNs.

We will need the following easy fact whose proof is left to the reader. Either all matrices in a (square) matrix-pattern have determinants of the same sign, or all signs are represented by the determinants of the matrix-pattern:
\begin{lemma1}
\label{lemNSnotSNS}
Let $\mathcal{A}$ be a matrix-pattern consisting of square matrices, and containing $A_1, A_2$ such that $\mathrm{sign}(\mathrm{det}\,A_1) \neq \mathrm{sign}(\mathrm{det}\,A_2)$. Then there exists $A_3 \in \mathcal{A}$ such that $\mathrm{sign}(\mathrm{det}\,A_3) \neq \mathrm{sign}(\mathrm{det}\,A_1)$ and $\mathrm{sign}(\mathrm{det}\,A_3) \neq \mathrm{sign}(\mathrm{det}\,A_2)$.
\end{lemma1}

To preview our interest in qualitative classes and semiclasses in the study of CRNs, we find, for example, that for an irreversible CRN with general kinetics, the matrix of partial derivatives of reaction rate functions explores a qualitative class $\mathcal{Q}(A)$, whereas in the case of mass action kinetics, this matrix explores a semiclass $\mathcal{Q}'(A)$. Convexity of $\mathcal{Q}(A)$ means that convex approaches arise very naturally in the study of CRNs with general kinetics; on the other hand the non-convexity of $\mathcal{Q}'(A)$ in general (though see Remark~\ref{remsemiclass}) suggests that these approaches may not work for mass action kinetics. We see, in Theorem~\ref{gen_powlaw} and subsequent related results, that this limitation is to some extent only apparent.

\subsection{The reduced determinant of a matrix product}

Let $A \in \mathbb{R}^{n \times m}$ have rank $r \geq 1$ and let $B \in \mathbb{R}^{m \times n}$. Given any basis for $\mathrm{im}\,A$ we can write down a square matrix describing the action of the product $A B$ on this basis. Different choices of basis lead to similar matrices, and so it makes sense to refer to the determinant of any such matrix as the ``reduced determinant'' of the product and denote this as $\mathrm{det}_A(A B)$ (see also the ``core determinant'' in \cite{helton}). The construction is provided explicitly in Appendix~\ref{appreduced}. Here we list only the important facts:

\begin{enumerate}[align=left,leftmargin=*]
\item $\mathrm{det}_A(A B) = \sum_{|\alpha| = r}(A B)[\alpha]$. In other words, the reduced determinant is the sum of the $r \times r$ principal minors of $A B$. We observe that (i) $\mathrm{det}_A(A B) = \sum_{|\alpha| = |\beta| = r}A [\alpha|\beta]B[\beta|\alpha] = \mathrm{trace}\,(A^{(r)} B^{(r)}) = \sum_{i,j}(A \rst B^t)_{ij}$ using Cauchy-Binet, and (ii) $\mathrm{det}_A(A B) = (-1)^ra_{n-r}$ where $a_k$ is the coefficient of $\lambda^k$ in the characteristic polynomial $\mathrm{det}(\lambda\,I - A B)$. 
\item $\mathrm{det}_A(A B) \neq 0$ if and only if $\mathrm{rank}(A B A) = r$. This is proved as Lemma~\ref{lemrank} in Appendix~\ref{appreduced}. 
\end{enumerate}
The first result is important because, for fixed $A$, the quantity $\sum_{|\alpha| = r}(A B)[\alpha]$ is a polynomial in the entries of $B$; if these entries vary, and we wish to make the claim that $\mathrm{det}_A(A B) \neq 0$ for all allowed $B$, this reduces to an algebraic claim about the nonvanishing of this polynomial on its domain. The second claim is almost obvious given the meaning of $\mathrm{det}_A(A B)$: we expect $\mathrm{det}_A(A B) = 0$ if and only if $\mathrm{im}\,A$ intersects $\mathrm{ker}\,(A B)$ nontrivially which occurs if and only if $\mathrm{rank}(A B A) < r$. We summarise some equivalent ways of regarding the condition $\mathrm{det}_A(A B) \neq 0$, at the heart of many results in this paper, where the equivalences follow straightforwardly from basic linear algebra:
\begin{enumerate}[align=left,leftmargin=*]
\item $\mathrm{rank}(A B A) = \mathrm{rank}\,A$.
\item $\mathrm{im}\,BA \oplus \mathrm{ker}\,A = \mathbb{R}^m$.
\item $\mathrm{im}\,A \oplus \mathrm{ker}\,AB = \mathbb{R}^n$. 
\item $\left.AB\right|_{\mathrm{im}\,A}\colon \mathrm{im}\,A \to \mathrm{im}\,A$ is a homeomorphism.
\item If $0$ is an eigenvalue of $AB$, then it is not ``defective'', namely it has the same algebraic and geometric multiplicity (this follows as $\mathrm{det}_A(A B) = (-1)^ra_{n-r} \neq 0$, and $n-r$ is the dimension of $\mathrm{ker}\,AB$).
\end{enumerate}

If the reader wishes to fix a single meaning for $\mathrm{det}_A\,AB \neq 0$, it is that $AB$ is a nonsingular transformation on $\mathrm{im}\,A$.
\begin{def1}[$A$-nonsingular]
Given $A \in \mathbb{R}^{n \times m}$ and $B \in \mathbb{R}^{m \times n}$, we will say that {\em $B$ is $A$-nonsingular} if $\mathrm{det}_A(A B) \not = 0$ (equivalently, $\mathrm{rank}(A B A) = \mathrm{rank}(A)$). Otherwise {\em $B$ is $A$-singular}. A set $\mathcal{B} \subseteq \mathbb{R}^{m \times n}$ is $A$-nonsingular if each $B \in \mathcal{B}$ is $A$-nonsingular and $A$-singular if each $B \in \mathcal{B}$ is $A$-singular. 
\end{def1}

``Reduced'' Jacobian matrices and reduced determinants are natural objects to consider in the study of systems of ODEs with linear integrals, and CRNs in particular. They appear directly or indirectly in many papers in this area, for example \cite{craciun2,banajidynsys,banajicraciun2,Craciun.2010ac,feliuwiufAMC2012,muellerAll}. They tell us about properties of the linearised system restricted to level sets of the integral. 
\begin{remark}\label{excompHad}The following example illustrates the notion of the reduced determinant of a matrix product, and equivalent ways of computing it. Let
\[
A = \left(\begin{array}{rr}-1 & 0\\1 & -1\\1 & 1\end{array}\right), \,\,\, B = \left(\begin{array}{rrr}-a & b & c\\0 & -d & e\end{array}\right), \,\,\, AB = \left(\begin{array}{rrr}a & -b & -c\\-a & b+d & c-e\\-a & b-d & c+e\end{array}\right)\,.
\]
As $A$ has rank $2$, we can compute $\mathrm{det}_A(A B)$ as a sum of $2 \times 2$ principal minors:
\[
\mathrm{det}_A(A B) = (AB)[\{1,2\}] + (AB)[\{1,3\}] + (AB)[\{2,3\}] = ad + ae + 2(be+cd)\,. 
\]
Alternatively, we also have
\[
A^{(2)} = \left(\begin{array}{r}1\\-1\\2\end{array}\right), \,\,\, B^{(2)} = \left(\begin{array}{ccc}ad & -ae & be+cd\end{array}\right),\,\,\, A \twost B^t = \left(\begin{array}{c}ad\\ ae\\ 2(be+cd)\end{array}\right)\,,
\]
giving, again, $\mathrm{det}_A(A B) = $ sum of entries in $A \twost B^t = ad + ae + 2(be+cd)$. If $a,b,c,d,e > 0$, then $\mathrm{det}_A(A B) >0$, and thus $AB$ acts as an orientation preserving (nonsingular) linear transformation on $\mathrm{im}\,A$.
\end{remark}

\subsection{Graphs associated with matrices and matrix-products}

\label{secgraph}
Graph theoretic approaches to the study of injectivity, and more particularly to injectivity of CRNs, are too extensive to be treated in this paper. However these approaches have a close relationship with the theory described here, both inspiring it, and in some cases deriving from it. We provide some basic definitions in order to be able to state without proof a few graph-theoretic corollaries. We also remark that approaches centred on matrix minors and matrix products as described here lend themselves very naturally to graph-theoretic formulations leaving much to explore in this area.

\begin{def1}[Bipartite graph of a matrix, SR graph of a matrix, DSR graph of a matrix product]
\label{defgraphs}
Given $A \in \mathbb{R}^{n \times m}$ define the {\em bipartite graph of $A$} as follows: $A$ is a graph on $n + m$ vertices, with vertices $\{X_1, \ldots, X_n\} \cup \{Y_1, \ldots, Y_m\}$, and with edge $X_iY_j$ present if and only if $A_{ij} \neq 0$. Edge $X_iY_j$ is given the sign of $A_{ij}$. To get the {\em SR graph of $A$}, $G_A$, as described in \cite{banajicraciun}, edge $X_iY_j$ in the bipartite graph of $A$ is also labelled with the magnitude of $A_{ij}$. Similarly, given $A \in \mathbb{R}^{n \times m}$, $B \in \mathbb{R}^{m \times n}$, associated with the product $AB$, is a bipartite generalised graph $G_{A,B}$ with signed, labelled edges some of which may be directed, termed the {\em directed SR graph} or {\em DSR graph} of $AB$ \cite{banajicraciun2}. The construction is provided in Appendix~\ref{appstar}, but note that if $B$ varies over a qualitative class then $G_{A, B}$ is constant. SR graphs are a special case of DSR graphs.
\end{def1}

\begin{remark}[SR and DSR graphs]
The original construction of the ``species-reaction graph'' for a CRN is given in Craciun and Feinberg \cite{craciun1}. The abstract constructions of SR and DSR graphs above follow Banaji and Craciun \cite{banajicraciun2,banajicraciun}. While these generalised graphs are defined for matrices and matrix products and appear to have no connection with CRNs, they can still naturally be associated with CRNs, as described in Appendix~\ref{appstar}. Examination of their properties plays a part in many results on CRNs, including  results on injectivity and multistationarity \cite{craciun1, banajicraciun2, banajicraciun}, but not restricted to these (see \cite{angelibanajipantea} for results connected with Hopf bifurcation and the possibility of oscillation for example). Drawing and some analysis of the DSR graph of a CRN are automated in {\tt CoNtRol} \cite{control}. 
\end{remark}

\subsection{Compatibility of matrices and related notions}

In the study of injectivity to follow we will frequently be concerned with the determinant, minors, or reduced determinant of a matrix product. In this context we define various important relationships between the sign patterns of compound matrices of a pair of matrices:
\begin{def1}[Compatibility and related notions]
\label{defcompat}
Given a pair of matrices $A, B \in \mathbb{R}^{n \times m}$ and $r \in \{1, \ldots, \min\{n,m\}\}$, $A$ and $B$ will be termed 
\begin{itemize}[align=left,leftmargin=*]
\item {\em $r$-compatible} if $A \rst B \geq 0$;
\item {\em $r$-strongly compatible} if $A \rst B > 0$;
\item {\em $r$-strongly negatively compatible} if $A \rst B < 0$;
\item {\em compatible} if $A \rst B \geq 0$ for each $r = 1, \ldots, \min\{n, m\}$. We abbreviate this as $A \Bumpeq B$.
\end{itemize}
Observe that these relations are not transitive; for example, $A \rst B \geq 0$ and $B \rst C \geq 0$ does not imply that $A \rst C \geq 0$. The notation may be applied to sets of matrices so, for example, if $\mathcal{B}$ is a set of matrices then $A \rst \mathcal{B} > 0$ will mean $A \rst B > 0$ for all $B \in \mathcal{B}$.
\end{def1}

\begin{remark}[Invariance of compatibility notions under row/column reordering]
\label{remreord}
We will frequently use without comment the fact that given $A, B \in \mathbb{R}^{n \times m}$, applying an arbitrary permutation to the rows/columns of $A$, and the same permutation to the rows/columns of $B$ does not alter compatibility relationships such as $A \rst B \geq 0$, $A \rst B > 0$, etc. In other words, if $P_1$ and $P_2$ are permutation matrices of appropriate dimension, then $A \rst B > 0 \Leftrightarrow P_1 AP_2 \rst P_1BP_2 > 0$, and so forth.
\end{remark}

The reader may confirm that if $a,b,c,d, e > 0$, the matrices $A$ and $B^t$ in the example of Remark~\ref{excompHad} are $2$-strongly compatible, so we can write $A \twost B^t > 0$. As they are both also  $1$-compatible and $2$-compatible, they are compatible, namely $A \Bumpeq B^t$. Clearly, if $n \neq m$, then $n$-compatibility does not imply $m$-compatibility: for example, if
\[
A = \left(\begin{array}{rr}-1&-1\\1&0\\0&1\end{array}\right) \quad \mbox{and} \quad B = \left(\begin{array}{rr}-1&0\\1&1\\0&0\end{array}\right), \quad \mbox{then} \quad A \circ B = \left(\begin{array}{rr}1&0\\1&0\\0&0\end{array}\right) > 0,
\]
so $A$ and $B$ are $1$-strongly compatible. But they are $2$-strongly negatively compatible as
\[
A^{(2)} = \left(\begin{array}{r}1\\-1\\1\end{array}\right),\,\,\, B^{(2)} = \left(\begin{array}{r}-1\\0\\0\end{array}\right), \quad \mbox{and} \quad A\twost B = \left(\begin{array}{r}-1\\0\\0\end{array}\right) < 0\,.
\]

The following lemma will prove useful. It provides some elementary consequences of compatibility, and shows how sometimes compatibility of a matrix $A$ with a set of matrices $\mathcal{B}$ is equivalent to compatibility between a new matrix $A'$ and a modified set of matrices $\mathcal{B}'$. Such constructions will allow us to pass easily between claims about sets of irreversible reactions and sets of reactions which are not necessarily irreversible.
\begin{lemma1}
\label{lemrevirrev}
Let $A, B,C,D\in\mathbb{R}^{n \times m}$, and $E,F\in\mathbb{R}^{n \times m'}$. For the first six claims, fix $r \in \{1,\ldots, \min\{n,m\}\}$. For the final claim, fix $r \in \{1,\ldots, \min\{n,m+m'\}\}$. 
\begin{enumerate}[align=left,leftmargin=*]
\item (i) If $A \rst B \geq 0$, then $(AB^t)[\alpha] \geq 0$ for all $\alpha \subseteq \mathbf{n}$ s.t. $|\alpha| = r$.\\ (ii) If $A \rst B \leq 0$, then $(AB^t)[\alpha] \leq 0$ for all $\alpha \subseteq \mathbf{n}$ s.t. $|\alpha| = r$.\\ (iii) If $A \rst B = 0$, then $(AB^t)[\alpha] = 0$ for all $\alpha \subseteq \mathbf{n}$ s.t. $|\alpha| = r$.
\item (i) If $A \rst B > 0$ then $(AB^t)[\alpha] > 0$ for some $\alpha \subseteq \mathbf{n}$ s.t. $|\alpha| = r$. \\(ii) If $A \rst B < 0$ then $(AB^t)[\alpha] < 0$ for some $\alpha \subseteq \mathbf{n}$ s.t. $|\alpha| = r$. 
\item[3.] (i) If $A \rst B \not \geq 0$ then $(AB_1^t)[\alpha] < 0$ for some $B_1 \in \mathcal{Q}'(B)$ and some $\alpha \subseteq \mathbf{n}$ s.t. $|\alpha| = r$.\\ (ii) If $A \rst B \not \leq 0$ then $(AB_1^t)[\alpha] > 0$ for some $B_1 \in \mathcal{Q}'(B)$ and some $\alpha \subseteq \mathbf{n}$ s.t. $|\alpha| = r$.\\ (iii) If $B_1, B_2 \in \mathcal{B} \subseteq \mathbb{R}^{n \times m}$ where $\mathcal{B}$ is path connected, and $A \rst B_1 \not < 0$, $A \rst B_2 \not > 0$, then there exists $B_3 \in \mathcal{B}$ such that $A \rst B_3 \not < 0$ and $A \rst B_3 \not > 0$. 
\item[4.] Each entry of $A \rst (C-D)$ is a sum of entries of $[A|{-A}] \rst [C|D]$. 
\item[5.] (i) $[A|{-A}] \rst [\mathcal{Q}(C)|\mathcal{Q}(D)] \geq 0$ iff $A \rst (\mathcal{Q}(C) - \mathcal{Q}(D)) \geq 0$.\\(ii) $[A|{-A}] \rst [\mathcal{Q}(C)|\mathcal{Q}(D)] \leq 0$ iff $A \rst (\mathcal{Q}(C) - \mathcal{Q}(D)) \leq 0$.
\item[6.] (i) $[A|{-A}] \rst [\mathcal{Q}(C)|\mathcal{Q}(D)] > 0$ iff $A \rst (\mathcal{Q}(C) - \mathcal{Q}(D)) > 0$.\\ (ii) $[A|{-A}] \rst [\mathcal{Q}(C)|\mathcal{Q}(D)] < 0$ iff $A \rst (\mathcal{Q}(C) - \mathcal{Q}(D)) < 0$.
\item[7.] (i) $[A|F|{-A}] \rst [\mathcal{Q}(C)|\mathcal{Q}(E)|\mathcal{Q}(D)] > 0$ iff $[A|F] \rst [\mathcal{Q}(C) - \mathcal{Q}(D)|\mathcal{Q}(E)] > 0$.\\ (ii) $[A|F|{-A}] \rst [\mathcal{Q}(C)|\mathcal{Q}(E)|\mathcal{Q}(D)] \geq 0$ iff $[A|F] \rst [\mathcal{Q}(C) - \mathcal{Q}(D)|\mathcal{Q}(E)] \geq 0$.\\ (iii) $[A|F|{-A}] \rst [\mathcal{Q}(C)|\mathcal{Q}(E)|\mathcal{Q}(D)] < 0$ iff $[A|F] \rst [\mathcal{Q}(C) - \mathcal{Q}(D)|\mathcal{Q}(E)] < 0$.\\ (iv) $[A|F|{-A}] \rst [\mathcal{Q}(C)|\mathcal{Q}(E)|\mathcal{Q}(D)] \leq 0$ iff $[A|F] \rst [\mathcal{Q}(C) - \mathcal{Q}(D)|\mathcal{Q}(E)] \leq 0$.
\end{enumerate}
\end{lemma1}
\begin{proof}
(1), (2), (3): The first two claims are immediate consequences of the Cauchy-Binet formula. (3)(i): suppose that $A[\alpha|\beta]B[\alpha|\beta] < 0$. Observe that $B_{\alpha, \beta}$ (Notation~\ref{notredmat}) lies in the closure of the semiclass $\mathcal{Q}'(B)$, and that $(A(B_{\alpha,\beta})^t)[\alpha] = A[\alpha|\beta]B[\alpha|\beta] < 0$; the claim now follows by choosing any $B_1 \in \mathcal{Q}'(B)$ sufficiently close to $B_{\alpha, \beta}$. 3(ii): In a similar way, if $A \rst B \not \leq 0$ then we find $B_1 \in \mathcal{Q}'(B)$ and $\alpha \subseteq \mathbf{n}$ s.t. $(AB_1^t)[\alpha] > 0$. For (3)(iii), by continuity of determinants, there exists on any path connecting $B_1$ and $B_2$ some $B_3$ such that $A\rst B_3 \not > 0$ and $A \rst B_3 \not < 0$. 
 
For the following three claims, let $\overline A = [A|{-A}]$.

(4) Let $B = C-D$ and $\overline B = [C|D]$. Fix $\alpha \subseteq \mathbf{n}$, $\beta \subseteq \mathbf{m}$ with $|\alpha| = |\beta| = r$ and consider the product $A[\alpha|\beta] B[\alpha|\beta] = \mathrm{det}(A(\alpha|\beta) B^t(\beta|\alpha))$. Clearly
\[
A(\alpha|\beta) B^t(\beta|\alpha) = \overline {A}(\alpha|\beta')\overline{B}^t(\beta'|\alpha),
\]
where $\beta' = (\beta_1, \ldots, \beta_r, \beta_1 + m, \ldots, \beta_r + m)$, and it follows, from the Cauchy-Binet formula applied to the product $\overline {A}(\alpha|\beta')\overline{B}^t(\beta'|\alpha)$, that
\[
A[\alpha|\beta] B[\alpha|\beta] = \sum_{\gamma \subseteq \beta', |\gamma| = r}\overline {A}[\alpha|\gamma]\overline{B}[\alpha|\gamma]\,.
\]
In the following two claims, we prove only part (i); the second part follows similarly.

(5) To see that $[A|{-A}] \rst [\mathcal{Q}(C)|\mathcal{Q}(D)] \geq 0$ implies that $A \rst (\mathcal{Q}(C) - \mathcal{Q}(D)) \geq 0$, take arbitrary $C' \in \mathcal{Q}(C)$ and $D' \in \mathcal{Q}(D)$, and apply (4) to get that $[A|-A] \rst [C'|D'] \geq 0 \Rightarrow A\rst(C'-D') \geq 0$. In the other direction, suppose that  $[A|{-A}] \rst [\mathcal{Q}(C)|\mathcal{Q}(D)] \not \geq 0$. By (3)(i), there exists $\overline{M} =[C'|D'] \in [\mathcal{Q}(C)|\mathcal{Q}(D)]$ and $\alpha \subseteq \mathbf{n}$ such that $(\overline {A}\, \overline{M}^t)[\alpha] < 0$. Setting $M = C'-D'$ gives $\overline {A}\, \overline{M}^t = AM^t$, and so $(A M^t)[\alpha] < 0$, proving, by (1)(i), that $A \rst M \not \geq 0$.

(6) By (5), $[A|{-A}] \rst [\mathcal{Q}(C)|\mathcal{Q}(D)] > 0 \,\Rightarrow\, A \rst (\mathcal{Q}(C) - \mathcal{Q}(D)) \geq 0$; to confirm that the inequality is strict, choose arbitrary $C' \in \mathcal{Q}(C)$ and $D' \in \mathcal{Q}(D)$, and set $M = C'-D'$, $\overline {M} = [C'|D']$. Choose $\alpha, \beta$ such that  $\overline {A}[\alpha|\beta] \,\overline {M}[\alpha|\beta] > 0$ and choose $\beta' \subseteq \{1, \ldots,m\}$ s.t. $|\beta'| = r$, and $\beta \subseteq \beta' \cup \{\beta'_1 + m, \ldots, \beta'_r + m\}$. Now, following the proof of (4), $A[\alpha|\beta']M[\alpha|\beta']$ is a sum of (nonnegative) entries of $\overline A \rst \overline M$ including $\overline {A}[\alpha|\beta] \,\overline {M}[\alpha|\beta]$, and so $A[\alpha|\beta']M[\alpha|\beta'] > 0$. In the other direction, by (5), $A \rst (\mathcal{Q}(C)-\mathcal{Q}(D)) \geq 0 \Rightarrow \overline{A}\rst [\mathcal{Q}(C)|\mathcal{Q}(D)] \geq 0$, and by (4), for any $C' \in \mathcal{Q}(C), D' \in \mathcal{Q}(D)$, $A \rst (C'-D') \neq 0 \Rightarrow \overline{A} \rst [C'|D'] \neq 0$. Together these imply that $A \rst (\mathcal{Q}(C)-\mathcal{Q}(D)) > 0 \Rightarrow \overline{A}\rst [\mathcal{Q}(C)|\mathcal{Q}(D)] > 0$.
(7)(i) and (ii). Let $A' = [A|F]$, $\overline A = [A|F|{-A}]$, and $\overline B = [C|E|D]$. Let $\overline{A}_+ = [A|F|{-A}|{-F}]$, $\overline {B}_{+} = [C|E|D|0]$. The results follows as:
\[
\begin{array}{ccccc}\overline A \rst \mathcal{Q}(\overline B) \geq 0  &  \Leftrightarrow &  \overline {A}_+ \rst \mathcal{Q}(\overline {B}_+) \geq 0 & \Leftrightarrow & A' \rst [\mathcal{Q}(C) - \mathcal{Q}(D)|\mathcal{Q}(E)] \geq 0,\\
\overline A \rst \mathcal{Q}(\overline B) > 0  &  \Leftrightarrow &  \overline {A}_+ \rst \mathcal{Q}(\overline {B}_+) > 0 & \Leftrightarrow & A' \rst [\mathcal{Q}(C) - \mathcal{Q}(D)|\mathcal{Q}(E)] > 0.\end{array}
\]
To see the first equivalence on each line, observe that given any $C' \in \mathcal{Q}(C)$, $D' \in \mathcal{Q}(D)$, and $E' \in \mathcal{Q}(E)$, $[A|F|{-A}|{-F}] \rst [C'|E'|D'|0]$ is simply $[A|F|{-A}] \rst [C'|E'|D']$ with additional zeros. The second equivalences follow from (5) and (6) with $A$ as $[A|F]$, $C$ as $[C|E]$ and $D$ as $[D|0]$. (iii) and (iv) follow similarly. 
\hfill
\end{proof}

\begin{lemma1}[Equivalent formulations of compatibility of two matrices]
\label{lembump}
Let $A, B \in \mathbb{R}^{n \times m}$, and define $\tilde A = [A\,|\,{-I}]$, $\tilde B = [B\,|\,{-I}]$ with $I$ the $n \times n$ identity matrix. Then the following are equivalent: (i) $A \Bumpeq B$, (ii) $\tilde A \nst \tilde B \geq 0$, (iii) $\tilde A \nst \tilde B > 0$, (iv) $\mathrm{det}(\tilde AD \tilde B^t) \geq 0$ for all $D\in \mathcal{D}_{n+m}$ and (v) $\mathrm{det}(\tilde AD \tilde B^t) > 0$ for all $D\in \mathcal{D}_{n+m}$. 
\end{lemma1}
\begin{proof}
Observe that:
\begin{equation}
\label{eqbump}
\tilde A[\mathbf{n}\,|\,\{m\!+\!1, \ldots, m\!+\!n\}]\tilde B[\mathbf{n}\,|\,\{m\!+\!1, \ldots, m\!+\!n\}] =1 > 0,
\end{equation}
and there is a one-to-one correspondence between the remaining products of the form $\tilde A[\mathbf{n}|\beta]\tilde B[\mathbf{n}|\beta]$ (where $\beta \subseteq \{1, \ldots, m+n\}, |\beta|=n$), and the products $A[\alpha'|\beta']B[\alpha'|\beta']$ (where $\alpha' \subseteq \mathbf{n}$, $\beta' \subseteq \mathbf{m}$, $0 \neq |\alpha'|=|\beta'|$). This immediately shows the equivalence of (i), (ii), and (iii). 

The Cauchy-Binet formula gives
\[
\mathrm{det}(\tilde A D \tilde B^t) = \sum_{\substack{\beta \subseteq \{1, \ldots, m+n\}\\ |\beta| = n}}\tilde A[\mathbf{n}|\beta]D[\beta]\tilde B[\mathbf{n}|\beta]
\]
for any $D\in \mathcal{D}_{n+m}$. If $\tilde A \nst \tilde B \geq 0$, then clearly $\mathrm{det}(\tilde A D \tilde B^t) \geq 0$ for all $D\in \mathcal{D}_{n+m}$. Conversely if $\tilde A \nst \tilde B^t \not \geq 0$, then there exists $\beta$ such that $\tilde A[\mathbf{n}|\beta]\tilde B[\mathbf{n}|\beta]<0$. Choosing $D\in \mathcal{D}_{n+m}$ such that $D_{ii} =1$ if $i \in \beta$, and $D_{ii}$ is sufficiently small if $i \not \in \beta$, we can ensure that $\mathrm{det}(\tilde A D \tilde B^t)<0$. This shows the equivalence of (ii) and (iv). 

(iii) implies (v) again follows from the Cauchy-Binet formula, and (v) implies (iv) is trivial. This completes the proof. \hfill
\end{proof}

Lemma~\ref{lembump} tells us that compatibility of two $n \times m$ matrices $A$ and $B$ is equivalent to $n$-strong compatibility of the matrices augmented with ${-I}$, namely $\tilde A$ and $\tilde B$. Further, $A \not \Bumpeq B$, namely $A$ and $B$ fail to be compatible, if and only if there exists $D\in \mathcal{D}_{n+m}$ such that $\mathrm{det}(\tilde AD \tilde B^t) < 0$. For later use we define a condition stronger than $A \not \Bumpeq B$. Unlike $A \not \Bumpeq B$, this next relationship is not symmetric in $A$ and $B$:

\begin{def1}[Strongly incompatible]\label{defstrongincompat}Let $A, B \in \mathbb{R}^{n \times m}$, with $\tilde A = [A|{-I}]$ and $\tilde B = [B|{-I}]$ as in Lemma~\ref{lembump}. $B$ is {\em strongly $A$-incompatible} if there exists $D\in \mathcal{D}_{n+m}$ such that $\mathrm{det}(\tilde AD \tilde B^t) < 0$ and $\tilde AD\mathbf{1} \leq 0$.\end{def1}

The next two results form the basis for several injectivity results in Banaji et al. \cite{banajiSIAM} and below.
\begin{lemma1}
\label{lemP0}
Let $A,B \in \mathbb{R}^{n \times m}$ and $\mathcal{B} \subseteq \mathbb{R}^{n \times m}$, with $\mathcal{B}$ satisfying $\mathcal{B} = \cup_{B \in \mathcal{B}}\mathcal{Q}'(B)$ (namely $\mathcal{B}$ is a union of semiclasses, e.g. a semiclass, a qualitative class or a matrix-pattern). Then (i) $A \Bumpeq B$ implies that $AB^t$ is a $P_0$-matrix; (ii) $A \Bumpeq \mathcal{B}$ if and only if $AB^t$ is a $P_0$-matrix for each $B \in \mathcal{B}$. 
\end{lemma1}
\begin{proof}
(i) $A \Bumpeq B \Rightarrow AB^t$ is a $P_0$-matrix, by the Cauchy-Binet formula (Lemma~\ref{defCauchyBinet}). (ii) The implication to the right is immediate from (i); the implication to the left follows from Lemma~\ref{lemrevirrev}(3)(i).\hfill
\end{proof}

\begin{lemma1}
\label{lemcompat_diag}
Given $A, B \in \mathbb{R}^{n \times m}$ and $r \in \{1, \ldots, \min\{n, m\}\}$:\\
\begin{tabular}{ll} (i) $A\rst B \geq 0 \Leftrightarrow A \rst \mathcal{Q}'(B) \geq 0$ & (ii) $A\rst B \leq 0 \Leftrightarrow A \rst \mathcal{Q}'(B) \leq 0$\\ (iii) $A\rst B > 0 \Leftrightarrow A \rst \mathcal{Q}'(B) > 0$& (iv) $A\rst B < 0 \Leftrightarrow A \rst \mathcal{Q}'(B) < 0$\\ (v) $A\rst B = 0 \Leftrightarrow A \rst \mathcal{Q}'(B) = 0$ & (vi) $A\Bumpeq B \Leftrightarrow A \Bumpeq \mathcal{Q}'(B)$.
\end{tabular}
\end{lemma1}
\begin{proof}
In one direction (to the left) the results are trivial as $B \in \mathcal{Q}'(B)$. In the other direction, the reader can easily confirm from the Cauchy-Binet formula that $(D_1 B D_2)^{(r)} \in \mathcal{Q}(B^{(r)})$ for any $r \in \{1, \ldots, \min\{n, m\}\}$, $D_1 \in \mathcal{D}_n$, and $D_2 \in \mathcal{D}_m$. In other words, pre- and post-multiplication of $B$ by positive diagonal matrices does not change the sign of any minor of $B$. The results then follow immediately from the definition of $A \rst B$. \hfill
\end{proof}

\begin{remark}[Invariance of signs of minors over a qualitative class]
The basis for Lemma~\ref{lemcompat_diag} is that the signs of minors of a matrix remain fixed as we vary within a {\em semiclass}, which can be expressed elegantly as:
\[
(\mathcal{Q}'(B))^{(r)} \subseteq \mathcal{Q}(B^{(r)}),
\]
for any matrix $B$, and so $A \rst \mathcal{Q}'(B) \subseteq \mathcal{Q}(A \rst B)$. As an aside, note that matrices whose minors maintain their signs as we explore a {\em qualitative class} are rather special: $(\mathcal{Q}(B))^{(r)} \subseteq \mathcal{Q}(B^{(r)})$ if and only if each square submatrix of $B$ is either sign nonsingular or sign singular; these are precisely the matrices such that $\mathcal{Q}(B) \Bumpeq \mathcal{Q}(B)$ (i.e., by Lemma~\ref{lemP0}, such that $B_1B_2^t$ is a $P_0$-matrix for each $B_1, B_2 \in \mathcal{Q}(B)$); equivalently those with ``$2$-odd'' bipartite graphs \cite{banajirutherfordtreeprod}, namely those whose SR graphs have no e-cycles (see \cite{banajiJMAA} and Appendix~\ref{appstar}). 
\end{remark}

\subsection{Compatibility and the reduced determinant of a general product}
The lemmas in this section relate the compatibility properties of pairs of matrices, computed by examining signs of their minors, to linear algebraic properties of various associated products. We are particularly interested in making simultaneous claims about sets of matrices, and the emphasis is on a constant first factor and a varying second factor. To preview roughly results to follow, strong compatibility of various matrices related to a CRN, particularly the stoichiometric matrix and the matrix of partial derivatives of the reaction rates, will imply injectivity of associated vector fields.

\begin{lemma1}
\label{lemmain0}
Let $0 \neq A \in \mathbb{R}^{n \times m}$, $\mathcal{B} \subseteq \mathbb{R}^{m \times n}$, and define $r = \mathrm{rank}\,A$. Define the six conditions:
\begin{enumerate}[align=left,leftmargin=*]
\item $A \rst \mathcal{B}^t > 0$ ($A$, $\mathcal{B}^t$ are $r$-strongly compatible).
\label{condrcompat}
\item $A \rst \mathcal{B}^t < 0$ ($A$, $\mathcal{B}^t$ are $r$-strongly negatively compatible).
\label{condrcompatn}
\item $\mathrm{det}_A(A B) > 0$ for each $B \in \mathcal{B}$ ($AB$ has positive reduced determinant).
\label{condredpos}
\item $\mathrm{det}_A(A B) < 0$ for each $B \in \mathcal{B}$ ($AB$ has negative reduced determinant).
\label{condredposn}
\item $\mathrm{rank}(A B A) = r$ for each $B \in \mathcal{B}$ ($\mathcal{B}$ is $A$-nonsingular).
\label{condgenrank}
\item Given any $k \geq 2$, every product of length $k$ of the form $A B_1 A B_2 \cdots$ or $B_1 A B_2 A \cdots$ where $B_i \in \mathcal{B}$, has rank $r$.
\label{condgenrank1}
\end{enumerate}

Then (\ref{condrcompat}) $\Rightarrow$ (\ref{condredpos}) $\Rightarrow$ (\ref{condgenrank}) $\Leftrightarrow$ (\ref{condgenrank1}). Similarly, (\ref{condrcompatn}) $\Rightarrow$ (\ref{condredposn}) $\Rightarrow$ (\ref{condgenrank}). If $\mathcal{B}$ is path connected and a union of semiclasses, then (\ref{condredpos}) $\Rightarrow$ (\ref{condrcompat}), (\ref{condredposn}) $\Rightarrow$ (\ref{condrcompatn}), and (\ref{condgenrank}) $\Rightarrow$ [(\ref{condredpos}) or (\ref{condredposn})].
\end{lemma1}
\begin{proof}
(\ref{condrcompat}) $\Rightarrow$ (\ref{condredpos}). If $A \rst B^t > 0$, then (from above) $\mathrm{det}_A(A B) = \sum_{i,j}(A \rst B^t)_{ij} > 0$. (\ref{condrcompatn}) $\Rightarrow$ (\ref{condredposn}) follows similarly.

(\ref{condredpos}) $\Rightarrow$ (\ref{condrcompat}) if $\mathcal{B} = \cup_{B \in \mathcal{B}}\mathcal{Q}'(B)$. (i) Suppose Condition~\ref{condrcompat} fails in such a way that $A \rst B^t \not \geq 0$ for some $B \in \mathcal{B}$, i.e., $A[\alpha'|\beta']B[\beta'|\alpha'] < 0$ for some $\alpha' \subseteq \mathbf{n}, \beta' \subseteq \mathbf{m}$ with $|\alpha'| = |\beta'| = r$. Then 
\[
\mathrm{det}_A(A B_{\beta',\alpha'}) = \sum_{|\alpha| = |\beta| = r}A[\alpha|\beta]B_{\beta',\alpha'}[\beta|\alpha] = A[\alpha'|\beta']B[\beta'|\alpha'] < 0\,.
\]
$B_{\beta',\alpha'}$ is in the closure of $\mathcal{Q}'(B)$ and by continuity, $\mathrm{det}_A(A B') < 0$ for all $B' \in \mathcal{Q}'(B)$ sufficiently close to $B_{\beta',\alpha'}$. (ii) Suppose instead that there exists $B \in \mathcal{B}$ such that $A \rst B^t  = 0$, i.e., $A[\alpha|\beta]B[\beta|\alpha] = 0$ for all $\alpha \subseteq \mathbf{n}, \beta \subseteq \mathbf{m}$ with $|\alpha| = |\beta| = r$. Then $\mathrm{det}_A(A B) = \sum_{i,j}(A \rst B^t)_{ij} = 0$. (\ref{condredposn}) $\Rightarrow$ (\ref{condrcompatn}) if $\mathcal{B} = \cup_{B \in \mathcal{B}}\mathcal{Q}'(B)$ follows in similar fashion.

(\ref{condredpos}) $\Rightarrow$ (\ref{condgenrank}) and (\ref{condredposn}) $\Rightarrow$ (\ref{condgenrank}) are immediate from Lemma~\ref{lemrank}. 

(\ref{condgenrank}) $\Rightarrow$ [(\ref{condredpos}) or (\ref{condredposn})] if $\mathcal{B}$ is path connected. Observe that if neither of $\mathrm{det}_A(AB)<0$ for all $B \in \mathcal{B}$, nor $\mathrm{det}_A(A B) > 0$ for all $B \in \mathcal{B}$, holds then, since $\mathcal{B}$ is path connected, there must exist $B' \in \mathcal{B}$ such that $\mathrm{det}_A(A B') = 0$. But then $\mathrm{rank}(AB'A) <r$ by Lemma~\ref{lemrank}.

(\ref{condgenrank1}) $\Rightarrow$ (\ref{condgenrank}) is trivial. For (\ref{condgenrank}) $\Rightarrow$ (\ref{condgenrank1}), suppose Condition~\ref{condgenrank} holds. Clearly then $\mathrm{rank}(A B A) = \mathrm{rank}(A B) = \mathrm{rank}(BA) = \mathrm{rank}(A)$ for all $B \in \mathcal{B}$, and so the result is true for all products of length $2$. Moreover, these cases imply that, for all $B \in \mathcal{B}$, $\mathrm{im}\,A \cap \mathrm{ker}\,B = \{0\}$ and $\mathrm{im}(BA) \cap \mathrm{ker}\,A = \{0\}$. Suppose the result holds for all products of length $n$ for some $n \geq 2$. Premultiplying a product $AB_1\cdots$ of length $n$ by some $B \in \mathcal{B}$ cannot decrease the rank of the product as $\mathrm{im}\,A \cap \mathrm{ker}\,B = \{0\}$. Similarly premultiplying a product $B_1A\cdots$ of length $n$ by $A$ cannot decrease the rank of the product as $\mathrm{im}\,(BA) \cap \mathrm{ker}\,A = \{0\}$ for all $B \in \mathcal{B}$. Thus the result holds for all products of length $n + 1$. 
\hfill
\end{proof}
\begin{remark}
\label{remAnonsing}
A consequence of Lemma~\ref{lemmain0} is that given $0 \neq A \in \mathbb{R}^{n \times m}$ with rank $r$, and a matrix-pattern $\mathcal{B} \subseteq \mathbb{R}^{m \times n}$, the condition ``$A \rst \mathcal{B}^t > 0$ or $A \rst \mathcal{B}^t < 0$'' is equivalent to ``$\mathcal{B}$ is $A$-nonsingular''. 
\end{remark}

The next results provide basic conditions guaranteeing that $r$-compatibility of $A\in\mathbb{R}^{n \times m}$ and $\mathcal{B} \subseteq \mathbb{R}^{n \times m}$ implies $r$-strong compatibility of $A$ and $\mathcal{B}$. They will prove useful in understanding the relationship between injectivity of a CRN and its so-called ``fully open extension''.
\begin{lemma1}
\label{compattoconcord}
Let $A \in \mathbb{R}^{n\times m}$ have rank $r$ and let $\mathcal{B} \subseteq \mathbb{R}^{n \times m}$ be a matrix-pattern. Then the following are equivalent:
\begin{enumerate}[align=left,leftmargin=*]
\item $A \rst \mathcal{B} \geq 0$ and $A \rst B_1 > 0$ for some $B_1 \in \mathcal{B}$.
\item $\mathrm{det}_A(AB^t) \geq 0$ for all $B \in \mathcal{B}$ and $\mathrm{det}_A(AB_1) > 0$ for some $B_1 \in \mathcal{B}$.
\item $A \rst \mathcal{B} > 0$.
\item $\mathrm{det}_A(AB^t) > 0$ for all $B \in \mathcal{B}$.
\end{enumerate}
\end{lemma1}
\begin{proof}
Note first that being a matrix-pattern, $\mathcal{B}$ is path connected and a union of semiclasses. (3) $\Leftrightarrow$ (4) is just the statement (1) $\Leftrightarrow$ (3) in Lemma~\ref{lemmain0}. The proof of (1) $\Leftrightarrow$ (2) follows easily in the same fashion. (3) $\Rightarrow$ (1) is trivial. To prove (1) $\Rightarrow$ (3), suppose (1) holds and observe that this implies the existence of $\alpha, \beta$ with $|\alpha|=|\beta| = r$ such that $A[\alpha|\beta]B_1[\alpha|\beta] > 0$. If (3) fails, then there exists $B_2 \in \mathcal{B}$ such that $A[\alpha|\beta]B_2[\alpha|\beta] = 0$. As $\mathcal{B}$ is a matrix-pattern, there exists $B_3 \in \mathcal{B}$ such that $A[\alpha|\beta]B_3[\alpha|\beta] < 0$ (Lemma~\ref{lemNSnotSNS}), namely $A \rst B_3 \not \geq 0$, contradicting the assumption that (1) holds.
\hfill \end{proof}

As an immediate corollary of Lemma~\ref{compattoconcord} we have:
\begin{cor}
\label{corcompattoconcord}
Let $A \in \mathbb{R}^{n\times m}$ have rank $r$ and let $\mathcal{B} \subseteq \mathbb{R}^{n \times m}$ be a matrix-pattern. If $A \Bumpeq \mathcal{B}$ and $A \rst B_1 > 0$ for some $B_1 \in \mathcal{B}$, then $A \rst \mathcal{B} > 0$ (equivalently $\mathrm{det}_A(AB^t) > 0$ for all $B \in \mathcal{B}$).
\end{cor}

\begin{lemma1}
\label{lemMcompat}
Let $0 \neq A \in \mathbb{R}^{n \times m}, B \in \mathbb{R}^{m\times n}$, and define $r = \mathrm{rank}\,A$. The following are equivalent:
\begin{enumerate}[align=left,leftmargin=*]
\item[(i)] $A \rst B^t > 0$ or $A \rst B^t < 0$,
\item[(ii)] $\mathrm{rank}\,(A D_1 B D_2 A) = r$ for all $D_1 \in \mathcal{D}_m$, $D_2\in \mathcal{D}_n$ ($\mathcal{Q}'(B)$ is $A$-nonsingular).
\end{enumerate}
\end{lemma1}
\begin{proof}
(i) $\Rightarrow$ (ii). By Lemma~\ref{lemcompat_diag} $A \rst B^t > 0 \Leftrightarrow A \rst \mathcal{Q}'(B^t) > 0$, and $A \rst B^t < 0 \Leftrightarrow A \rst \mathcal{Q}'(B^t) < 0$. Thus, by implications (\ref{condrcompat}) $\Rightarrow$ (\ref{condgenrank}) and (\ref{condrcompatn}) $\Rightarrow$ (\ref{condgenrank}) of Lemma~\ref{lemmain0}, $\mathrm{rank}\,(A D_1 B D_2 A) = \mathrm{rank}\,A$.

(ii) $\Rightarrow$ (i). Observe that $\mathcal{Q}'(B)$ is path connected and is trivially a union of semiclasses. By implications (\ref{condgenrank}) $\Rightarrow$ [(\ref{condredpos}) or (\ref{condredposn})] $\Rightarrow$ [(\ref{condrcompat}) or (\ref{condrcompatn})] of Lemma~\ref{lemmain0}, $\mathcal{Q}'(B)$ is $A$-nonsingular implies that $A \rst \mathcal{Q}'(B^t) > 0$ or $A \rst \mathcal{Q}'(B^t) < 0$, and so certainly $A \rst B^t > 0$ or $A \rst B^t < 0$. 
\hfill \end{proof}

The following result illustrates one of the primary uses of the DSR graph: graph theoretic tests for compatibility of matrices can be significantly more efficient than direct approaches.
\begin{lemma1}
\label{lemmaina}
Let $A \in \mathbb{R}^{n \times m}$ and $B \in \mathbb{R}^{m \times n}$. If the DSR graph $G_{A, B}$ satisfies Condition~($*$) in Appendix~\ref{appstar} then $A \Bumpeq B^t$.
\end{lemma1}
\begin{proof}
This is shown in \cite{banajicraciun2}. 
\hfill
\end{proof}

\begin{remark}[Condition~($*$): history and previous results]
\label{remnondegen}
Condition~($*$) is an easily computable condition, described in Appendix~\ref{appstar} and implemented algorithmically in {\tt CoNtRol} \cite{control}. It originated in Craciun and Feinberg \cite{craciun1} where the condition was applied to SR graphs of a CRN, and used to make injectivity claims about CRNs with mass action kinetics. It was then used to make injectivity claims about CRNs with general kinetics in Banaji and Craciun \cite{banajicraciun}, and was further extended to DSR graphs and used to make claims about a very general class of dynamical systems termed ``interaction networks'' (which include, but go beyond, CRNs) in Banaji and Craciun \cite{banajicraciun2}. By Lemma~\ref{lemmaina}, if $\mathcal{B}$ is a matrix-pattern and $G_{A, B}$ satisfies Condition~($*$) for all $B \in \mathcal{B}$, then this implies in particular that $A \Bumpeq \mathcal{B}^t$. Corollary~\ref{corcompattoconcord} states that if we can additionally confirm that $A \rst B^t \neq 0$ for {\bf some} $B \in \mathcal{B}$ (where $r = \mathrm{rank}\,A$), then $A \rst \mathcal{B}^t > 0$. In some situations this is automatic (see e.g., Lemma~\ref{lemmain2} below).
\end{remark}

\subsection{Compatibility and the reduced determinant in the case $\mathcal{B} = \mathcal{Q}(A^t)$}

While, in the previous section, $\mathcal{B}$ is an arbitrary set of matrices, at most assumed to be a matrix-pattern, the following results focus on the important special case $\mathcal{B} = \mathcal{Q}(A^t)$, particularly relevant to the study of certain classes of CRNs, termed simply reversible CRNs below. There exist simple necessary and sufficient conditions for a matrix $A$ to be compatible, or $r$-strongly compatible, with its entire qualitative class $\mathcal{Q}(A)$.
\begin{def1}[SSD, $r$-SSD]
\label{defSSD}
Given $A \in \mathbb{R}^{n \times m}$ and $r \in \{1, \ldots, \min\{n, m\}\}$, $A$ is termed  $r$-{\em SSD} if every $r \times r$ submatrix of $A$ is either singular or sign nonsingular. It is {\em SSD} if all square submatrices of $A$ are either sign nonsingular or singular, i.e., $A$ is $r$-SSD for each allowed $r$. (The acronym {\em SSD} was originally an abbreviation of {\em strongly sign determined} and the concept was introduced in Banaji et al. \cite{banajiSIAM}.)
\end{def1}

\begin{lemma1}
\label{lemmain}
The following conditions on $A \in \mathbb{R}^{n \times m}$ with rank $r>0$ are equivalent:
\begin{enumerate}[align=left,leftmargin=*]
\item $A$ is $r$-SSD. 
\label{condrSSD}
\item $A \rst \mathcal{Q}(A) \geq 0$.
\label{condrcomp}
\item $A\rst \mathcal{Q}(A) > 0$.
\label{condrstrongcomp}
\item $\mathrm{det}_A(A B) > 0$ for each $B \in \mathcal{Q}(A^t)$. 
\label{condredJnonz}
\item $\mathrm{rank}(A B A) = r$ for each $B \in \mathcal{Q}(A^t)$ ($\mathcal{Q}(A^t)$ is $A$-nonsingular).
\label{condrank}
\item Given any $k \geq 2$, every product of length $k$ of the form $A B_1 A B_2 \cdots$ or $B_1 A B_2 A \cdots$ where $B_i \in  \mathcal{Q}(A^t)$, has rank $r$.
\label{condrank1}
\end{enumerate}
\end{lemma1}
\begin{proof}
(\ref{condrSSD}) $\Leftrightarrow$ (\ref{condrcomp}). The implication (\ref{condrSSD}) $\Rightarrow$ (\ref{condrcomp}) is immediate by definition. In the other direction, if $A$ is not $r$-SSD, then there exist $\alpha \subseteq \mathbf{n}, \beta \subseteq \mathbf{m}$ such that $|\alpha| = |\beta| = r$, $A[\alpha|\beta] \neq 0$ but $A(\alpha|\beta)$ is not sign nonsingular.  By Lemma~\ref{lemNSnotSNS}, there exists $\tilde B \in \mathcal{Q}(A)$ such that $A[\alpha|\beta]\tilde{B}[\alpha|\beta] <0$, i.e., $A \rst \tilde{B} \not \geq 0$. 

(\ref{condrcomp}) $\Leftrightarrow$ (\ref{condrstrongcomp}). The implication (\ref{condrstrongcomp}) $\Rightarrow$ (\ref{condrcomp}) is immediate. In the other direction, since all $r \times r$ submatrices of $A$ are either singular or sign nonsingular, but $A$ has rank $r$, there must be a sign nonsingular $r \times r$ submatrix of $A$, say $A(\alpha|\beta)$. So, by definition, $A[\alpha|\beta]B[\alpha|\beta] >0$ for all $B \in \mathcal{Q}(A)$. 

(\ref{condrstrongcomp}) $\Leftrightarrow$ (\ref{condredJnonz}) follows from Lemma~\ref{lemmain0} with $\mathcal{B} = \mathcal{Q}(A^t)$. 

(\ref{condredJnonz}) $\Leftrightarrow$ (\ref{condrank}). Condition~\ref{condredJnonz} implies Condition~\ref{condrank} by Lemma~\ref{lemrank}. In the other direction, suppose $\mathrm{rank}(ABA) = r$ for each $B \in \mathcal{Q}(A^t)$. Choosing $B' = A^t$, it is immediate that $\mathrm{det}_A(AB') = \sum_{|\alpha| = r}(A B')[\alpha] > 0$. As $\mathcal{Q}(A^t)$ is path connected, it now follows from implication (\ref{condgenrank}) $\Rightarrow$ [(\ref{condredpos}) or (\ref{condredposn})] of Lemma~\ref{lemmain0} that $\mathrm{det}_A (AB) > 0$ for each $B \in \mathcal{Q}(A^t)$.

(\ref{condrank}) $\Leftrightarrow$ (\ref{condrank1}) follows from Lemma~\ref{lemmain0}  with $\mathcal{B} = \mathcal{Q}(A^t)$. 
\hfill
\end{proof}

\begin{remark}
\label{remnegSSD}
Observe that given a real matrix $A$ with rank $r$, $A\rst A > 0$, and consequently $A\rst \mathcal{Q}(A) \leq 0$ is impossible.
\end{remark}
\begin{remark}
\label{remAnonsing2}
A consequence of Lemma~\ref{lemmain} is that given $0 \neq A \in \mathbb{R}^{n \times m}$ with rank $r$ the condition ``$A$ is r-SSD'' is equivalent to ``$\mathcal{Q}(A^t)$ is $A$-nonsingular''. 
\end{remark}

\begin{remark}
The condition that $\mathrm{rank}(A B A) = \mathrm{rank}\,A$ for each $B \in \mathcal{Q}(A^t)$ (namely $\mathcal{Q}(A^t)$ is $A$-nonsingular) is a stronger claim than merely that $\mathrm{rank}(A\,B) = \mathrm{rank}\,A$ for all $B \in \mathcal{Q}(A^t)$: consider the matrices
\[
A = \left(\begin{array}{cc}2&1\\1&1\\1&0\end{array}\right), \quad B = \left(\begin{array}{ccc}a&b&e\\c&d&0\end{array}\right)\,\quad \mbox{so that} \quad A B = \left(\begin{array}{ccc}2a+c&2b+d&2e\\a+c&b+d&e\\a&b&e\end{array}\right)\,,
\]
where $a,b,c,d,e > 0$. Then $\mathrm{rank}(A B) = \mathrm{rank}\,A = 2$ for all such $B$ ($A B$ has a nonsingular $2 \times 2$ submatrix). But $A$ is not SSD and $\mathrm{rank}(A BA)$ can equal $1$. In particular, the sum of the $2 \times 2$ principal minors of $A B$ is $ad + ce + de - bc$ which may be zero. 
\end{remark}

\begin{lemma1}
\label{lemmain2}
Define the following conditions on a matrix $A \in \mathbb{R}^{n \times m}$ with rank $r$:
\begin{enumerate}[align=left,leftmargin=*]
\item The SR graph $G_A$ satisfies Condition~($*$) in Appendix~\ref{appstar}.
\label{condstar}
\item $A$ is SSD.
\label{condSSD}
\item $A \Bumpeq \mathcal{Q}(A)$.
\label{condallnonz}
\item $A B$ is a $P_0$-matrix for each $B \in \mathcal{Q}_0(A^t)$.
\label{condP0}
\item $\mathrm{rank}(A B A) = r$ for each $B \in \mathcal{Q}(A^t)$ (i.e., $\mathcal{Q}(A^t)$ is $A$-nonsingular).
\label{condranka}
\end{enumerate}
The following implications hold: (\ref{condstar}) $\Rightarrow$ (\ref{condSSD}) $\Leftrightarrow$ (\ref{condallnonz}) $\Leftrightarrow$  (\ref{condP0}) $\Rightarrow$ (\ref{condranka}).
\end{lemma1}
\begin{proof}
(\ref{condstar}) $\Rightarrow$ (\ref{condSSD}) is proved in \cite{banajicraciun}. (\ref{condSSD}) $\Leftrightarrow$ (\ref{condallnonz}) follows by applying the proof of (\ref{condrSSD}) $\Leftrightarrow$ (\ref{condrcomp}) in Lemma~\ref{lemmain} to each dimension; (\ref{condallnonz}) $\Rightarrow$ (\ref{condranka}) is immediate from Lemma~\ref{lemmain} (the special case $r=0$ is trivial). That (\ref{condallnonz}) implies that $A B$ is a $P_0$-matrix for all $B \in \mathcal{Q}(A^t)$ follows from Lemma~\ref{lemP0}; (\ref{condP0}) then follows by closure of the $P_0$-matrices. On the other hand if (\ref{condallnonz}) is violated and there exist $\alpha \subseteq \mathbf{n}, \beta \subseteq \mathbf{m}$ such that $0 \neq |\alpha| = |\beta|$ and some $B \in \mathcal{Q}(A^t)$ such that $A[\alpha|\beta]B[\beta|\alpha] < 0$, then $B_{\beta,\alpha} \in \mathcal{Q}_0(A^t)$ (Notation~\ref{notredmat}), but $(A B_{\beta,\alpha})[\alpha] = A[\alpha|\beta]B[\beta|\alpha] < 0$. So (\ref{condP0}) is violated. 
\hfill
\end{proof}

\section{Injectivity results}

We recall that a function $f$ with domain $X$ is injective on $X$ if $a, b \in X$, $a \neq b$ implies $f(a) \neq f(b)$. In the study of CRNs, we will be concerned with functions of the form $\Gamma v(x)$ on $\mathbb{R}^n_{\geq 0}$ or $\mathbb{R}^n_{\gg 0}$, where $\Gamma \in \mathbb{R}^{n \times m}$ is the ``stoichiometric matrix'' of the system (to be defined below), and the function $v$ is a vector of reaction rates. We note that the choice to discuss functions of the form $\Gamma v$, namely with a constant first factor, is not really limiting: any vector field with linear integrals can be written in this way (the choice of $\Gamma$ is not in general unique), and in fact any function can be cast in this form by choosing $\Gamma$ to be the identity. That even this latter approach can produce nontrivial results on the injectivity of functions is demonstrated in \cite{banajiJMAA}. We proceed to examine such functions, noting that the discussion at this stage is quite general. 

\subsection{The general case}

Note first that if we state that a function $f$ is $C^1$ (continuously differentiable) on some subset $U \subseteq \mathbb{R}^n$, not necessarily open, we mean that $f$ extends to a $C^1$ function on some open neighbourhood of $U$. Let $\Gamma \in \mathbb{R}^{n \times m}$, $U \subseteq \mathbb{R}^n$, and $v\colon U \to \mathbb{R}^m$. We will examine conditions on $\Gamma$ and $v$ which allow us to make claims termed IC1, \apost{IC1}, \apostt{IC1}, IC2, \apost{IC2}, \apostt{IC2}, and IC1$^{-}$, about the function $\Gamma v\colon U \to \mathbb{R}^n$ (a further claim, termed IC1a, will be discussed later). These claims are all about the possibility of $\Gamma v$ or a related function taking the same value at distinct points and can, roughly speaking, be termed ``injectivity claims''. Claims IC1, IC2, and IC1$^{-}$ will be relevant when $v$ is defined and $C^1$ on $\mathbb{R}^n_{\gg 0}$; \apost{IC1} and \apost{IC2} when $v$ is defined and continuous on $\mathbb{R}^n_{\geq 0}$, and $C^1$ on $\mathbb{R}^n_{\gg 0}$; and \apostt{IC1} and \apostt{IC2} when $v$ is defined and $C^1$ on $\mathbb{R}^n_{\geq 0}$.
\begin{enumerate}[align=left,leftmargin=*]
\item[(\apostt{IC1})] If $x,y \in \mathbb{R}^n_{\geq 0}$, $x\simneq^\Gamma\! y$ and $\Gamma v(x) = \Gamma v(y)$, then $x$ and $y$ share a facet (Definition~\ref{deffacet}) of $\mathbb{R}^n_{\geq 0}$. 
\item[(\apost{IC1})] $x\in \mathbb{R}^n_{\gg 0}$, $y \in \mathbb{R}^n_{\geq 0}$, $x\simneq^\Gamma\! y$ imply $\Gamma v(y) \neq \Gamma v(x)$.
\item[(IC1)] $x,y \in \mathbb{R}^n_{\gg 0}$, $x\simneq^\Gamma\! y$ imply $\Gamma v(y) \neq \Gamma v(x)$. 
\item[(IC1$^{-}$)] $x,y \in \mathbb{R}^n_{\gg 0}$, $x\simneq^\Gamma\! y$, and $\Gamma v(y) = \Gamma v(x)$ imply that either $\mathrm{det}_\Gamma(\Gamma\,Dv(x)) = 0$ or $\mathrm{det}_\Gamma(\Gamma\,Dv(y)) = 0$.
\item[(\apostt{IC2})] If $x,y \in \mathbb{R}^n_{\geq 0}$, $x\neq y$, and $q\colon \mathbb{R}^n_{\geq 0} \to \mathbb{R}^n$ is $C^1$ with $Dq \in \mathcal{D}_n$ on $\mathbb{R}^n_{\geq 0}$, then $\Gamma v(x) - q(x) \neq \Gamma v(y) - q(y)$. 
\item[(\apost{IC2})] If $x\in \mathbb{R}^n_{\gg 0}$, $y \in \mathbb{R}^n_{\geq 0}$, $x \neq y$, and $q\colon \mathbb{R}^n_{\geq 0} \to \mathbb{R}^n$ is continuous, and $C^1$ on $\mathbb{R}^n_{\gg 0}$ with $Dq \in \mathcal{D}_n$ on $\mathbb{R}^n_{\gg 0}$, then $\Gamma v(x) - q(x) \neq \Gamma v(y) - q(y)$.
\item[(IC2)] If $x,y \in \mathbb{R}^n_{\gg 0}$, $x \neq y$, and $q\colon \mathbb{R}^n_{\gg 0} \to \mathbb{R}^n$ is $C^1$ with $Dq \in \mathcal{D}_n$ on $\mathbb{R}^n_{\gg 0}$, then $\Gamma v(x) - q(x) \neq \Gamma v(y) - q(y)$.
\end{enumerate}

\begin{remark}[Motivation for the different injectivity claims]
In the literature on chemical reaction systems, the most commonly used notion when discussing injectivity of CRNs is IC1. Observe that if $\Gamma v(x)$ fails condition IC1, this does not imply that {\em every} coset of $\mathrm{im}\,\Gamma$ intersecting $\mathbb{R}^n_{\gg 0}$ contains $x,y \in \mathbb{R}^n_{\gg 0}$, $x\simneq^\Gamma\! y$ such that $\Gamma v(x) = \Gamma v(y)$, only that this occurs on some coset of $\mathrm{im}\,\Gamma$. IC1, IC2, and \apostt{IC2} are true injectivity claims. \apost{IC1} and \apostt{IC1} are partial extensions of IC1 to the boundary. \apost{IC2} is a partial extension of IC2 to the boundary, while \apostt{IC2} is a complete extension of IC2 to the boundary. The variety of different closely related conditions allows for a range of assumptions on reaction rates, on inflows and outflows (to be defined later), and potentially allows claims about dynamical systems going beyond chemistry. In particular, we leave open the possibilities that $v$ fails to be defined on $\partial\mathbb{R}^n_{\geq 0}$, or is defined but fails to be differentiable on $\partial \mathbb{R}^n_{\geq 0}$, particularly relevant to power-law functions discussed later. After we have developed the appropriate notions, in Section~\ref{secCRNinj} we describe the implications of the different conditions for the possibility of multiple equilibria in a CRN.
\end{remark}

We next list some relationships between the claims. In particular, we note that \apost{IC1} and \apost{IC2} are entirely natural extensions of IC1 and IC2 respectively provided the function $\Gamma v$ is defined and continuous on $\mathbb{R}^n_{\geq 0}$.

\begin{lemma1}[Automatic relationships between the injectivity claims]
\label{lemICrels}
The following implications between claims about a function $f = \Gamma v$ are automatic: \apostt{IC1} $\Rightarrow$ \apost{IC1} $\Rightarrow$ IC1, and \apostt{IC2} $\Rightarrow$ \apost{IC2} $\Rightarrow$ IC2. Provided $f$ is defined and $C^1$ on $\mathbb{R}^n_{\gg 0}$, IC1 $\Rightarrow$ IC1$^{-}$ and IC2 $\Rightarrow$ IC1$^{-}$. Provided $f$ is defined and continuous on $\mathbb{R}^n_{\geq 0}$, IC1 $\Rightarrow$ \apost{IC1} and IC2 $\Rightarrow$ \apost{IC2}. Thus we have the following implications:
\begin{center}
\begin{tikzpicture}[scale=0.7]
\node at (0,1) {\apostt{IC1}};
\node at (1,0.95) {$\Rightarrow$};
\node at (2,1) {\apost{IC1}};
\node at (3,0.95) {$\Leftrightarrow$};
\node at (4,1) {{IC1}};
\node at (0,0) {\apostt{IC2}};
\node at (1,-0.05) {$\Rightarrow$};
\node at (2,0) {\apost{IC2}};
\node at (3,-0.05) {$\Leftrightarrow$};
\node at (4,0) {{IC2}};
\node at (5,0.75)[rotate=-20] {$\Rightarrow$};
\node at (5,0.15)[rotate=20] {$\Rightarrow$};
\node at (6,0.5) {IC1$^{-}$};
\end{tikzpicture}
\end{center}
provided the assumptions on existence and differentiability are fulfilled.
\end{lemma1}
\begin{proof}
\apostt{IC1} $\Rightarrow$ \apost{IC1} $\Rightarrow$ IC1 $\Rightarrow$ IC1$^{-}$ are immediate. \apostt{IC2} $\Rightarrow$ \apost{IC2} $\Rightarrow$ IC2 are immediate; IC2 $\Rightarrow$ IC1$^{-}$ can be proved using arguments involving the invariance of Brouwer degree. See Lemma~B1 in \cite{banajicraciun2} for the details, and \cite{craciun2, soule, parthasarathy} for closely related results. The final two claims, namely that IC1 $\Rightarrow$ \apost{IC1} and IC2 $\Rightarrow$ \apost{IC2} provided $f$ is continuous on $\mathbb{R}^n_{\geq 0}$, follow from Lemma~\ref{lemICclosure} below.
\hfill
\end{proof}

\begin{lemma1}
\label{lemICclosure}
Let $V$ be a vector subspace of $\mathbb{R}^n$, and fix $c \in \mathbb{R}^n$. Let $U$ be a relatively open subset of the affine set $c + V$, with closure $\overline{U}$. If $f\colon \overline U \to V$ is continuous on $\overline U$ and injective on $U$, then $x \in U, y \in \overline U$, $x \neq y$ implies $f(x) \neq f(y)$.
\end{lemma1}
\begin{proof} Suppose there exist $x \in U, y \in \overline U$, $x \neq y$ such that $f(x) = f(y)$. As $f$ is continuous and injective on $U$, by the invariance of domain theorem (see e.g. Propostion~7.4 in \cite{dold}) $f$ maps some open neighbourhood $N$ of $x$ in $U$ homeomorphically onto some open neighbourhood $N'$ of $f(x)$ in $V$. Choose $(y_i) \subseteq U\backslash N$, $y_i \to y$; then continuity of $f$ implies $f(y_i) \to f(y) = f(x)$, and so for sufficiently large $i$, $f(y_i) \in N'$, contradicting injectivity of $f$ on $U$. 
\hfill \end{proof}

\begin{remark}[The trivial case $\Gamma = 0$]
If $\Gamma = 0$, then the claims are all satisfied: \apostt{IC1}, \apost{IC1}, IC1, and IC1$^{-}$ are trivial (since $x\simneq^\Gamma\! y$ is impossible), while \apostt{IC2}, IC2 are easy via the fundamental theorem of calculus. \apost{IC2} follows from IC2 by Lemma~\ref{lemICclosure}. In the results below we assume that $\Gamma \neq 0$. 
\end{remark}

\begin{remark}[IC2 and fully open CRNs]
\label{remCRNoutflows}
IC2 (resp., \apostt{IC2}) can be interpreted as stating that all functions of the form $c + \Gamma v(\cdot) - q(\cdot)$ are injective on $\mathbb{R}^n_{\gg 0}$ (resp., $\mathbb{R}^n_{\geq 0}$), where $c \in \mathbb{R}^n$ is a constant vector and $q$ satisfies the assumptions of the claim. IC2, \apost{IC2} and \apostt{IC2} are of interest in the study of ``fully open'' CRNs (to be defined below), namely for situations where outflows of all species are to be expected (see Craciun and Feinberg \cite{craciun,craciun1} for example). 
\end{remark}

\begin{def1}[Nondegenerate equilibria]
Given $0 \neq \Gamma \in \mathbb{R}^{n \times m}$, $U \subseteq \mathbb{R}^n$, and $v\colon U \to \mathbb{R}^m$ as above, $p \in U$ is termed a {\em nondegenerate equilibrium} of $\Gamma v$ if $\Gamma v(p) = 0$ and $\mathrm{det}_\Gamma(\Gamma\,Dv(p)) \neq 0$. If $\mathrm{det}_\Gamma(\Gamma\,Dv(p)) = 0$, then $p$ is {\em degenerate} \cite[Definition 4]{craciun2}.
\end{def1}

\begin{remark}[IC1$^{-}$]
A consequence of IC1$^{-}$ is that ``$\Gamma v$ forbids multiple positive nondegenerate equilibria'', namely if $x,y \in \mathbb{R}^n_{\gg 0}$, $x\simneq^\Gamma\! y$, and $\Gamma v(y) = \Gamma v(x) = 0$, then at least one of $x,y$ must be degenerate. We will rarely explicitly mention IC1$^{-}$, but the reader should bear in mind that IC2 $\Rightarrow$ IC1$^{-}$ for all functions of the kind treated in this paper.
\end{remark}

\begin{notation}[Closure of a set of matrices]
Given a set of real matrices $\mathcal{V}$, $\overline{\mathcal{V}}$ will refer to the closure of $\mathcal{V}$. 
\end{notation}

\begin{def1}[Stable/strongly stable under path integration]
A set of matrices $\mathcal{V}$ is {\em stable under path integration} if given any continuous $\gamma:[0,1] \to \mathcal{V}$, the integral $\int_0^1\gamma(s)\,\mathrm{d}s \in \mathcal{V}$. $\mathcal{V}$ is {\em strongly stable under path integration} if given any continuous $\gamma:[0,1] \to \overline{\mathcal{V}}$ with $\gamma(c) \in \mathcal{V}$ for some $c \in [0,1]$, then the integral $\int_0^1\gamma(s)\,\mathrm{d}s \in \mathcal{V}$. 
\end{def1}

\begin{remark}[Matrix-patterns are strongly stable under path integration]
\label{remineq}
Any set of (real) matrices $\mathcal{V}$ defined by a set of linear equalities and inequalities on its entries is stable under path integration. For example let $\gamma:[0,1] \to \mathcal{V}$ be continuous and let $\mathrm{\prec}$ be any of $=$, $\leq$, $\geq$, $<$ or $>$. If $A$ is some matrix such that $\mathrm{trace}(A^t\gamma(s)) \prec 0$ for each $s \in [0,1]$ then clearly $\mathrm{trace}(A^t\int_0^1\gamma(s)\,\mathrm{d}s) = \int_0^1\mathrm{trace}(A^t\gamma(s))\,\mathrm{d}s) \prec 0$. By similar reasoning, any set of matrices defined by a set of linear equalities and {\em strict} inequalities, such as a matrix-pattern (Definition~\ref{defmatrixpattern}) for example, is {\em strongly} stable under path integration. 
\end{remark}

For completeness observe that:
\begin{lemma1}
If $\emptyset \neq S\subseteq \mathbb{R}^{m \times n}$ is convex then it is stable under path integration.
\end{lemma1}
\begin{proof}
The result is immediately true for any closed nonempty convex set in $\mathbb{R}^{m \times n}$: such sets are the intersection of their supporting half-spaces (Thm~2.7(ii) in \cite{giaquinta} for example), and stability under path integration then follows from Remark~\ref{remineq}. Now consider arbitrary convex $S \subseteq \mathbb{R}^{m \times n}$ and some continuous $\gamma:[0,1] \to S$. As $\gamma([0,1])$ is compact as the continuous image of a compact set, so is its convex hull $C$, which is again the continuous image of a compact set. Thus we can regard $\gamma$ as a path in the compact convex set $C$, and so $\int_0^1 \gamma(s)\,\mathrm{d} s \in C$. But $C \subseteq S$ and the result follows. 
\hfill
\end{proof}

The proof of the following theorem follows the argument of Gouz\'e \cite{gouze98} where a version of the first Thomas conjecture is proved. The result of \cite{gouze98} can in turn be deduced as a corollary of Theorem~\ref{thminj}.

\begin{thm}
\label{thminj}
Let $0 \neq \Gamma \in \mathbb{R}^{n \times m}$, and let $\mathcal{V} \subseteq \mathbb{R}^{m \times n}$ be such that $\Gamma \rst \mathcal{V}^t > 0$ or $\Gamma \rst \mathcal{V}^t < 0$. Further, let $\mathcal{V}$ be strongly stable under path integration. Let $U\supseteq \mathbb{R}^n_{\gg 0}$ and $v\colon U \to \mathbb{R}^m$. Then
\begin{enumerate}[align=left,leftmargin=*]
\item Given any $x,y \in U$, $x\simneq^\Gamma\! y$, suppose $v$ is defined and $C^1$ on the line segment $[x,y]$ joining $x$ and $y$, with $Dv(p) \in \overline{\mathcal{V}}$ on $[x,y]$, and $Dv(p) \in \mathcal{V}$ for some $p \in [x,y]$. Then $\Gamma v(x) \neq \Gamma v(y)$. 
\item Suppose $v$ is $C^1$ on $\mathbb{R}^n_{\gg 0}$, and $Dv(x) \in \mathcal{V}$ for $x \in \mathbb{R}^n_{\gg 0}$. Then $\Gamma v$ satisfies claim IC1. If $v$ is defined and continuous on $\mathbb{R}^n_{\geq 0}$ and $C^1$ on $\mathbb{R}^n_{\gg 0}$, then $\Gamma v$ satisfies claim \apost{IC1}. If $v$ is defined and $C^1$ on $\mathbb{R}^n_{\geq 0}$, then $\Gamma v$ satisfies claim \apostt{IC1}.
\end{enumerate}
\end{thm}
\begin{proof}
Write $y-x = \Gamma z$. Then by the fundamental theorem of calculus, 
\[
\Gamma v(y) - \Gamma v(x)  =  \Gamma\left[\int_0^1Dv(t y + (1-t) x)\, \mathrm{d}t\right]\Gamma z = \Gamma\tilde V \Gamma z\,,
\]
where the final equality defines $\tilde V$. By the assumptions on $[x,y]$ and $\mathcal{V}$, $\tilde V\in \mathcal{V}$, and hence $\Gamma \rst \tilde{V}^t > 0$ or $\Gamma \rst \tilde{V}^t < 0$. By Lemma~\ref{lemmain0}, $\mathrm{rank}(\Gamma \tilde V \Gamma) = \mathrm{rank}\,\Gamma$. Thus since $\Gamma z \neq 0$, $\Gamma \tilde V \Gamma z \neq 0$, and the first claim follows.

Suppose that $Dv(x) \in \mathcal{V}$ for $x \in \mathbb{R}^n_{\gg 0}$. That $\Gamma v$ satisfies IC1 follows immediately from the first claim. Provided $v$ is additionally defined and continuous on $\mathbb{R}^n_{\geq 0}$, $\Gamma v$ satisfies \apost{IC1} by Lemma~\ref{lemICclosure}. If $v$ is defined and $C^1$ on $\mathbb{R}^n_{\geq 0}$, then $Dv(x) \in \overline{\mathcal{V}}$ for $x \in\mathbb{R}^n_{\geq 0}$ (since $v$ is $C^1$); that $\Gamma v$ satisfies claim \apostt{IC1} follows from the first claim by noting that any line segment in $\mathbb{R}^n_{\geq 0}$ either lies entirely in some facet of $\mathbb{R}^n_{\geq 0}$ or intersects $\mathbb{R}^n_{\gg 0}$ thus containing $p$ such that $Dv(p) \in \mathcal{V}$.
\hfill
\end{proof}

In the important special case where $\mathcal{V} = \mathcal{Q}(\Gamma^t)$ we have:
\begin{lemma1}
\label{leminj1}
Let $0 \neq \Gamma \in \mathbb{R}^{n \times m}$ have rank $r$ and be $r$-SSD. Let $U\supseteq \mathbb{R}^n_{\gg 0}$, and $v\colon U \to \mathbb{R}^m$ be $C^1$ on $\mathbb{R}^n_{\gg 0}$, with $Dv \in \mathcal{Q}(\Gamma^t)$ on $\mathbb{R}^n_{\gg 0}$. Then $\Gamma v$ satisfies claim IC1. If $v$ is defined and continuous on $\mathbb{R}^n_{\geq 0}$, and $C^1$ on $\mathbb{R}^n_{\gg 0}$, then $\Gamma v$ satisfies claim \apost{IC1}. If $v$ is defined and $C^1$ on $\mathbb{R}^n_{\geq 0}$, then $\Gamma v$ satisfies claim \apostt{IC1}.
\end{lemma1}
\begin{proof}
As $\Gamma$ is $r$-SSD, $\Gamma \rst \mathcal{Q}(\Gamma) > 0$ by Lemma~\ref{lemmain}. The result now follows from the second part of Theorem~\ref{thminj} with $\mathcal{V} = \mathcal{Q}(\Gamma^t)$.
\hfill
\end{proof}

\begin{lemma1}
\label{leminj2}
Let $\Gamma \in \mathbb{R}^{n \times m}$, $U\supseteq \mathbb{R}^n_{\gg 0}$, and let $v\colon U \to \mathbb{R}^m$ be $C^1$ on $\mathbb{R}^n_{\gg 0}$. If $-\Gamma Dv(x)$ is a $P_0$-matrix for each $x \in \mathbb{R}^n_{\gg 0}$, then $\Gamma v$ satisfies claim IC2. Additionally, if $v$ is defined and continuous (resp., $C^1$) on $\mathbb{R}^n_{\geq 0}$, then $\Gamma v$ satisfies claim \apost{IC2} (resp., \apostt{IC2}). 
\end{lemma1}
\begin{proof}
The claims follow from the injectivity of functions on rectangular domains with $P$-matrix Jacobians (Gale and Nikaido \cite{gale}). In brief, the conditions of the lemma guarantee that given $q(\cdot)$ as in IC2, the function $-\Gamma v(\cdot) + q(\cdot)$ has $P$-matrix Jacobian matrix on $\mathbb{R}^n_{\gg 0}$ and is hence injective on $\mathbb{R}^n_{\gg 0}$; consequently $\Gamma v(\cdot) - q(\cdot)$ is also injective on $\mathbb{R}^n_{\gg 0}$. See Banaji and Craciun \cite{banajicraciun2} for more details. This fact is also behind the proof by Soul\'e \cite{soule} of a version of the first Thomas conjecture. Lemma~\ref{lemICclosure} ensures that \apost{IC2} is satisfied provided $v$ is defined and continuous on $\mathbb{R}^n_{\geq 0}$. If $v$ is in fact $C^1$ on $\mathbb{R}^n_{\geq 0}$ then $-\Gamma Dv(x)$ is a $P_0$-matrix for each $x \in \mathbb{R}^n_{\geq 0}$ (by continuity of the derivative and closure of the $P_0$-matrices), and with $q(\cdot)$ as in \apostt{IC2}, the function $-\Gamma v(\cdot) + q(\cdot)$ has $P$-matrix Jacobian matrix on $\mathbb{R}^n_{\geq 0}$, ensuring that $\Gamma v$ satisfies claim \apostt{IC2}. 
\hfill
\end{proof}

\subsection{Power-law functions}
\label{secpowlaw}

In Theorem~\ref{thminj}, $v$ was a general $C^1$ function. We now examine a special case, ``power-law functions'', where $Dv$ belongs to a set which is not in general convex and not in general stable under path integration, while nevertheless we are able to make claims about injectivity using techniques similar to those for general kinetics. 
\begin{def1}[Exponential and logarithmic functions]
Define the exponential and logarithmic functions $\exp\colon \mathbb{R}^n \to \mathbb{R}^n_{\gg 0}$ and $\ln\colon \mathbb{R}^n_{\gg 0} \to \mathbb{R}^n$ componentwise in the natural way, i.e., $(\exp\,x)_i = \exp\,x_i$ and $(\ln\,x)_i = \ln\,x_i$ for each $i$. Clearly $\exp$ and $\ln$ are inverse functions and are diagonal, namely $(\exp\,x)_i, (\ln\,x)_i$ depend on $x_i$ only. 
\end{def1}
\begin{notation}[Generalised monomials $x^M$]
Given $M \in \mathbb{R}^{m \times n}$, $x^M$ is a convenient abbreviation for the vector of generalised monomials $w = (w_1, \ldots, w_m)^t$ with $w_j = \prod_{i=1}^n x_i^{M_{ji}}, \,\, (j=1, \ldots, m)$. Note that if we regard $x^M$ as a function on $\mathbb{R}^n_{\gg 0}$, we can write $x^M = \mathrm{exp}(M\ln x)$. 
\end{notation}
\begin{def1}[Power-law function]
\label{defpowlaw}
Given $\Gamma \in \mathbb{R}^{n \times m}$, $M \in \mathbb{R}^{m \times n}$, and $E \in \mathcal{D}_m$, we refer to any function of the form $\Gamma E x^M$ as a {\em power-law function} and to $M$ as the {\em matrix of exponents}. 
\end{def1}

The following technical lemma is useful:
\begin{lemma1}
\label{lemWhitney}
Consider $\Gamma \in \mathbb{R}^{n \times m}$, $E \in \mathcal{D}_m$, and $0 \leq M \in \mathbb{R}^{m \times n}$ such that the nonzero entries of $M$ are all greater than or equal to $1$. Then the function $f = \Gamma E x^M$, with domain $\mathbb{R}^n_{\geq 0}$, can be extended to a $C^1$ function $\bar f:\mathbb R^n\to \mathbb R^n$.
\end{lemma1}
\begin{proof}
The partial derivatives of $f$ can clearly be extended continuously on $\partial \mathbb R_{\geq 0}^n$. The result now follows from a version of the Whitney extension theorem \cite[Theorem 4]{Whitney.1934aa}. 
\hfill \end{proof}

\begin{remark}[Domain/differentiability of power-law functions]
Observe that the power-law functions of Definition~\ref{defpowlaw} are defined on $\mathbb{R}^n_{\gg 0}$ for arbitrary $M$. However, if $M$ is nonnegative, then clearly $\Gamma E x^M$ extends continuously to all of $\mathbb{R}^n_{\geq 0}$. If the nonzero entries of $M$ are greater than or equal to $1$, then $\Gamma E x^M$ can be extended to a $C^1$ function on $\mathbb{R}^n_{\geq 0}$ (Lemma~\ref{lemWhitney}). If $M$ is a nonnegative integer matrix, then $\Gamma E x^M$ is in fact a polynomial function on $\mathbb{R}^n$.
\end{remark}

\begin{remark}[Jacobian matrix of a power-law function on $\mathbb{R}^n_{\gg 0}$]
\label{remPLJac}
By a quick computation, the Jacobian matrix of the power-law function $\Gamma E w(x)$ where $w(x) = \mathrm{exp}(M\ln x)$ (on $\mathbb{R}^n_{\gg 0}$) is $\Gamma D_{Ew} M D_{1/x}$, where $D_{Ew} \in \mathcal{D}_m$ is defined by $(D_{Ew})_{jj} = E_{jj} w_j$ and $D_{1/x} \in \mathcal{D}_n$ is defined by $(D_{1/x})_{jj} = 1/x_j$ (see also Banaji et al. \cite{banajiSIAM}, and Remark~3.1 in Craciun and Feinberg \cite{Craciun.2010ac} for an equivalent formulation). By Lemma~\ref{lemrank}, the reduced determinant $\mathrm{det}_\Gamma(\Gamma D_{Ew} M D_{1/x})$ is nonzero if and only if $\mathrm{rank}(\Gamma D_{Ew} M D_{1/x} \Gamma) = \mathrm{rank}\,\Gamma$. It is easy to see that for fixed $\Gamma$ and $M$,  the set of all possible Jacobian matrices of power-law functions of the form $\Gamma E \mathrm{exp}(M\ln x)$ (obtained by varying $x$ over $\mathbb{R}^n_{\gg 0}$ and $E$ over $\mathcal{D}_m$) is precisely equal to $\{\Gamma V\colon V \in \mathcal{Q}'(M)\}$ (see also \cite{banajiSIAM}).
\end{remark}

\begin{thm}
\label{gen_powlaw}
Let $0 \neq \Gamma \in \mathbb{R}^{n \times m}$ have rank $r$ and $M \in \mathbb{R}^{m \times n}$. The following statements are equivalent:
\begin{enumerate}[align=left,leftmargin=*]
\item[(i)] $\Gamma \rst -M^t > 0$ or $\Gamma \rst -M^t < 0$ 
\item[(ii)] $\mathrm{rank}\,(\Gamma D_1 M D_2 \Gamma) = \mathrm{rank}\,\Gamma$ for all $D_1 \in \mathcal{D}_m$ and $D_2 \in \mathcal{D}_n$ ($\mathcal{Q}'(M)$ is $\Gamma$-nonsingular). 
\item[(iii)] For each $E \in \mathcal{D}_m$ the function $\Gamma E \mathrm{exp}(M\ln x)$ satisfies claim IC1. 
\end{enumerate}
\end{thm}
\begin{proof}
Define $w(x) = \mathrm{exp}(M\ln x)$. Notation is as in Remark~\ref{remPLJac}.

(i) $\Leftrightarrow$ (ii) follows immediately from Lemma~\ref{lemMcompat} since $\mathcal{Q}'(M)$ is $\Gamma$-nonsingular if and only if $\mathcal{Q}'(-M)$ is $\Gamma$-nonsingular. 

To prove (ii) $\Leftrightarrow$ (iii), we first show that given $x, y \in \mathbb{R}^n_{\gg 0}$, there exist $\overline{D} \in \mathcal{D}_n$ and $\tilde D \in \mathcal{D}_m$, dependent on $x$ and $y$, and such that:
\begin{equation}
\label{eqDv}
w(y)-w(x) = \tilde D M\overline{D}(y-x)\,.
\end{equation}
Choose and fix arbitrary  $x,y \in \mathbb{R}^n_{\gg 0}$, and define $\Delta x = y-x$. Note that $x + t\Delta x \in \mathbb{R}^n_{\gg 0}$ for $t \in [0,1]$ by convexity of $\mathbb{R}^n_{\gg 0}$. Since $\ln(w(x)) = M\ln x$, the Jacobian matrix of $\ln(w(x))$ is $MD_{1/x}$. By the fundamental theorem of calculus:
\[
\ln(w(y)) = \ln(w(x)) + \int_{0}^1 MD_{1/(x + t\Delta x)}\Delta x\,\mathrm{d}t = \ln(w(x)) + M\overline{D}\Delta x\,,
\]
where $\overline{D} \stackrel{\text{\tiny def}}{=} \int_{0}^1 D_{1/(x + t\Delta x)}\,\mathrm{d}t \in \mathcal{D}_n$. Consequently:
\begin{eqnarray*}
w(y) - w(x) & = & \exp(\ln(w(y))) - w(x)\\
& = & \exp(\ln(w(x)) + M\overline{D}\Delta x) - w(x)\\
& = & w(x) \circ (\exp(M\overline{D}\Delta x) - \mathbf{1})\,,
\end{eqnarray*}
where $\mathbf{1} \in \mathbb{R}^m$ is a vector of ones. As $w(x)$ is positive, and $\exp(M\overline{D}\Delta x) - \mathbf{1}$ is in the qualitative class of $M\overline{D}\Delta x$, we can define $\tilde D \in \mathcal{D}_m$ via $w(y)-w(x) = \tilde D M\overline{D}\Delta x$. 

(ii) $\Rightarrow$ (iii). Suppose there exist $x, y \in \mathbb{R}^n_{\gg 0}$, $x \simneq^\Gamma\! y$, and $E \in \mathcal{D}_m$ such that $\Gamma E w(y) = \Gamma E w(x)$. Defining $\overline{D}$ and $\tilde D$ as above and applying (\ref{eqDv}) gives
\[
0 = \Gamma E (w(y)-w(x)) = \Gamma E \tilde D M\overline{D}(y-x)\,.
\]
Then defining $D_1 \stackrel{\text{\tiny def}}{=} E \tilde D \in \mathcal{D}_m$, $D_2 \stackrel{\text{\tiny def}}{=} \overline{D} \in \mathcal{D}_n$, we see that $y-x$ is a nonzero vector in $\mathrm{im}\,\Gamma \cap \mathrm{ker}(\Gamma D_1 MD_2)$, implying that $\mathrm{rank}\,(\Gamma D_1 M D_2 \Gamma) < \mathrm{rank}\,\Gamma$.

(iii) $\Rightarrow$ (ii). Suppose there exist $D_1 \in \mathcal{D}_m, D_2 \in \mathcal{D}_n$ such that $\mathrm{rank}\,(\Gamma D_1 M D_2 \Gamma) < \mathrm{rank}\,\Gamma$, and choose $0 \neq \Delta x \in \mathrm{im}\,\Gamma \cap \mathrm{ker}(\Gamma D_1 M D_2)$. Define $x, y$ by
\[
x_i = \left\{\begin{array}{ll}\frac{\Delta x_i}{[\exp(D_2 \Delta x)]_i - 1} & (\Delta x_i \neq 0)\\1/(D_2)_{ii} & \mbox{otherwise}\end{array}\right., \quad y_i = x_i + \Delta x_i = [\exp(D_2 \Delta x)]_ix_i\,.
\]
Clearly $x$ and $y$ are positive vectors. Define $\overline{D} \in \mathcal{D}_n$ and $\tilde D \in \mathcal{D}_m$ (dependent on $x, y$) as above. Computation quickly confirms that $\overline{D} = D_2$. Set $E = D_1\tilde D^{-1} \in \mathcal{D}_m$. Then applying (\ref{eqDv}) gives
\[
\Gamma E (w(y)-w(x)) = \Gamma E \tilde D M D_2 (y-x) = \Gamma D_1 M D_2 (y-x) = 0\,.
\]
This completes the proof. \hfill
\end{proof}
\begin{remark}
\label{remMscale}
By equivalence (i) $\Leftrightarrow$ (iii) of Theorem~\ref{gen_powlaw}, observe that $\Gamma \rst -M^t \not> 0$ and $\Gamma \rst -M^t \not < 0$ occurs if and only if there exists $E \in \mathcal{D}_m$ such that $\Gamma E \mathrm{exp}(M\ln x)$ fails condition IC1. As the condition $\Gamma \rst -M^t \not> 0$ and $\Gamma \rst -M^t \not < 0$ is invariant under positive scaling of $M$, an immediate consequence is that given any $\alpha > 0$, $\Gamma E \mathrm{exp}(M\ln x)$ fails condition IC1 for some $E \in \mathcal{D}_m$ if and only if $\Gamma E' \mathrm{exp}(\alpha M\ln x)$ fails condition IC1 for some $E' \in \mathcal{D}_m$.
\end{remark}

\begin{remark} 
It may be helpful to restate the findings of Theorem~\ref{gen_powlaw} in words. Given $0 \neq \Gamma \in \mathbb{R}^{n \times m}$ and $M \in \mathbb{R}^{m \times n}$ the following are equivalent:
\begin{enumerate}[align=left,leftmargin=*]
\item Either $\Gamma$ is $r$-strongly compatible or $r$-strongly negatively compatible with $-M^t$ (where $r = \mathrm{rank}\,\Gamma$). Later, when we consider $\Gamma$ to be the stoichiometric matrix of a CRN, we will say that the CRN is ``$M$-concordant''.
\item The semiclass $\mathcal{Q}'(M)$ is $\Gamma$-nonsingular, or equivalently the reduced determinant of every power-law function $\Gamma E \mathrm{exp}(M\ln x)$ is nonvanishing on $\mathbb{R}^n_{\gg 0}$.
\item For each $c \in \mathbb{R}^n_{\gg 0}$, all power-law functions $\Gamma E \mathrm{exp}(M\ln x)$ are injective on the set $\{x \in \mathbb{R}^n_{\gg 0}\colon x \sim^\Gamma c\}$. Later, when discussing CRNs, we will term such a set a ``positive stoichiometry class''.
\end{enumerate}
\end{remark}

\begin{remark}[Extensions of Theorem~\ref{gen_powlaw}]
While Theorem~\ref{gen_powlaw} is apparently about power-law functions, the main conclusion is easily seen to apply to a much wider class of functions. Replacing $\exp(\cdot)$ and $\ln(\cdot)$ by any pair of strictly increasing diagonal $C^1$-diffeomorphisms $\theta(\cdot)$ and $\phi(\cdot)$, inverse to each other, and with domains/codomains such that $w(\cdot) = \theta(M\phi(\cdot))$ is well defined and preserves $\mathbb{R}^n_{\gg 0}$, leads nevertheless to the conclusion of Equation~\ref{eqDv}, namely that $\theta(M\phi(y))-\theta(M\phi(x)) = \tilde D M\overline{D}(y-x)$ for $x, y \in \mathbb{R}^n_{\gg 0}$.
\end{remark}

\begin{remark}[Results related to Theorem~\ref{gen_powlaw}]\label{rem:signCond} 
While the proofs here appear formally different, the fundamental ideas for the proof of Theorem~\ref{gen_powlaw} can be traced back to Craciun and Feinberg \cite{craciun}. The equivalence of (i) and (iii) in Theorem~\ref{gen_powlaw} is the object of Proposition 8.4. in Wiuf and Feliu's paper \cite{feliuwiufSIADS2013} (see also \cite[Corollary 7.6]{feliuwiufAMC2012}). 
The statement ``$\mathcal{Q}'(M)$ is $\Gamma$-nonsingular'' can be rephrased as follows: ``$M$ cannot map any nonzero vector from any qualitative class intersecting $\mathrm{im}\,\Gamma$ into any qualitative class intersecting $\mathrm{ker}\,\Gamma$'', or in more abbreviated notation: 
\[
\quad M(\mathcal{Q}(\mathrm{im}\,\Gamma)\backslash\{0\}) \cap \mathcal{Q}(\mathrm{ker}\,\Gamma)=\emptyset\,.
\]
This formulation makes the connection between Theorem~\ref{gen_powlaw} and Theorem~1.4 in  M\"uller et al. \cite{muellerAll}. Determinant conditions for injectivity in the spirit of Theorem~\ref{gen_powlaw}$(i)$ can be inferred from Craciun and Feinberg \cite{Craciun.2010ac}. 
Related determinant conditions may be obtained by exploiting the non-vanishing of the reduced determinant of $-\Gamma M$ (see Lemma \ref{lemmain0}), and by an explicit choice of basis of $\mathrm{im}\,\Gamma$; results along these lines are given in Feliu and Wiuf \cite{feliuwiufAMC2012, feliuwiufSIADS2013} and Gnacadja \cite{gnacadja}. 
\end{remark}

\begin{remark}[Condition IC1a]
\label{remic2dash}
Provided $M\geq 0$,  $w(x) = x^M$ is a continuous function on $\mathbb{R}^n_{\geq 0}$, and Lemma~\ref{lemICclosure} allows us to extend the final statement of Theorem~\ref{gen_powlaw} to ``for each $E \in \mathcal{D}_m$ the function $\Gamma v(x) = \Gamma E w(x)$ satisfies claim \apost{IC1}''. In fact following Proposition~5.2 in Feliu and Wiuf \cite{feliuwiufAMC2012}, we can do a little better: if $M \geq 0$, the final statement of Theorem~\ref{gen_powlaw} can in fact be replaced with: ``For each $E \in \mathcal{D}_m$ the function $\Gamma v(x) = \Gamma E w(x)$ satisfies claim IC1a'', where IC1a is defined as:
\begin{enumerate}[align=left,leftmargin=*]
\item[IC1a.] $x,y \in \mathbb{R}^n_{\geq 0}$, $x\simneq^\Gamma\! y$, and $v(x)+v(y) \gg 0$, imply $\Gamma v(y) \neq \Gamma v(x)$. 
\end{enumerate}
Observe that $v(x)+v(y) \gg 0$ if and only if $w(x)+w(y) \gg 0$, which is satisfied provided at least one of $x$ or $y$ lies in $\mathbb{R}^n_{\gg 0}$. Thus IC1a $\Rightarrow$ \apost{IC1} $\Rightarrow$ IC1. To see that IC1a can then replace IC1 in the final statement of Theorem~\ref{gen_powlaw}, we need only confirm that (\ref{eqDv}) remains true wherever $w(x)+w(y) \gg 0$; the remaining arguments follow through without alteration. Fix some $x,y \in \mathbb{R}^n_{\geq 0}$ such that $w_j(x)+w_j(y) > 0$ for each $j=1, \ldots, m$. Define $\mathbf{1}_x$ by $(\mathbf{1}_x)_i = 1$ if $x_i=0$ and $(\mathbf{1}_x)_i = 0$ otherwise. Define $\mathbf{1}_y$ similarly, and given $\delta>0$, define $x_\delta = x+\delta\mathbf{1}_x$, $y_\delta = y+\delta\mathbf{1}_y$. For any $\delta>0$, $x_\delta,y_\delta \in \mathbb{R}^n_{\gg 0}$ and so, by Equation~\ref{eqDv}, $w(y_\delta)-w(x_\delta) = \tilde D M\overline{D}(y_\delta-x_\delta)$ for some $\overline{D} \in \mathcal{D}_n$ and $\tilde D \in \mathcal{D}_m$ (dependent on $\delta$). 

(i) For small enough $\delta$, it is clear that $y_\delta-x_\delta \in \mathcal{Q}(y-x)$, i.e., $y_\delta-x_\delta = D'(y-x)$ for some $D' \in \mathcal{D}_n$ (dependent on $\delta$). 

(ii) For small enough $\delta$, $w(y_\delta)-w(x_\delta) \in \mathcal{Q}(w(y)-w(x))$, i.e., $w(y)-w(x) = D''(w(y_\delta) - w(x_\delta))$ for some $D'' \in \mathcal{D}_m$ (dependent on $\delta$): (a) If $w_j(y)-w_j(x) \neq 0$, then for small enough $\delta$, $(w_j(y_\delta)-w_j(x_\delta))(w_j(y)-w_j(x))>0$ by continuity of $w$. (b) If $w_j(x) = w_j(y) > 0$; then $x_i, y_i>0$ for each $i$ such that $M_{ji}>0$ and hence (for arbitrary $\delta$) $w_j(x_\delta) = w_j(x)$ and $w_j(y_\delta) = w_j(y)$. (c) Finally, $w_j(x)= w_j(y)=0$ is ruled out by assumption. 

Choosing $\delta > 0$ sufficiently small, (i) and (ii) give:
\[
w(y) - w(x) = D''(w(y_\delta)-w(x_\delta)) = D''\tilde D M\overline{D}(y_\delta-x_\delta) = D''\tilde D M\overline{D}D'(y-x).
\]
As $D''\tilde D \in \mathcal{D}_m$ and $\overline{D}D' \in \mathcal{D}_n$, (\ref{eqDv}) holds.
\end{remark}

\begin{lemma1}
\label{gen_powlaw1}
Let $\Gamma \in \mathbb{R}^{n \times m}$, $M \in \mathbb{R}^{m \times n}$ with $M_{ij} = 0$ or $M_{ij} \geq 1$ for all $i,j$ (resp., $0 \leq M \in \mathbb{R}^{m \times n}$, resp., $M \in \mathbb{R}^{m \times n}$), and suppose that $\Gamma \Bumpeq -M^t$. Then $\Gamma E x^M$ satisfies conditions \apostt{IC2} (resp, \apost{IC2}, resp., IC2), for each $E \in \mathcal{D}_m$.
\end{lemma1}
\begin{proof}
Fix $E \in \mathcal{D}_m$, define the map $w(x) = x^M$ with codomain $\mathbb{R}^m$. In the case that $M \in \mathbb{R}^{m \times n}$ with $M_{ij} = 0$ or $M_{ij} \geq 1$ for each $i,j$, $w$ is a $C^1$ map on $\mathbb{R}^{n}_{\geq 0}$ by Lemma~\ref{lemWhitney}; if $0 \leq M \in \mathbb{R}^{m \times n}$, $w$ is continuous on $\mathbb{R}^n_{\geq 0}$, and $C^1$ on $\mathbb{R}^n_{\gg 0}$; otherwise $w$ is $C^1$ on $\mathbb{R}^n_{\gg 0}$. Let $v = Ew$. For $x \in \mathbb{R}^n_{\gg 0}$ the Jacobian matrix $\Gamma\,Dv(x)$ takes the form $\Gamma M'$ where $M' \in \mathcal{Q}'(M)$ (Remark~\ref{remPLJac}). From Lemma~\ref{lemcompat_diag}, $\Gamma \Bumpeq -M^t$ implies $\Gamma \Bumpeq \mathcal{Q}'(-M^t)$ and hence, by Lemma~\ref{lemP0}, $-\Gamma M'$ is a $P_0$-matrix for all $M' \in \mathcal{Q}'(M)$. The result in each case now follows from Lemma~\ref{leminj2}. 
\hfill
\end{proof}

\begin{remark}
Clearly, we could replace the condition $\Gamma \Bumpeq -M^t$ by $\Gamma \Bumpeq M^t$ in Lemma~\ref{gen_powlaw1}: however, $\Gamma \Bumpeq -M^t$ is the situation arising in the study of CRNs.  
\end{remark}

\section{Injectivity results for CRNs}
\label{secCRNinj}

We apply the results of the previous sections to chemical reaction networks treating both general kinetics and power-law/mass action kinetics (all to be formally defined below). Throughout this section we consider a system of $m$ chemical reactions on $n$ species and generally choose and fix an ordering on species and reactions. We emphasise that no results are dependent on the choice of orderings. Reactions may or may not be reversible, but each reaction must be assigned a ``left-hand side'' and a ``right-hand side''. Where a reaction is irreversible we assume that reactants occur on the left and products on the right, namely the reaction proceeds from left to right. These conventions are merely to simplify the exposition. 

\begin{def1}[Stoichiometric matrix, left stoichiometric matrix, right stoichiometric matrix]
\label{defstoich}
Given a system of chemical reactions, define the {\em left stoichiometric matrix} $0 \leq \Gamma_l \in \mathbb{R}^{n \times m}$ and {\em right stoichiometric matrix} $0 \leq \Gamma_r \in \mathbb{R}^{n \times m}$ as follows: $(\Gamma_l)_{ij}$ is the number of molecules of species $i$ occurring on the left-hand side of reaction $j$; $(\Gamma_r)_{ij}$ is the number of molecules of species $i$ occurring on the right-hand side of reaction $j$. Define the {\em stoichiometric matrix} of the network as $\Gamma = \Gamma_r - \Gamma_l$. Any pair out of $\Gamma$, $\Gamma_l$, and $\Gamma_r$ fully specify a CRN.
\end{def1}

\begin{remark}
Note that the stoichiometric matrix is not uniquely defined, depending on the choice of orderings on the species and reactions, and (for reversible reactions) on the choice of left- and right-hand side for each reaction; when referring to the stoichiometric matrix of a system without further comment it will be assumed that these choices have been made and fixed. 
\end{remark}

\begin{def1}[Irreversible stoichiometric matrix]
Given an arbitrary CRN we may consider any reversible reaction as a pair of irreversible ones with reactants on the left and products on the right. Choosing and fixing any convenient ordering for these irreversible reactions gives a new CRN (formally speaking) whose stoichiometric matrix will be referred to as the {\em irreversible stoichiometric matrix} of the CRN. Notationally, where we need to refer both to the original stoichiometric matrix $\Gamma$ of a CRN and its irreversible stoichiometric matrix, we write $\overline{\Gamma}$ for the latter (although where there is no need for both, we generally write an arbitrary stoichiometric matrix as $\Gamma$). 
\end{def1}

\begin{def1}[Complexes, the complex digraph, and weak reversibility]
For a given CRN, the columns of $\Gamma_l$ and $\Gamma_r$ are a set of nonnegative vectors termed the {\em complexes} of the network \cite{hornjackson}. We allow, as a special case, the empty complex corresponding to the zero vector, and denoted $\emptyset$. Regarding these complexes as the vertices of a digraph, each irreversible reaction is now representable as an arc, converting a {\em source} complex into a {\em product} complex. Note that this digraph, which we term the {\em complex digraph} of the CRN, is quite distinct from its DSR graph whose vertices are individual species or reactions. A digraph is {\em weakly reversible} if each of its connected components is strongly connected, or equivalently, each arc figures in a cycle. A CRN is defined to be weakly reversible if its complex digraph is weakly reversible. Clearly, reversible CRNs are special cases of weakly reversible ones. Gunawardena \cite{gunawardenaCRNT} provides a number of equivalent characterisations of weak reversibility for CRNs.
\end{def1}

A system of chemical reactions with stoichiometric matrix $\Gamma \in\mathbb{R}^{n \times m}$ gives rise to the ODE
\begin{equation}
\label{reacsys}
\dot x = \Gamma v(x)\,.
\end{equation}
Here $x \in \mathbb{R}^n_{\geq 0}$, and $v$ describes the rates of reaction or ``kinetics'' of the system. We now consider different choices of kinetics which will play an important part in the results to follow.

\begin{def1}[General kinetics, weak general kinetics, positive general kinetics]
\label{defgenkin}
Given a CRN described by (\ref{reacsys}), we define some classes of kinetics as follows:
\begin{enumerate}[align=left,leftmargin=*]
\item {\bf General kinetics}: (i) $v$ is defined and $C^1$ on $\mathbb{R}^n_{\geq 0}$; (ii) $v$ satisfies Assumption K described in Appendix~\ref{appkin}.
\item {\bf Weak general kinetics}: (i) $v$ is defined and $C^1$ on $\mathbb{R}^n_{\gg 0}$, and continuous on $\mathbb{R}^n_{\geq 0}$; (ii) we ignore any parts of Assumption K that assume differentiability on $\partial\mathbb{R}^n_{\geq 0}$ (the details are in Appendix~\ref{appkin}). 
\item {\bf Positive general kinetics}: this is the restriction of general kinetics to the interior of the nonnegative orthant. (i) $v$ is defined and $C^1$ on $\mathbb{R}^n_{\gg 0}$; (ii) we ignore all elements of Assumption K which apply only on $\partial\mathbb{R}^n_{\geq 0}$. (The considerably reduced assumptions in this case are termed Assumption K$_\mathrm{o}$ in Appendix~\ref{appkin}.)
\end{enumerate}
\end{def1}

\begin{remark}[Assumption K]
Assumption K is a weak and physically reasonable assumption which can be summarised very roughly as ``reactions proceed if and only if all reactants are present, reaction rates are nondecreasing with reaction concentration, and reaction rates increase strictly with reactant concentration if and only if all reactants are present.'' General kinetics implies that the nonnegative orthant is forward invariant under the local semiflow generated by (\ref{reacsys}) (Lemma~\ref{lemma:invariance} in Appendix~\ref{appkin}). In the case of irreversible reactions, it also implies the assumptions termed K.1~and~K.2 in Feinberg \cite{feinberg}. Early papers treating CRNs with minimal kinetic assumptions include Angeli et al. \cite{angelipetrinet}, Banaji et al. \cite{banajiSIAM} and Craciun et al. \cite{CraciunHeltonWilliams}.
\end{remark}

\begin{def1}[Rate pattern]
\label{defratepattern}
Given a CRN $\mathcal{R}$ with some fixed left/right stoichiometric matrices, under the assumption of positive general kinetics (namely, Assumption K$_\mathrm{o}$ in Appendix~\ref{appkin}), as $x$ explores $\mathbb{R}^n_{\gg 0}$, the derivative $Dv(x)$ of $v(x)$ in (\ref{reacsys}) may vary within a set termed the {\em rate pattern} of the CRN. More precisely, the rate pattern is the set of all possible $Dv(x)$ for all functions satisfying Assumption K$_\mathrm{o}$ associated with $\mathcal{R}$. The rate pattern is a matrix-pattern (Definition~\ref{defmatrixpattern}) which is, in fact, a single qualitative class if and only if the CRN includes no reversible reaction with some species occurring on both sides of the reaction. In the case of a CRN with irreversible stoichiometric matrix $\Gamma$ and corresponding left stoichiometric matrix $\Gamma_l$, the reader may confirm that the rate pattern is just $\mathcal{Q}(\Gamma_l^t)$. In the case of a CRN with some reversible reactions, the rate pattern is given explicitly in Lemma~\ref{lemconcordrev} below.
\end{def1}

\begin{def1}[Power-law kinetics, physical power-law kinetics, power-law general kinetics, mass action kinetics, rate constants]
\label{defMACRN}
Let $\Gamma_l, \Gamma_r\in \mathbb{Z}^{n \times m}$ be the left and right stoichiometric matrices of an irreversible system of reactions, now assumed to be nonnegative integer matrices. Let $\Gamma = \Gamma_r-\Gamma_l$. Given $M \in \mathbb{R}^{n \times m}$, and $E \in \mathcal{D}_m$, we refer to (\ref{reacsys}) with $v = Ex^{M^t}$ as a {\em CRN with power-law kinetics}. Note that in general $v$ is only defined on $\mathbb{R}^n_{\gg 0}$. It is convenient to abbreviate {\em CRN with power-law kinetics and matrix of exponents $M^t$} to {\em CRN with $M$-power-law kinetics}. If $M\in \mathcal{Q}(\Gamma_l)$, we say that the system is a {\em CRN with physical power-law kinetics}. In this case, as $M \geq 0$, $\Gamma E x^{M^t}$ is defined and continuous on $\mathbb{R}^n_{\geq 0}$, and $C^1$ on $\mathbb{R}^{n}_{\gg 0}$. It is sometimes useful to consider {\em power-law general kinetics}, the intersection of power-law kinetics and general kinetics: in particular, if $M\in \mathcal{Q}(\Gamma_l)$ and all nonzero entries in $M$ are greater than or equal to $1$, $\Gamma E x^{M^t}$ is defined and $C^1$ on $\mathbb{R}^n_{\geq 0}$ (Lemma~\ref{lemWhitney}) and hence we get an instance of power-law general kinetics. The special case $M = \Gamma_{l}$ gives a CRN with {\em mass action kinetics}. In this case $\Gamma E x^{M^t}$ is a polynomial vector field on $\mathbb{R}^n$. In all cases, the diagonal entries of $E$ are termed the {\em rate constants} for the reactions. Note that rate constants are always assumed to be positive.
\end{def1}

\begin{remark}[Relationships among the different classes of kinetics]
\label{remgenMA}
Clearly every CRN with general kinetics (GK) is a CRN with weak general kinetics (WGK), which is in turn a CRN with positive general kinetics (GK$_+$). A CRN with physical power-law kinetics (PPLK) is a CRN with weak general kinetics (WGK) (see Remark~\ref{remPLJac}), and also a CRN with power-law kinetics (PLK). A CRN with power-law general kinetics (PLGK) is by definition both a CRN with physical power-law kinetics and a CRN with general kinetics. Mass action kinetics (MAK), giving rise to polynomial vector fields, is a special case of power-law general kinetics. These inclusions (all strict) are illustrated graphically. 
\begin{center}
\begin{tikzpicture}[scale=0.8]
\node at (-2.3,1) {MAK};
\node at (-1.4,1) {$\subseteq$};
\node at (-0.3,1) {PLGK};
\node[rotate=-25] at (0.85,0.7) {$\subseteq$};
\node at (1.6,0.5) {GK};
\node[rotate=25] at (2.4,0.7) {$\subseteq$};
\node[rotate=25] at (0.75,1.25) {$\subseteq$};
\node at (1.6,1.5) {PPLK};
\node[rotate=-25] at (2.5,1.25) {$\subseteq$};
\node at (3.3,1) {WGK};
\node[rotate=25] at (2.5,1.75) {$\subseteq$};
\node at (3.3,2) {PLK};
\node at (4.25,1) {$\subseteq$};
\node at (5,0.95) {GK$_+$};
\end{tikzpicture}
\end{center}
Note that a given CRN with mass action kinetics -- or indeed any fixed power-law kinetics -- is a family of vector fields parameterised by the vector of rate constants, a much smaller family than the whole class of general kinetics. As we will see, a CRN with mass action kinetics, or some other fixed physical power-law kinetics, may be injective where the same CRN may fail to be injective with general kinetics.
\end{remark}

\begin{def1}[Stoichiometric subspace, stoichiometry class, nontrivial stoichiometry class, positive stoichiometry class]
Given a CRN with stoichiometric matrix $\Gamma \in \mathbb{R}^{n \times m}$, $\mathrm{im}\,\Gamma \subseteq \mathbb{R}^n$ is termed the {\em stoichiometric subspace} of the network. Given $p \in \mathbb{R}^n_{\geq 0}$, the set:
\[
S_p = \{y \in \mathbb{R}^n_{\geq 0} \colon y \sim^\Gamma p\}
\]
is the {\em stoichiometry class of $p$}. A stoichiometry class which intersects $\mathbb{R}^n_{\gg 0}$ is {\em nontrivial}. The intersection of a nontrivial stoichiometry class with $\mathbb{R}^n_{\gg 0}$ is a {\em positive stoichiometry class}.
\end{def1}

Since Assumption K ensures forward invariance of $\mathbb{R}^n_{\geq 0}$ (Lemma~\ref{lemma:invariance} in Appendix~\ref{appkin}) and cosets of $\mathrm{im}\,\Gamma$ are also forward invariant for (\ref{reacsys}), stoichiometry classes are forward invariant sets for (\ref{reacsys}) under the assumption of general kinetics.

\begin{def1}[Fully open extension of a CRN]
Consider the system $\dot x = \Gamma v(x)$ in (\ref{reacsys}). Let $c \in \mathbb{R}^n_{\geq 0}$, $U$ be the domain of $v$, and $q\colon U \to \mathbb{R}^n_{\geq 0}$ have the same differentiability as $v$ with derivative $Dq(x) \in \mathcal{D}_n$ where differentiable. The system
\begin{equation}
\label{reacsys1}
\dot x = \Gamma v(x) + c - q(x)
\end{equation}
will be termed the {\em fully open extension} of (\ref{reacsys}) (also referred to as {\em the system with inflows and outflows}). If a claim is made for a fully open system without qualification, this means that it holds for all allowed rates $v$ and all $c$ and $q$ as above. The term {\em fully open extension} makes sense as (\ref{reacsys}) is precisely the ODE obtained by adding to the CRN inflow and outflow reactions for each species (namely reactions of the form $\emptyset \rightleftharpoons A_i$ for each species $A_i$), with the assumption of general kinetics, weak general kinetics, or general kinetics on $\mathbb{R}^n_{\geq 0}$ depending on the assumptions about domain and differentiability of $v$. Note however that if we refer to the fully open extension of a CRN with, say, mass action kinetics, to maximise generality we do not necessarily assume that the inflow and outflow reactions have mass action kinetics.
\end{def1}

\begin{remark}[Injectivity of CRNs and of their fully open extensions]
We will see below several related but distinct results about injectivity of CRNs which are not necessarily fully open on the one hand, and injectivity of fully open CRNs on the other. A natural question is how these claims relate to each other. This question has been discussed in \cite{craciun2,banajicraciun2, Craciun.2010ac,banajiJMAA,shinarfeinbergconcord2}. Roughly speaking, conditions which imply injectivity of the fully open extension of a CRN also imply injectivity of the original CRN, if and only if certain additional ``nondegeneracy'' conditions are met. These nondegeneracy conditions take slightly different forms depending on the kinetics. The details are in Corollary~\ref{corstructdiscord} and Remark~\ref{remstructdiscord} below (for general kinetics), and Corollary~\ref{cornormal} and Remark~\ref{remnormal} below (for power-law kinetics). 
\end{remark}

\subsection{Injectivity of arbitrary CRNs without kinetic assumptions}
\label{secCRNarbitrary}

We examine the functions defined by (\ref{reacsys}) and (\ref{reacsys1}). Note that a system of the form (\ref{reacsys}) or (\ref{reacsys1}) with some choice of kinetics defines a set of allowed vector fields; here a CRN with a choice of kinetics is said to be injective on some set if each allowed vector field is injective on this set. At this stage it may prove helpful to list the implications of some of the injectivity conditions defined earlier for the possibility of multiple equilibria. We refer to equilibria on the same stoichiometry class as ``compatible''. Note that the list below states {\em implications}, not definitions, of the conditions:  
\begin{enumerate}[align=left,leftmargin=*]
\item[(\apostt{IC1})] If there are two compatible equilibria, they must both lie on $\partial \mathbb{R}^n_{\geq 0}$, and further must both lie on the same facet of $\mathbb{R}^n_{\geq 0}$. Any positive equilibrium is the sole equilibrium on its class. 
\item[(\apost{IC1})] If there are two compatible equilibria, they must both lie on $\partial \mathbb{R}^n_{\geq 0}$. Any positive equilibrium is the sole equilibrium on its class. As IC1a $\Rightarrow$ \apost{IC1}, the same holds for IC1a.
\item[(IC1)] If there are two compatible equilibria, at least one must be on  $\partial \mathbb{R}^n_{\geq 0}$.
\item[(IC1$^{-}$)] Two positive, compatible, equilibria cannot both be nondegenerate. 
\item[(\apostt{IC2})] The fully open system has no more than one equilibrium on $\mathbb{R}^n_{\geq 0}$. 
\item[(\apost{IC2})] If the fully open system has two equilibria, they must both lie on $\partial \mathbb{R}^n_{\geq 0}$. Any positive equilibrium of the fully open system is the sole equilibrium of the system.
\item[(IC2)] The fully open system can have no more than one positive equilibrium. 
\end{enumerate}

Several useful results are gathered in the following lemma: in order to highlight the purely matrix-theoretic aspect of many of the results, we do not assume any class of kinetics for the time-being, but only that the derivatives of reaction rates on $\mathbb{R}^n_{\gg 0}$ belong to some matrix-pattern. However application of the lemma to CRNs with general kinetics will be immediate by Remark~\ref{remgenkin} below.
\begin{lemma1}
\label{thmCRNgen}
Let the stoichiometric matrix $0 \neq \Gamma \in \mathbb{R}^{n \times m}$ of a CRN have rank $r$, and consider the vector field $\Gamma v(x)$ defined by (\ref{reacsys}). Let $v$ be defined and $C^1$ on $\mathbb{R}^n_{\gg 0}$. Let $Dv \in \mathcal{V}$ on $\mathbb{R}^n_{\gg 0}$ where $\mathcal{V}\subseteq \mathbb{R}^{m \times n}$ is a matrix-pattern (Definition~\ref{defratepattern}). Define the conditions:
\begin{enumerate}[align=left,leftmargin=*]
\item The DSR graph $G_{\Gamma, -V}$ satisfies Condition~($*$) for each $V \in \mathcal{V}$.
\label{gencstar}
\item $\Gamma \Bumpeq -\mathcal{V}^t$.
\label{gencompat}
\item $\Gamma \rst -V^t \neq 0$ for some $V \in \mathcal{V}$. 
\label{gennondegen}
\item $\Gamma \rst -\mathcal{V}^t > 0$.
\label{genrscompat}
\item $\Gamma \rst -\mathcal{V}^t < 0$.
\label{genrsncompat}
\item $\Gamma v$ satisfies claim IC1, namely, it is injective on the relative interior of each stoichiometry class. 
\label{genIC1}
\item $\Gamma v$ satisfies claim \apost{IC1}, namely, $\Gamma v$ can only take the same value at distinct points on a stoichiometry class if they are both on $\partial \mathbb{R}^n_{\geq 0}$. 
\label{genIC11}
\item $\Gamma v$ satisfies claim \apostt{IC1}, namely, $\Gamma v$ can only take the same value at distinct points on a stoichiometry class if they are both on $\partial \mathbb{R}^n_{\geq 0}$, and in fact belong to the same facet of $\mathbb{R}^n_{\geq 0}$.
\label{genconcl1}
\item $\Gamma v$ satisfies claim IC2, namely, the fully open system is injective on $\mathbb{R}^n_{\gg 0}$.
\label{genIC2}
\item $\Gamma v$ satisfies claim \apost{IC2}, namely, the fully open system can only take the same value at two distinct points of $\mathbb{R}^n_{\geq 0}$ if they are both on $\partial \mathbb{R}^n_{\geq 0}$.
\label{genIC21}
\item $\Gamma v$ satisfies claim \apostt{IC2}, namely, the fully open system is injective on $\mathbb{R}^n_{\geq 0}$.
\label{genconcl2}
\item $\Gamma v$ satisfies claim IC1$^{-}$, namely the system can only take the same value at two distinct compatible points of $\mathbb{R}^n_{\gg 0}$ if at least one is degenerate.
\label{genIC1minus}
\end{enumerate}
Then (\ref{gencstar}) $\Rightarrow$ (\ref{gencompat}) $\Rightarrow$ (\ref{genIC2}) $\Rightarrow$ (\ref{genIC1minus}), (\ref{genrscompat}) $\Rightarrow$ (\ref{genIC1}), (\ref{genrsncompat}) $\Rightarrow$ (\ref{genIC1}), and [(\ref{gencompat}) and (\ref{gennondegen})] $\Rightarrow$ (\ref{genrscompat}). If $v$ is additionally defined and continuous on $\mathbb{R}^n_{\geq 0}$, then (\ref{genIC2}) $\Rightarrow$ (\ref{genIC21}) and (\ref{genIC1}) $\Rightarrow$ (\ref{genIC11}). If $v$ is defined and $C^1$ on $\mathbb{R}^n_{\geq 0}$, then in addition (\ref{gencompat}) $\Rightarrow$ (\ref{genconcl2}), (\ref{genrscompat}) $\Rightarrow$ (\ref{genconcl1}), and (\ref{genrsncompat}) $\Rightarrow$ (\ref{genconcl1}). These conclusions are summarised graphically as follows:

\begin{center}
\begin{tikzpicture}[scale=1.1]
\fill[color=black!30] (0.5,-0.3) -- (3.4,-0.3) -- (3.4,0.7) -- (0.5,0.7) -- cycle;
\fill[color=black!30] (0.5,2.3) -- (3.4,2.3) -- (3.4,3.3) -- (0.5,3.3) -- cycle;
\node at (1,0) {\apost{IC2}};
\node at (2.8,0) {\apost{IC1}};
\node at (1,0.5) {$\Downarrow$};
\node at (2.8,0.5) {$\Downarrow$};
\node at (1,1) {IC2};
\node at (1.55,0.95) {$\Rightarrow$};
\node at (2.1,1) {IC1$^{-}$};
\node at (2.8,1) {IC1};
\node at (1,1.5) {$\Downarrow$};
\node at (2.8,1.5) {$\Downarrow$};
\node at (1,2) {(\ref{gencompat})};
\node at (2.8,2) {(\ref{genrscompat})};
\node at (0,2) {(\ref{gencstar})};
\node at (0.5,2) {$\Rightarrow$};
\node[rotate=50] at (3.2,2.5) {$\Uparrow$};
\node[rotate=310] at (3.2,1.5) {$\Downarrow$};
\node at (1.9,2) {{\Large $\Longrightarrow$}};
\node at (1.9,1.7) {\tiny with (\ref{gennondegen})};
\node at (3.5,2) {(\ref{genrsncompat})};
\node at (1,2.5) {$\Uparrow$};
\node at (2.8,2.5) {$\Uparrow$};
\node at (1,3) {\apostt{IC2}};
\node at (2.8,3) {\apostt{IC1}};

\node at (5.1,0.4) {(if $v$ is defined and};
\node at (5.2,-0.1) {continuous on $\mathbb{R}^n_{\geq 0}$)};

\node at (4.75,3.0) {(if $v$ is defined};
\node at (4.9,2.5) {and $C^1$ on $\mathbb{R}^n_{\geq 0}$)};
\end{tikzpicture}
\end{center}

\end{lemma1}
\begin{proof}
Observe first that being a matrix-pattern $\mathcal{V}$ is strongly stable under path integration (Remark~\ref{remineq}). (\ref{gencstar}) $\Rightarrow$ (\ref{gencompat}): this is the claim of Lemma~\ref{lemmaina}. (\ref{gencompat}) $\Rightarrow$ (\ref{genIC2}): by Lemma~\ref{lemP0}, if $\Gamma \Bumpeq -\mathcal{V}^t$ then $-\Gamma V$ is a $P_0$-matrix for each $V \in \mathcal{V}$; the claim now follows from Lemma~\ref{leminj2}. (\ref{genIC2}) $\Rightarrow$ (\ref{genIC1minus}) follows from Lemma~\ref{lemICrels}. (\ref{genrscompat}) $\Rightarrow$ (\ref{genIC1}) and (\ref{genrsncompat}) $\Rightarrow$ (\ref{genIC1}) follow from Theorem~\ref{thminj}. [(\ref{gencompat}) and (\ref{gennondegen})] $\Rightarrow$ (\ref{genrscompat}) by definition and Corollary~\ref{corcompattoconcord} (see also Remark~\ref{remnondegen}). 

If $v$ is defined and $C^1$ on $\mathbb{R}^n_{\geq 0}$, then $Dv \in \overline{\mathcal{V}}$ on $\mathbb{R}^n_{\geq 0}$. (\ref{gencompat}) $\Rightarrow$ (\ref{genconcl2}): by Lemma~\ref{lemP0}, if $\Gamma \Bumpeq -V^t$ for each $V \in \mathcal{V}$ then $-\Gamma V$ is a $P_0$-matrix for each $V \in \mathcal{V}$, and by closure for each $V \in \overline{\mathcal{V}}$; the claim now follows from Lemma~\ref{leminj2}. (\ref{genrscompat}) $\Rightarrow$ (\ref{genconcl1}): that $-\Gamma v$ satisfies claim \apostt{IC1} follows from the second part of Theorem~\ref{thminj}; immediately the same holds for $\Gamma v$. (\ref{genrsncompat}) $\Rightarrow$ (\ref{genconcl1}) follows similarly from Theorem~\ref{thminj}. 

If $v$ is defined and continuous on $\mathbb{R}^n_{\geq 0}$, then the conclusions (\ref{genIC2}) $\Rightarrow$ (\ref{genIC21}) and (\ref{genIC1}) $\Rightarrow$ (\ref{genIC11}) follow from Lemma~\ref{lemICclosure}.
\hfill
\end{proof}

\begin{remark}[Implications of Lemma~\ref{thmCRNgen} for general kinetics]
\label{remgenkin}
Note that only the implications (\ref{genrscompat}) $\Rightarrow$ (\ref{genIC1}), (\ref{genrsncompat}) $\Rightarrow$ (\ref{genIC1}), (\ref{genrscompat}) $\Rightarrow$ (\ref{genconcl1}), (\ref{genrsncompat}) $\Rightarrow$ (\ref{genconcl1}), and [(\ref{gencompat}) and (\ref{gennondegen})] $\Rightarrow$ (\ref{genrscompat}) of Lemma~\ref{thmCRNgen} require $\mathcal{V}$ to be a matrix pattern: all others follow for arbitrary $\mathcal{V}$. Observe also that the Lemma immediately translates into statements about CRNs with positive general kinetics, weak general kinetics, and general kinetics if we fix the stoichiometric matrix $\Gamma$, and set $\mathcal{V}$ to be the associated rate pattern (Definition~\ref{defratepattern}).
\end{remark}

\begin{remark}
\label{remCRNtoopen}
The diagram accompanying Lemma~\ref{thmCRNgen} divides naturally into the left hand side, concerned with conclusions about a fully open CRN (IC2, \apost{IC2}, \apostt{IC2}), and the right hand side, concerned with conclusions which apply on each stoichiometry class (IC1, \apost{IC1}, \apostt{IC1}, IC1$^{-}$). The implication (9) $\Rightarrow$ (12) (namely, IC2 $\Rightarrow$ IC1$^{-}$) provides an automatic link. More important is the implication [(2) and (3)] $\Rightarrow$ (4) (namely, $\Gamma \Bumpeq -\mathcal{V}^t$ and $\neg(\Gamma \rst -\mathcal{V}^t = 0)$ $\Rightarrow$ $\Gamma \rst -\mathcal{V}^t > 0$) which provides the ``bridge'' between questions of injectivity of a CRN with general kinetics and its fully open extension, discussed further in Corollary~\ref{corstructdiscord} and Remark~\ref{remstructdiscord} below.
\end{remark}

\subsection{Concordance and accordance} 

We will provide several equivalent definitions, and a variety of results, associated with two important concepts: ``concordance'' and ``accordance''. Concordance, and related notions, are associated with injectivity of CRNs on stoichiometry classes, while accordance, and related notions, are associated with injectivity of fully open CRNs. The term ``concordance'' originates in Shinar and Feinberg \cite{shinarfeinbergconcord1}, although elements of the notion appear in various earlier papers, including Banaji and Craciun \cite{banajicraciun2} and Banaji \cite{banajiJMAA}, and the form in which we present concordance is rather different from \cite{shinarfeinbergconcord1}. The term ``accordance'' is used for the first time here, but note that the concept figures heavily in Banaji et al. \cite{banajiSIAM} and Banaji and Craciun \cite{banajicraciun2} and other related work of the first author. Apart from these references, results connected closely to both concordance and accordance have appeared implicitly before in the literature, as detailed in remarks below. We begin with some general and abstract definitions, followed by various more computationally useful formulations.

\begin{def1}[Concordance, discordance, structural discordance, accordance]
\label{defconcordbasic}
A CRN $\mathcal{R}$ is
\begin{enumerate}[align=left,leftmargin=*]
\item {\bf Concordant} iff, for all positive general kinetics, the reduced determinant of $\mathcal{R}$ is nonzero, namely all Jacobian matrices are homeomorphisms on the stoichiometric subspace. 
\item {\bf Discordant} if it is not concordant.
\item {\bf Structurally discordant} iff, for all positive general kinetics, the reduced determinant of $\mathcal{R}$ is zero, namely all Jacobian matrices, restricted to the stoichiometric subspace, are singular.
\item {\bf Accordant} iff, for all positive general kinetics, the negative of the Jacobian matrix of $\mathcal{R}$ is a $P_0$-matrix. Equivalently, all Jacobian matrices of the fully open system (\ref{reacsys1}) with positive general kinetics are nonsingular (Remark~\ref{remP0}).
\end{enumerate}
We will shortly see that these definitions make sense, namely they are true properties of a CRN, and independent of the choice of ordering on species and reactions, and of whether we treat reversible reactions as single objects or as pairs of irreversible reactions. In particular, if we make some choices and fix the stoichiometric matrix $\Gamma$, so that Assumption K$_\mathrm{o}$ (Appendix~\ref{appkin}) gives us the rate pattern $\mathcal{V}$ (Definition~\ref{defratepattern}), we have that $\mathcal{R}$ is:
\begin{enumerate}[align=left,leftmargin=*]
\item {\bf Concordant} iff $\mathrm{det}_\Gamma\,\Gamma\, V \neq 0$ for all $V \in \mathcal{V}$.
\item {\bf Discordant} iff $\mathrm{det}_\Gamma\,\Gamma\, V = 0$ for some $V \in \mathcal{V}$.
\item {\bf Structurally discordant} iff $\mathrm{det}_\Gamma\,\Gamma\, V = 0$ for all $V \in \mathcal{V}$.
\item {\bf Accordant} iff $-\Gamma\,V$ is a $P_0$-matrix for all $V \in \mathcal{V}$, namely $\mathrm{det}(-\Gamma\,V + D) > 0$ for all $V \in \mathcal{V}$ and all $D \in \mathcal{D}_n$.
\end{enumerate}
\end{def1}

Depending on the task in hand, different, equivalent, characterisations of concordance, accordance, etc., prove useful. For example, the best characterisation from the point of view of computing whether a given CRN is concordant, may not be the best from the point of view of proving additional theoretical results. 

\begin{remark}[Accordance as concordance of the fully open extension of a network]
The notion of accordance and some of its implications are developed in section 3.3 of \cite{banajicraciun2}, although the term is not used. The characterisation of accordance in Definition~\ref{defconcordbasic} as nonsingularity of all Jacobian matrices of the fully open extension of a CRN under the assumption of positive general kinetics, makes it clear that a CRN is accordant if and only if its fully open extension is concordant. Conversely concordance is the natural generalisation of accordance to CRNs which are not necessarily fully open.  
\end{remark}

The following notions are all so closely related that we present them in a group.  
\begin{def1}[$M$-concordant, semiconcordant, $M$-normal, normal, $M$-accordant, semiaccordant]
\label{defconcord}
Consider a CRN $\mathcal{R}$ with irreversible stoichiometric matrix $\Gamma \in \mathbb{R}^{n \times m}$ having rank $r$, and left stoichiometric matrix $\Gamma_l$. Let $M \in \mathbb{R}^{n \times m}$ be arbitrary. $\mathcal{R}$ is:
\begin{enumerate}[align=left,leftmargin=*]
\item {\bf $M$-concordant} if $\Gamma \rst M > 0$ or $\Gamma \rst M < 0$; equivalently, by Lemma~\ref{lemMcompat}, $\mathcal{Q}'(M^t)$ is $\Gamma$-nonsingular.
\item {\bf Semiconcordant} if it is $\Gamma_l$-concordant; equivalently $\mathcal{Q}'(\Gamma_l^t)$ is $\Gamma$-nonsingular. 
\item {\bf $M$-normal} if $\Gamma \rst M \neq 0$; equivalently, by Lemma~\ref{lemMcompat}, $\mathcal{Q}'(M^t)$ is not $\Gamma$-singular.
\item {\bf Normal} if it is $\Gamma_l$-normal; equivalently, $\mathcal{Q}'(\Gamma_l^t)$ is not $\Gamma$-singular. (See also \cite{Craciun.2010ac}). 
\item {\bf $M$-accordant} if $\Gamma \Bumpeq -M$; equivalently, by Lemma~\ref{lemP0}, $-\Gamma V$ is a $P_0$-matrix for all $V \in \mathcal{Q}'(M^t)$.
\item {\bf Semiaccordant} if it is $\Gamma_l$-accordant; equivalently, by Lemma~\ref{lemP0}, $-\Gamma V$ is a $P_0$-matrix for all $V \in \mathcal{Q}'(\Gamma_l^t)$).
\end{enumerate}
\end{def1}

\begin{lemma1}[Concordance and discordance in terms of minors for an irreversible CRN]
\label{lemconcordbasic}
Consider a CRN $\mathcal{R}$ with irreversible stoichiometric matrix $\Gamma \in \mathbb{R}^{n \times m}$ having rank $r$, and left stoichiometric matrix $\Gamma_l$. Then $\mathcal{R}$ is:
\begin{enumerate}[align=left,leftmargin=*]
\item {\bf Concordant} iff it is $M$-concordant for each $M \in \mathcal{Q}(\Gamma_l)$, namely $\mathcal{Q}(\Gamma_l^t)$ is $\Gamma$-nonsingular. Equivalently, $\Gamma \rst \mathcal{Q}(\Gamma_l) > 0$ or $\Gamma \rst \mathcal{Q}(\Gamma_l) < 0$. 
\item {\bf Discordant} iff there exists $M \in \mathcal{Q}(\Gamma_l)$ such that $\Gamma\rst M \not > 0$ and $\Gamma\rst M \not < 0$.
\item {\bf Structurally discordant} iff $\Gamma \rst \mathcal{Q}(\Gamma_l) = 0$, namely $\mathcal{Q}(\Gamma_l^t)$ is $\Gamma$-singular. Equivalently, it is not $M$-normal for any $M \in \mathcal{Q}(\Gamma_l)$.
\item {\bf Accordant} iff $\Gamma \Bumpeq \mathcal{Q}(-\Gamma_l)$. Equivalently, $-\Gamma V$ is a $P_0$-matrix for all $V \in \mathcal{Q}(\Gamma_l^t)$).
\end{enumerate}
\end{lemma1}
\begin{proof}
We need only note that the assumption of positive general kinetics implies that $\mathcal{R}$ has rate pattern $\mathcal{Q}(\Gamma^t_l)$. The claims now follow immediately from Definitions~\ref{defconcordbasic}~and~\ref{defconcord}, noting that the characterisation of discordance follows from Lemma~\ref{lemrevirrev}(3)(iii), and $\Gamma \Bumpeq \mathcal{Q}(-\Gamma_l)$ is equivalent to $-\Gamma V$ is a $P_0$-matrix for all $V \in \mathcal{Q}(\Gamma_l^t)$ by Lemma~\ref{lemP0}.
\hfill \end{proof}

\begin{remark}[Concordance as defined by Shinar and Feinberg in \cite{shinarfeinbergconcord1}]
\label{remconcord1}
It can be confirmed that for an irreversible CRN the following are equivalent:
\begin{itemize}[align=left,leftmargin=*]
\item The network is concordant in the sense of Shinar and Feinberg \cite{shinarfeinbergconcord1}.
\item The network is concordant as defined here.
\end{itemize}
Shinar and Feinberg's definition of concordance is presented in \cite{shinarfeinbergconcord1}, and this equivalence is shown in Appendix~\ref{appconcord}. Shinar and Feinberg showed that a network is concordant if and only if it is injective in a sense similar to \apost{IC1} for any {\em weakly monotonic kinetics} \cite[Definition 4.5]{shinarfeinbergconcord1}, thus obtaining a result related to some of the claims in Theorem~\ref{thmnoninjgen} below. Further details are given below.
\end{remark}

The following lemma provides computational conditions for concordance and accordance of a CRN in full generality, and confirms that these are consistent with Lemma~\ref{lemconcordbasic} for an irreversible CRN. Together with Remark~\ref{remreord}, this tells us that to confirm concordance/accordance of CRNs we can ignore both species and reaction ordering and also choose to treat reversible reactions as irreversible pairs, or not, as we wish. The freedom this latter choice affords us may lead to significant computational simplification. 
\begin{lemma1}[Concordance and accordance in terms of minors for a general CRN]
\label{lemconcordrev}
Suppose a CRN $\mathcal{R}$ has stoichiometric matrix $\Gamma = [\Gamma^1|\Gamma^2]$ with rank $r$, left stoichiometric matrix $\Gamma_l = [\Gamma_l^1|\Gamma_l^2]$, and right stoichiometric matrix $\Gamma_r = [\Gamma_r^1|\Gamma_r^2]$, where reactions corresponding to $\Gamma^1$ are reversible and those corresponding to $\Gamma^2$ are irreversible. Then with $\mathcal{V} = [\mathcal{Q}(\Gamma_l^1) - \mathcal{Q}(\Gamma_r^1)|\mathcal{Q}(\Gamma_l^2)]$, $\mathcal{R}$ is:
\begin{enumerate}[align=left,leftmargin=*]
\item {\bf Concordant} iff $\Gamma \rst \mathcal{V} > 0$ or $\Gamma \rst \mathcal{V} < 0$.
\item {\bf Discordant} iff there exists $V \in \mathcal{V}$ such that $\Gamma \rst V \not > 0$ and $\Gamma \rst V \not < 0$.
\item {\bf Structurally discordant} iff $\Gamma \rst \mathcal{V} = 0$.
\item {\bf Accordant} iff $\Gamma \Bumpeq - \mathcal{V}$.
\end{enumerate}
Moreover these characterisations are consistent with those in Lemma~\ref{lemconcordbasic}: if $\overline{\mathcal{R}}$ is the corresponding irreversible CRN, then $\mathcal{R}$ is concordant (resp., structurally discordant, resp., accordant) if and only if $\overline{\mathcal{R}}$ is concordant (resp., structurally discordant, resp., accordant) in the sense of Lemma~\ref{lemconcordbasic}.
\end{lemma1}
\begin{proof}
Assumption K$_o$ (Appendix~\ref{appkin}) implies that the rate pattern of $\mathcal{R}$ is precisely the matrix-pattern $\mathcal{V} = [\mathcal{Q}(\Gamma_l^1) - \mathcal{Q}(\Gamma_r^1)|\mathcal{Q}(\Gamma_l^2)]$. The characterisations now follow from the definitions in Definition~\ref{defconcordbasic}, noting that (2) is the negation of (1) via Lemma~\ref{lemrevirrev}(3)(iii). To directly confirm consistency with Lemma~\ref{lemconcordbasic}, without loss of generality let $\overline{\mathcal{R}}$ have stoichiometric matrix $\overline{\Gamma} = [\Gamma^1|\Gamma^2|{-\Gamma^1}]$ and left stoichiometric matrix $\overline{\Gamma}_l = [\Gamma_l^1|\Gamma_l^2|\Gamma_r^1]$. Clearly $\mathrm{rank}\,\Gamma = \mathrm{rank}\,\overline{\Gamma}$. Then by Lemma~\ref{lemrevirrev}, claim (7): 
\[
\overline{\Gamma} \rst \mathcal{Q}(\overline{\Gamma}_l) > 0\,\, \mbox{($<0$, $=0$)} \,\, \Leftrightarrow \,\, \Gamma \rst [\mathcal{Q}(\Gamma_l^1) - \mathcal{Q}(\Gamma_r^1)|\mathcal{Q}(\Gamma_l^2)] > 0\,\, \mbox{($<0$, $=0$)}\,.
\]
By the same result, for each $n \in \{1, \ldots, r\}$:
\[
\overline{\Gamma} \nst \mathcal{Q}({-\overline{\Gamma}_l}) \geq 0 \,\, \Leftrightarrow \,\, \Gamma \nst [{-\mathcal{Q}(\Gamma_l^1)} + \mathcal{Q}(\Gamma_r^1)|\mathcal{Q}({-\Gamma_l^2})] \geq 0\,,
\]
and so, $\overline{\Gamma} \Bumpeq \mathcal{Q}({-\overline{\Gamma}_l})\, \Leftrightarrow\, \Gamma \Bumpeq [{-\mathcal{Q}(\Gamma_l^1)} + \mathcal{Q}(\Gamma_r^1)|\mathcal{Q}(-\Gamma_l^2)]$. 
\hfill
\end{proof}

We close this section by noting that in the special case of weakly reversible CRNs, we need only check ``half'' of the concordance/semiconcordance conditions. 
\begin{lemma1}[Concordance/semiconcordance for weakly reversible CRNs]
\label{lemWRconcord}
Let $\mathcal{R}$ be a weakly reversible CRN with stoichiometric matrix $\Gamma$ and rate pattern $\mathcal{V}$. Then (i) $\mathcal{R}$ is concordant iff $\Gamma \rst -\mathcal{V}^t > 0$, and (ii) assuming $\Gamma$ is the irreversible stoichiometric matrix of $\mathcal{R}$, $\mathcal{R}$ is semiconcordant iff $\Gamma \rst -\Gamma_l > 0$.
\end{lemma1}
\begin{proof}
We can assume, without loss of generality by Lemma~\ref{lemconcordrev}, that $\Gamma$ is the irreversible stoichiometric matrix of $\mathcal{R}$, $\Gamma_l$ is the corresponding left stoichiometric matrix, and $\mathcal{V} = \mathcal{Q}(\Gamma_l^t)$. By Corollary~\ref{corWRconcord} in Appendix~\ref{appWRconcord}, as $\mathcal{R}$ is weakly reversible there exists a positive diagonal matrix $D$ such that $\mathrm{det}_\Gamma(-\Gamma D \Gamma_l^t) > 0$. Note that $D \Gamma_l^t \in \mathcal{Q}'(\Gamma_l^t)$, and so certainly $\mathrm{det}_\Gamma(-\Gamma V) < 0$ for all $V \in \mathcal{Q}'(\Gamma_l^t)$ is not true. Equivalently, by Lemma~\ref{lemmain0}, $\Gamma \rst \mathcal{Q}'(-\Gamma_l) < 0$ is not true. As $\mathcal{Q}'(\Gamma_l^t) \subseteq \mathcal{Q}(\Gamma_l^t)$, certainly $\Gamma \rst \mathcal{Q}(-\Gamma_l) < 0$ is not true.

(i) By Lemma~\ref{lemconcordbasic} concordance is equivalent to $\Gamma \rst \mathcal{Q}(-\Gamma_l) > 0$ or $\Gamma \rst \mathcal{Q}(-\Gamma_l) < 0$. As weak reversibility rules out $\Gamma \rst \mathcal{Q}(-\Gamma_l) < 0$, the result follows. (ii) By definition, $\mathcal{R}$ is semiconcordant iff $\Gamma \rst -\Gamma_l > 0$ or $\Gamma \rst -\Gamma_l < 0$; equivalently, by Lemma~\ref{lemcompat_diag}, $\Gamma \rst \mathcal{Q}'(-\Gamma_l) > 0$ or $\Gamma \rst \mathcal{Q}'(-\Gamma_l) < 0$. As weak reversibility rules out $\Gamma \rst \mathcal{Q}'(-\Gamma_l) < 0$, semiconcordance is equivalent to $\Gamma \rst \mathcal{Q}'(-\Gamma_l) > 0$, namely $\Gamma \rst -\Gamma_l > 0$. 
\hfill \end{proof}

\subsection{Injectivity of CRNs with general kinetics: implications of accordance and concordance}
\label{secCRNgenkin}

In this section, we spell out the implications of concordance/accordance, and of their negations, on injectivity and the existence of multiple positive equilibria for a CRN with general kinetics. We begin by noting that some CRNs never admit positive equilibria for any reasonable kinetics.
\begin{def1}[CRNs which admit positive equilibria]
\label{defposeq}
Lemma~\ref{lemnoposeq} in Appendix~\ref{appadditional} tells us that if the irreversible stoichiometric matrix $\overline \Gamma$ of a CRN has no positive vector in its kernel, then the CRN admits no positive equilibria for any class of kinetics considered in this paper (the only assumption on the kinetics is that an irreversible reaction proceeds at positive speed if all reactants are present). In this case we simply say that the CRN {\em admits no positive equilibria}. If $\mathrm{ker}\,\overline{\Gamma}$ includes a positive vector, then the CRN has a positive equilibrium for some choice of, say, mass action kinetics (Lemma~\ref{lemnoposeq}). In this case we say the CRN {\em admits positive equilibria}.
\end{def1}

Lemma~\ref{thmCRNgen} and Remark~\ref{remgenkin} tell us that concordance of a CRN is sufficient for injectivity of the system in the sense of IC1, \apost{IC1}, or \apostt{IC1} (depending on the kinetics), while accordance is sufficient for injectivity in the sense of IC2, \apost{IC2}, or \apostt{IC2} (depending on the kinetics). Part (a) of the next theorem tells us that concordance is also {\em necessary} for injectivity (in the sense of IC1) of all CRNs with physical power-law kinetics (see also Theorem~4.11 in Shinar and Feinberg \cite{shinarfeinbergconcord1}). Further, a discordant CRN either admits no positive equilibria, or admits multiple positive equilibria on a stoichiometry class for some choice of power-law general kinetics. Part (b) informs us that the fully open extension of any CRN which fails to be accordant has multiple positive equilibria for some choice of power-law general kinetics.

\begin{thm}
\label{thmnoninjgen}
Consider a CRN $\mathcal{R}$. 
\begin{enumerate}[align=left,leftmargin=*]
\item[(a)] If $\mathcal{R}$ is concordant it satisfies \apostt{IC1} for general kinetics, IC1 for positive general kinetics, and \apost{IC1} for weak general kinetics. If $\mathcal{R}$ is discordant, there exists a choice of power-law general kinetics such that $\mathcal{R}$ fails condition IC1; further, either $\mathcal{R}$ admits no positive equilibria in the sense of Definition \ref{defposeq}, or there exists a choice of power-law general kinetics such that $\mathcal{R}$ has multiple positive equilibria on some stoichiometry class.
\item[(b)] If $\mathcal{R}$ is accordant it satisfies \apostt{IC2} for general kinetics, IC2 for positive general kinetics, and \apost{IC2} for weak general kinetics. If $\mathcal{R}$ is not accordant, there exists a choice of power-law general kinetics such that $\mathcal{R}$ fails condition IC2; further, we can choose power-law general kinetics and inflows and outflows, namely $c, q(\cdot)$ in (\ref{reacsys1}), such that the fully open system has multiple positive equilibria.
\end{enumerate}
\end{thm}
\begin{proof}
Let $\mathcal{R}$ have irreversible stoichiometric matrix $\Gamma \in \mathbb{R}^{n \times m}$ with rank $r \geq 1$, and corresponding left stoichiometric matrix $\Gamma_l$.

(a) We already know from Lemma~\ref{thmCRNgen} with $\mathcal{V} = \mathcal{Q}(\Gamma_l^t)$ (see Remark~\ref{remgenkin}) that if $\mathcal{R}$ is concordant, then it satisfies condition \apostt{IC1} for general kinetics, IC1 for positive general kinetics, and hence, via Lemma~\ref{lemICrels}, \apost{IC1} for weak general kinetics. Suppose $\mathcal{R}$ is discordant so there exists $M' \in \mathcal{Q}(\Gamma_l)$ such that $\Gamma\rst -M' \not > 0$ and $\Gamma\rst -M' \not < 0$. We can assume without loss of generality that nonzero entries of $M'$ are greater than or equal to $1$ (see Remark~\ref{remMscale}). Applying Theorem~\ref{gen_powlaw}, we can choose $E \in \mathcal{D}_m$ such that $\Gamma E \mathrm{exp}(M'^t\ln x)$ fails condition IC1.

Now suppose $\mathcal{R}$ admits positive equilibria (Definition~\ref{defposeq}), namely there exists $0 \ll z \in \mathrm{ker}\,\Gamma$. Define $x = \mathbf{1}$ and $E \in \mathcal{D}_m$ via $E_{ii} = z_i$ so that for any $M \in \mathbb{R}^{m \times n}$, $\Gamma E \exp(M \ln x) = \Gamma E \mathbf{1} = 0$. Choose $M' \in \mathcal{Q}(\Gamma_l)$ as above. By Theorem~\ref{gen_powlaw}, there then exists $M_0 \in \mathcal{Q}'(M') \subseteq \mathcal{Q}(\Gamma_l)$, and nonzero $\Delta x \in \mathrm{im}\,\Gamma$ such that $\Gamma M_0 \Delta x = 0$. Assume, by scaling $\Delta x$ if necessary, that for each $i$, $|\Delta x_i| < 1$ and $|(E^{-1}M_0\Delta x)_i| < 1$.

Define $D_2 = \int_{0}^1 D_{1/(x + t\Delta x)}\,\mathrm{d}t \in \mathcal{D}_n$ as in the proof of Theorem~\ref{gen_powlaw}. Observe that $D_2$ is well defined by the assumption that $|\Delta x_i| < 1$, and that $y = x + \Delta x = \exp(D_2\Delta x) \gg 0$. Define the diagonal matrix $D_1$ via
\[
(D_1)_{ii} = \left\{\begin{array}{ll}\frac{(M_0\Delta x)_i}{\ln\left[(E^{-1}M_0\Delta x)_i + 1\right]} & \mbox{if }(M_0\Delta x)_i \neq 0\\1 & \mbox{otherwise,}\end{array}\right.
\]
for each $i = 1, \ldots, m$. Observe that $D_1$ is well defined as $|(E^{-1}M_0\Delta x)_i| < 1$, and that $D_1 \in \mathcal{D}_m$. With $M = D_1^{-1}M_0D_2^{-1}$, we can also compute that 
\[
\exp(M \ln y) = \exp(D_1^{-1}M_0\Delta x) = (E^{-1}M_0\Delta x) + \mathbf{1}\,.
\]
So $\Gamma E \exp(M \ln y) = \Gamma M_0\Delta x + \Gamma E\mathbf{1} = 0$. Since $x, y \gg 0$, $x\simneq^\Gamma\! y$, $\Gamma E \exp(M \ln x) = 0$ and $\Gamma E \exp(M \ln y) = 0$, we have shown that if $\mathcal{R}$ is discordant and admits positive equilibria, then it admits multiple positive equilibria on some stoichiometry class for $M$-power-law kinetics where $M \in \mathcal{Q}(\Gamma_l)$. To see that we can, in fact, make the nonzero entries of $M$ as large as we like, fix $\lambda > 1$ and consider the transformation $M_0 \to \lambda M_0$, $\Delta x \to \frac{1}{\lambda}\Delta x$. Clearly $\Delta x \ll \mathbf{1}$ remains true; $M_0\Delta x$, $E^{-1}M_0\Delta x$, and $D_1$ are unchanged; and $D_2 = \int_{0}^1 D_{1/(x + t\Delta x/\lambda)}\,\mathrm{d}t$ approaches the identity as $\lambda \to \infty$. By choosing $\lambda$ large, the nonzero entries of $M = D_1^{-1}M_0 D_2^{-1}$ can be made as large as we like.

(b) From Lemma~\ref{thmCRNgen} and Remark~\ref{remgenkin} it follows that if $\Gamma \Bumpeq \mathcal{Q}(-\Gamma_l)$ then $\mathcal{R}$ satisfies \apostt{IC2} for general kinetics, IC2 for positive general kinetics and hence, by Lemma~\ref{lemICrels}, \apost{IC2} for weak general kinetics. Suppose on the contrary that $\Gamma \not \Bumpeq -M'$ for some $M' \in \mathcal{Q}(\Gamma_l)$. Without loss of generality we can assume that nonzero entries of $M'$ are greater than or equal to $1$, as the relation $\Gamma \not \Bumpeq -M'$ is invariant under positive scaling of $M'$. Define $\tilde \Gamma = [\Gamma\,|\,{-I}]$, $\tilde M = [M'\,|\,I]$, so that, by Lemma~\ref{lembump}, $\tilde \Gamma \nst (-\tilde M) \not > 0$. Also, by Equation (\ref{eqbump}), $\tilde \Gamma \nst (-\tilde M) \not < 0$. By Theorem~\ref{gen_powlaw}, we can choose $E \in \mathcal{D}_{n+m}$ such that $\tilde{\Gamma} E \mathrm{exp}(\tilde{M}^t\ln x)$ fails IC1, namely $\Gamma E' \mathrm{exp}(M'^t\ln x)$ fails IC2, where $E' = E(\mathbf{m}) \in \mathcal{D}_m$.

We now follow the approach in part (a). First, choose $z \gg 0$ s.t. $c \stackrel{\text{\tiny def}}{=} -\tilde \Gamma z \gg 0$, possible by the structure of $\tilde \Gamma$. Let $x = \mathbf{1}$, so that $\exp(M\ln\,x) = \mathbf{1}$ for any $M$. Define $E \in \mathcal{D}_{m+n}$ via $E_{ii} = z_i$, so that $\tilde \Gamma E\mathbf{1} = -c$.

Since $\tilde \Gamma \nst (-\tilde M) \not > 0$ and $\tilde \Gamma \nst (-\tilde M) \not < 0$, by Theorem~\ref{gen_powlaw} there exists $M_0 \in \mathcal{Q}'(\tilde M)$ and $0 \neq \Delta x \in \mathrm{im}\,\tilde \Gamma = \mathbb{R}^n$ such that $\tilde \Gamma M_0 \Delta x = 0$. By scaling $\Delta x$ if necessary, assume for each $i$ that $|\Delta x_i| < 1$ and that $|(E^{-1}M_0\Delta x)_i| < 1$. As above, define $D_2 = \int_{0}^1 D_{1/(x + t\Delta x)}\,\mathrm{d}t \in \mathcal{D}_n$ and $D_1 \in \mathcal{D}_{n+m}$ via
\[
(D_1)_{ii} = \left\{\begin{array}{ll}\frac{(M_0\Delta x)_i}{\ln\left[(E^{-1}M_0\Delta x)_i + 1\right]} & \mbox{if }(M_0\Delta x)_i \neq 0\\1 & \mbox{otherwise.}\end{array}\right.
\]
Observe that $y = x + \Delta x = \exp(D_2\Delta x) \gg 0$ as $|\Delta x_i| < 1$, and the assumption that $|(E^{-1}M_0\Delta x)_i| < 1$ ensures that $D_1$ is well defined. With $M = D_1^{-1}M_0D_2^{-1}$, we can compute that $\exp(M \ln y) = \exp(D_1^{-1}M_0 \Delta x) = (E^{-1}M_0\Delta x) + \mathbf{1}$ and so $\tilde \Gamma E \exp(M \ln y) = \tilde \Gamma M_0 \Delta x + \tilde \Gamma E\mathbf{1} = 0 -c = -c$.

We see that $c + \tilde \Gamma E\exp(M \ln y) = c + \tilde \Gamma E \exp(M \ln x) = 0$. Exactly as in part (a), we can scale $M_0$ and $\Delta x$ so as to maintain $M_0 \Delta x$ constant, and thus make the nonzero entries of $M$ as large as we like.
\hfill \end{proof}

Weakly reversible CRNs admit positive equilibria, so we have:
\begin{cor}
\label{corWR}
A weakly reversible CRN $\mathcal{R}$ has no more than one positive equilibrium on each stoichiometry class for all choices of physical power-law kinetics if and only if it is concordant.
\end{cor}
\begin{proof}
This is immediate from Theorem~\ref{thmnoninjgen}(a) once we note that weak reversibility easily implies the existence of a positive vector in the kernel of $\overline \Gamma$ the irreversible stoichiometric matrix of $\mathcal{R}$. \hfill
\end{proof}

We summarise in Corollary~\ref{corequivs} a number of equivalences which follow from Theorem~\ref{thmnoninjgen} and earlier results with little effort, noting in advance that the equivalence of (1) and (2) in Corollary \ref{corequivs} reads almost identically to Theorem 4.11 in Shinar and Feinberg \cite{shinarfeinbergconcord1} (with positive general kinetics replaced by ``weakly monotonic kinetics''). Recall that given a function $f(x) = \Gamma v(x)$ as is (\ref{reacsys}), IC1 means injectivity of $f$ on each positive stoichiometry class, \apost{IC1} means that $f$ can only take the same value at two distinct points on a stoichiometry class if they are both on $\partial \mathbb{R}^n_{\geq 0}$, and \apostt{IC1} means that $f$ can only take the same value at distinct points on a stoichiometry class if they share a facet of $\partial \mathbb{R}^n_{\geq 0}$.
\begin{cor}
\label{corequivs}
The following are equivalent for a CRN $\mathcal{R}$:
\begin{enumerate}[align=left,leftmargin=*]
\item $\mathcal{R}$ is concordant.
\item $\mathcal{R}$ satisfies IC1 for all positive general kinetics. 
\item $\mathcal{R}$ satisfies IC1 for all weak general kinetics. 
\item $\mathcal{R}$ satisfies \apost{IC1} for all weak general kinetics.
\item $\mathcal{R}$ satisfies IC1 for all physical power-law kinetics.
\item $\mathcal{R}$ satisfies \apost{IC1} for all physical power-law kinetics. 
\item $\cal R$ satisfies IC1 for all general kinetics.
\item $\cal R$ satisfies \apost{IC1} for all general kinetics.
\item  $\cal R$ satisfies \apostt{IC1} for all general kinetics.
\end{enumerate}
\end{cor}
\begin{proof}
First, by Lemma~\ref{lemconcordrev}, we may assume without loss of generality that $\mathcal{R}$ is a system of irreversible reactions, namely, any reversible reaction can be treated as a pair of irreversible ones. (1) $\Rightarrow$ (2) follows from Theorem~\ref{thmnoninjgen}(a). (2) $\Rightarrow$ (3), (3) $\Rightarrow$ (5), and (4) $\Rightarrow$ (6) are immediate as weak general kinetics is a special case of positive general kinetics, and physical power-law kinetics is a special case of weak general kinetics (Remark~\ref{remgenMA}). (3) $\Leftrightarrow$ (4) and (5) $\Leftrightarrow$ (6) follow from Lemma~\ref{lemICrels}. (1) $\Rightarrow$ (9) follows from Theorem~\ref{thmnoninjgen}(a). (9) $\Rightarrow$ (8) $\Rightarrow$ (7) is immediate. Finally, (5) $\Rightarrow$ (1) and (7) $\Rightarrow$ (1) follow from Theorem~\ref{thmnoninjgen}(a) as power-law general kinetics is a case of both physical power-law kinetics and general kinetics.
\hfill \end{proof}

We have an analogous, but stronger, corollary for fully open systems. The result is stronger because failure of accordance is equivalent to the existence of multiple positive equilibria in the fully open system for some choice of physical power-law kinetics, without any additional assumptions. Recall that given a function $f(x) = \Gamma v(x) + c - q(x)$ as in (\ref{reacsys1}), IC2 means injectivity of $f$ on $\mathbb{R}^n_{\gg 0}$, \apost{IC2} means that $f$ can only take the same value at two distinct points of $\mathbb{R}^n_{\geq 0}$ if they are both on $\partial \mathbb{R}^n_{\geq 0}$, and \apostt{IC2} means injectivity of $f$ on $\mathbb{R}^n_{\geq 0}$.
\begin{cor}
\label{corequivs1}
The following are equivalent for a CRN $\mathcal{R}$ with fully open extension $\mathcal{R}_o$:
\begin{enumerate}[align=left,leftmargin=*]
\item $\mathcal{R}$ is accordant.
\item $\mathcal{R}$ satisfies IC2 for all positive general kinetics. 
\item $\mathcal{R}$ satisfies IC2 for all weak general kinetics. 
\item $\mathcal{R}$ satisfies \apost{IC2} for all weak general kinetics.
\item $\mathcal{R}$ satisfies IC2 for all physical power-law kinetics.
\item $\mathcal{R}$ satisfies \apost{IC2} for all physical power-law kinetics.
\item $\mathcal{R}_o$ forbids multiple positive equilibria for all physical power-law kinetics.
\item  $\cal R$ satisfies IC2 for all general kinetics.
\item  $\cal R$ satisfies \apost{IC2} for all general kinetics.
\item  $\cal R$ satisfies \apostt{IC2} for all general kinetics.
\item $\mathcal{R}_o$ forbids multiple positive equilibria for all general kinetics.
\end{enumerate}
\end{cor}
\begin{proof}
By Lemma~\ref{lemconcordrev} we may assume without loss of generality that $\mathcal{R}$ is a system of irreversible reactions.  (1) $\Rightarrow$ (2) follows from Theorem~\ref{thmnoninjgen}(b). (2) $\Rightarrow$ (3), (3) $\Rightarrow$ (5), and (4) $\Rightarrow$ (6) are immediate as weak general kinetics is a special case of positive general kinetics, and physical power-law kinetics is a special case of weak general kinetics (Remark~\ref{remgenMA}). (3) $\Leftrightarrow$ (4) and (5) $\Leftrightarrow$ (6) follow from Lemma~\ref{lemICrels}. (6) $\Rightarrow$ (7) is immediate. (7) $\Rightarrow$ (1) follows from Theorem~\ref{thmnoninjgen}(b) as power-law general kinetics is a special case of physical power-law kinetics. (1) $\Rightarrow$ (10) follows from Theorem~\ref{thmnoninjgen}(b). (10) $\Rightarrow$ (9) $\Rightarrow$ (8) $\Rightarrow$ (11) is immediate. (11) $\Rightarrow$ (1) follows from Theorem~\ref{thmnoninjgen}(b) as power-law general kinetics is a special case of general kinetics. 
\hfill \end{proof}

\begin{remark}[Concordance and weak reversibility imply persistence] 
\label{concordtopersist}
In addition to discussing the implications of concordance for injectivity, Shinar and Feinberg \cite{shinarfeinbergconcord1} proved the remarkable result that if a concordant network is weakly reversible, then it has no critical siphons, and is ``structurally persistent'' (see Appendix~\ref{appadditional}) under very weak assumptions on the kinetics. This result is reproved in elementary linear algebraic/combinatorial ways in Appendix~\ref{appWRconcord}. It follows immediately that a weakly reversible, concordant CRN with bounded stoichiometry classes has precisely one equilibrium on each nontrivial stoichiometry class, and this equilibrium is positive.
\end{remark}

{\bf Injectivity of a CRN with general kinetics and its fully open extension.} An important question is when injectivity of the fully open extension of a CRN in the sense of IC2 (resp., \apost{IC2}, resp., \apostt{IC2}) implies injectivity of the original CRN in the sense of IC1 (resp., \apost{IC1}, resp., \apostt{IC1}). This question has been answered in the results above, but we state the conclusion explicitly for completeness:
\begin{cor}
\label{corstructdiscord}
(i) An accordant CRN is concordant if and only if it is not structurally discordant (Definition~\ref{defconcord}). (ii) If a CRN satisfies IC2 (resp., \apost{IC2}, resp., \apostt{IC2}) for positive general kinetics (resp., weak general kinetics, resp., general kinetics), then it satisfies IC1 (resp., \apost{IC1}, resp., \apostt{IC1}) for positive general kinetics (resp., weak general kinetics, resp., general kinetics) if and only if it is not structurally discordant. (iii) A weakly reversible, accordant CRN is concordant. 
\end{cor}
\begin{proof}
(i) Clearly an accordant, but structurally discordant, CRN is not concordant. In the other direction, the implication [(2) and (3)] $\Rightarrow$ (4) in Lemma~\ref{thmCRNgen}, combined with Remark~\ref{remgenkin}, tells us that an accordant CRN that is not structurally discordant is concordant. (ii) follows immediately as injectivity in the sense of IC2, \apost{IC2} or \apostt{IC2} (for the relevant kinetics) is equivalent to accordance (Corollary~\ref{corequivs1}), and injectivity in the sense of IC1, \apost{IC1} or \apostt{IC1} (for the relevant kinetics) is equivalent to concordance (Corollary~\ref{corequivs}). (iii) Weakly reversible CRNs are normal (Theorem 7.2 in \cite{Craciun.2010ac}, see Lemma~\ref{WRnormal} in Appendix~\ref{appWRconcord} for a proof), and hence not structurally discordant. The result now follows from (i). \hfill
\end{proof}
\begin{remark}[Accordant + normal $\Rightarrow$ concordant (Theorem 7.4 Shinar and Feinberg \cite{shinarfeinbergconcord1})]
Corollary~\ref{corstructdiscord} tells us that an accordant CRN is concordant if and only if it is not structually discordant. However, given any $M \in \mathcal{Q}(\Gamma_l)$, we can also easily infer that an accordant network is concordant if and only if it is $M$-normal; in particular, ``accordant + not structurally discordant'' $\Leftrightarrow$ ``accordant + normal''. The implication to the left is obvious as the normal CRNs are a subset of CRNs which are not structurally discordant; in the other direction an accordant CRN which is not structurally discordant is concordant (Corollary~\ref{corstructdiscord}), and concordant CRNs are certainly normal, by Definition~\ref{defconcord}. 
\end{remark}
\begin{remark}[Injectivity of a CRN and its fully open extension: related results]
\label{remstructdiscord}
The first claim in Corollary~\ref{corstructdiscord} is closely related to Lemma~6 in Banaji \cite{banajiJMAA}, where a graph-theoretic analogue of this claim is made. The connections between injectivity of a CRN and injectivity of its fully open counterpart are the object of Theorem~8.2 in Craciun and Feinberg \cite{Craciun.2010ac}, and of related results: Theorem~2 in Craciun and Feinberg \cite{craciun2}, Theorem 7.11 in Shinar and Feinberg \cite{shinarfeinbergconcord2}, and Corollary~5.12 in Feliu and Wiuf \cite{feliuwiufAMC2012}. Underlying several such results are a basic argument on persistence of nondegenerate equilibria under small perturbations of the vector field (Lemma~B.1 in Banaji and Craciun \cite{banajicraciun2} for example), although here this argument is not required. Craciun and Feinberg \cite{Craciun.2010ac} show that normal CRNs have the property that injectivity of the fully open extension guarantees injectivity of the network for mass action kinetics (see also Shinar and Feinberg \cite{shinarfeinbergconcord1, shinarfeinbergconcord2}). This result will turn out to be an immediate consequence of results below (see Corollary~\ref{cornormal}). 
\end{remark}

\subsection{Injectivity of simply reversible CRNs with general kinetics}
In the special case where all reactions are reversible, and no species occurs on both sides of a reaction, the results of the previous section take rather special forms. The results are stated for general kinetics, but the modifications required for weak general kinetics, or positive general kinetics are minor and are left to the reader. 

\begin{def1}[Simple, simply reversible, simply irreversible]
A CRN is referred to as {\em simple} if no species occurs on both sides of any reaction. It is {\em simply reversible} if it is simple and all reactions are reversible. Implicit in this term is the choice to treat each reversible reaction as a single reaction contributing only one column to the stoichiometric matrix, rather than as a pair of irreversible reactions. A CRN is {\em simply irreversible} if it is simple and all reactions are irreversible. Each simple CRN defines a simply irreversible one where we treat each reversible reaction as a pair of irreversible ones.
\end{def1}

\begin{def1}[Positive and negative parts of a matrix: $\Gamma_{+}$, $\Gamma_{-}$]
Given a real matrix $\Gamma$, write $\Gamma_{+}$ to mean the positive part of $\Gamma$ (i.e., we set all negative entries in $\Gamma$ to zero to obtain $\Gamma_{+}$.) Similarly, define $\Gamma_{-}$ to be the negative part of $\Gamma$, so that $\Gamma = \Gamma_{+}-\Gamma_{-}$.
\end{def1}

We first show that for a simply reversible CRN $\mathcal{R}$, concordance and accordance are combinatorial properties of its stoichiometric matrix $\Gamma$ alone. Recall that a matrix is $r$-SSD if all of its $r \times r$ submatrices are either singular or sign nonsingular, and SSD if it is $r$-SSD for each $r$ (Definition~\ref{defSSD}). 
\begin{lemma1}
\label{rSSDtoconcord}
Consider a simply reversible CRN $\mathcal{R}$ with stoichiometric matrix $0 \neq \Gamma \in \mathbb{R}^{n \times m}$ having rank $r$. Let $\overline{\mathcal{R}}$ be the corresponding simply irreversible CRN with stoichiometric matrix $\overline \Gamma$ and left stoichiometric matrix $\overline{\Gamma}_l$. Then the following are equivalent: (1) $\Gamma$ is $r$-SSD; (2) $\overline \Gamma$ is $r$-SSD; (3) $\mathcal{R}$ is concordant in the sense of Lemma~\ref{lemconcordrev}; (4) $\overline{\mathcal{R}}$ is concordant, namely $\overline{\Gamma} \rst \mathcal{Q}(\overline{\Gamma}_l) > 0$. Similarly the following are equivalent: (1a) $\Gamma$ is SSD; (2a) $\overline \Gamma$ is SSD; (3a) $\mathcal{R}$ is accordant in the sense of Lemma~\ref{lemconcordrev}; (4a) $\overline{\mathcal{R}}$ is accordant, namely $\overline{\Gamma} \Bumpeq \mathcal{Q}(\overline{\Gamma}_l)$. 
\end{lemma1}
\begin{proof}
Without loss of generality assume that $\overline \Gamma = [\Gamma|{-\Gamma}]$ and hence $\overline{\Gamma}_l = [\Gamma_{-}|\Gamma_{+}]$. Observe that $\mathrm{rank}\,\Gamma = \mathrm{rank}\,\overline{\Gamma}$ and that $\mathcal{Q}(-\overline{\Gamma}_l) \subseteq \mathcal{Q}_0(\overline{\Gamma})$. (1) $\Rightarrow$ (2) and (1a) $\Rightarrow$ (2a): each $r \times r$ submatrix of $\overline \Gamma$ is either automatically singular having two collinear columns, or is simply an $r \times r$ submatrix of $\Gamma$, possibly with some columns reordered and re-signed: these operations preserve singularity and sign nonsingularity. (2) $\Rightarrow$ (1) and (2a) $\Rightarrow$ (1a) are automatic as  $\Gamma$ is a submatrix of $\overline \Gamma$ of the same rank. That (1) $\Leftrightarrow$ (3) is immediate once we observe that, for a simply reversible system: (i) $\Gamma$ is $r$-SSD is equivalent to $\Gamma \rst \mathcal{Q}(\Gamma) >0$ (Lemma~\ref{lemmain}), and (ii) $\mathcal{Q}(\Gamma) = -\mathcal{Q}(\Gamma_{-}) + \mathcal{Q}(\Gamma_{+}) = -\mathcal{Q}(\Gamma_l) + \mathcal{Q}(\Gamma_r)$. (1a) $\Leftrightarrow$ (3a) follows similarly: $\Gamma$ is SSD is equivalent to $\Gamma \Bumpeq \mathcal{Q}(\Gamma)$ (Lemma~\ref{lemmain2}), namely $\Gamma \Bumpeq (\mathcal{Q}(\Gamma_r)-\mathcal{Q}(\Gamma_l))$. (3) $\Leftrightarrow$ (4) and (3a) $\Leftrightarrow$ (4a) follow from Lemma~\ref{lemconcordrev}. 
\hfill \end{proof}

Thus for a simply reversible CRN $\mathcal{R}$:
\begin{itemize}[align=left,leftmargin=*]
\item $\mathcal{R}$ is accordant $\Leftrightarrow$ $\Gamma$ is SSD.
\item $\mathcal{R}$ is concordant $\Leftrightarrow$ $\Gamma$ is $r$-SSD, where $r = \mathrm{rank}\,\Gamma$.
\end{itemize}
As $\Gamma$ is SSD implies $\Gamma$ is $r$-SSD, accorance implies concordance for simply reversible CRNs. This if of course also automatic from Corollary~\ref{corstructdiscord}(iii), as simply reversible CRNs are weakly reversible.

\begin{thm}
\label{thmCRN1}
Consider a simply reversible CRN $\mathcal{R}$ with stoichiometric matrix $0 \neq \Gamma \in \mathbb{R}^{n \times m}$. Let $G_\Gamma$ be the SR graph of $\Gamma$. Then, with $r = \mathrm{rank}\,\Gamma$,
\begin{enumerate}[align=left,leftmargin=*]
\item[(a)] If $\Gamma$ is $r$-SSD, then $\mathcal{R}$ satisfies claim \apostt{IC1} for general kinetics. If $\Gamma$ fails to be $r$-SSD, then there exists a choice of power-law general kinetics such that $\mathcal{R}$ has multiple positive equilibria on some stoichiometry class.
\label{cc3}
\item[(b)] If $\Gamma$ is SSD, then $\mathcal{R}$ satisfies claims \apostt{IC1} and \apostt{IC2} for general kinetics. If $\Gamma$ fails to be SSD, then there exists a choice of power-law general kinetics, and inflows and outflows, such that the fully open system has multiple positive equilibria.
\label{cc2}
\item[(c)] If $G_\Gamma$ satisfies Condition~($*$), then $\mathcal{R}$ satisfies claims \apostt{IC1} and \apostt{IC2} for general kinetics.
\label{cc1}
\end{enumerate}
\end{thm}
\begin{proof}
(a) By Lemma~\ref{rSSDtoconcord}, $\Gamma$ is $r$-SSD implies that $\mathcal{R}$ is concordant, and $\mathcal{R}$ satsifies claim \apostt{IC1} for general kinetics by Theorem~\ref{thmnoninjgen}(a). If $\Gamma$ fails to be $r$-SSD then, by Lemma~\ref{rSSDtoconcord}, $\mathcal{R}$ is discordant. Observe that $\mathbf{1} \in \mathrm{ker}\,\overline{\Gamma}$ for any choice of irreversible stoichiometric matrix $\overline{\Gamma}$ (as each reaction has a corresponding oppositely directed reaction), and the existence of multiple positive equilibria on some stoichiometry class for some choice of power-law general kinetics now follows by Theorem~\ref{thmnoninjgen}(a).

(b) If $\Gamma$ is SSD, then it is certainly $r$-SSD and so satisfies claim \apostt{IC1} for general kinetics as before. By Lemma~\ref{rSSDtoconcord}, $\Gamma$ is SSD if and only if $\mathcal{R}$ is accordant, and $\mathcal{R}$ satisfies \apostt{IC2} for general kinetics by Theorem~\ref{thmnoninjgen}(b). The conclusion about multistationarity is also an immediate special case of Theorem~\ref{thmnoninjgen}(b).

(c) Finally, if $G_\Gamma$ satisfies Condition~($*$) then $\Gamma$ is SSD (Lemma~\ref{lemmain2}), and consequently $r$-SSD. The claim now follows from (a) and (b). 
\hfill
\end{proof}

\begin{remark}[Related results]
The conclusions that if $G_\Gamma$ satisfies Condition~($*$) then $\Gamma$ is SSD, and that this implies \apostt{IC2} is satisfied for general kinetics, are the subject of Banaji et al. \cite{banajiSIAM} and Banaji and Craciun \cite{banajicraciun}.
\end{remark}

\begin{remark}
Theorem~\ref{thmCRN1} and preceding results imply that a simply reversible CRN with general kinetics and SSD stoichiometric matrix satisfies:
\begin{itemize}[align=left,leftmargin=*]
\item Any positive equilibrium is the unique equilibrium on its stoichiometry class. If stoichiometry classes are bounded then each nontrivial stoichiometry class contains a positive equilibrium (Lemma~\ref{rSSDtoconcord} and Remark~\ref{concordtopersist}).
\item The fully open system has no more than one equilibrium in $\mathbb{R}^n_{\geq 0}$.
\end{itemize}
\end{remark}

\subsection{Injectivity of arbitrary CRNs with power-law/mass action kinetics}
\label{secMA}
In the discussion in this subsection and the next the stoichiometric matrix $\Gamma$ is always the irreversible stoichiometric matrix of the system. 

First we provide another characterisation of $M$-concordance and $M$-accordance (Definition~\ref{defconcord}) which makes clear the close and surprising parallels between results for power-law kinetics (with mass action as a special case), and for general kinetics, discussed further in the conclusions. 
\begin{lemma1}[$M$-concordance, $M$-accordance]
\label{lemMconcordbasic}
Let $\mathcal{R}$ be a CRN with irreversible stoichiometric matrix $\Gamma$. Let $M$ be a fixed matrix with the dimensions of $\Gamma$. Then:
\begin{enumerate}[align=left,leftmargin=*]
\item $\mathcal{R}$ is {\bf $M$-concordant} $\Leftrightarrow$ for all $M$-power-law kinetics the reduced determinant of $\mathcal{R}$ on $\mathbb{R}^n_{\gg 0}$ is nonzero $\Leftrightarrow$ $\mathrm{det}_\Gamma\,\Gamma\, V \neq 0$ for all $V \in \mathcal{Q}'(M^t)$ $\Leftrightarrow$ $\mathcal{Q}'(M^t)$ is $\Gamma$-nonsingular.
\item $\mathcal{R}$ is {\bf $M$-accordant} $\Leftrightarrow$ for all $M$-power-law kinetics, the negative of the Jacobian matrix of $\mathcal{R}$ on $\mathbb{R}^n_{\gg 0}$ is a $P_0$-matrix $\Leftrightarrow$ $-\Gamma V$ is a $P_0$-matrix for all $V \in \mathcal{Q}'(M^t)$ $\Leftrightarrow$ $\Gamma \Bumpeq \mathcal{Q}'(-M)$.
\end{enumerate}
\end{lemma1}
\begin{proof}
Recall that by Remark~\ref{remPLJac} the set of all Jacobian matrices of a CRN with $M$-power-law kinetics is $\{\Gamma V\colon V \in \mathcal{Q}'(M^t)\}$. The first result is now immediate by Lemma~\ref{lemMcompat}, and the second by Lemmas~\ref{lemP0}~and~\ref{lemcompat_diag}. \hfill
\end{proof}

We immediately have the corollary for mass action:
\begin{cor}[Semiconcordance, semiaccordance]
\label{lemsemiconcordbasic}
Let $\mathcal{R}$ be a CRN with irreversible stoichiometric matrix $\Gamma$ and left stoichiometric matrix $\Gamma_l$. Then:
\begin{enumerate}[align=left,leftmargin=*]
\item $\mathcal{R}$ is {\bf semiconcordant} $\Leftrightarrow$ for all mass action kinetics the reduced determinant of $\mathcal{R}$ on $\mathbb{R}^n_{\gg 0}$ is nonzero $\Leftrightarrow$ $\mathrm{det}_\Gamma\,\Gamma\, V \neq 0$ for all $V \in \mathcal{Q}'(\Gamma_l^t)$ $\Leftrightarrow$ $\mathcal{Q}'(\Gamma_l^t)$ is $\Gamma$-nonsingular.
\item $\mathcal{R}$ is {\bf semiaccordant} $\Leftrightarrow$ for all mass action kinetics, the negative of the Jacobian matrix of $\mathcal{R}$ on $\mathbb{R}^n_{\gg 0}$ is a $P_0$-matrix $\Leftrightarrow$ $-\Gamma V$ is a $P_0$-matrix for all $V \in \mathcal{Q}'(\Gamma_l^t)$ $\Leftrightarrow$ $\Gamma \Bumpeq \mathcal{Q}'(-\Gamma_l)$.
\end{enumerate}
\end{cor}

Observe that where (for an irreversible CRN) concordance and accordance are conditions relating minors of $\Gamma$ to minors of $M$ {\em for each $M\in\mathcal{Q}(\Gamma_l)$}, semiconcordance and semiaccordance are simply a condition relating minors of $\Gamma$ to minors of $\Gamma_l$. However, both concordance and semiconcordance can be interpreted as $\Gamma$-nonsingularity of sets of matrices related to $\Gamma_l$: the qualitative class $\mathcal{Q}(\Gamma_l)$ in the case of concordance, and the semiclass $\mathcal{Q}'(\Gamma_l)$ in the case of semiconcordance. Similarly both accordance and semiaccordance can be seen as nonsingularity of a set of matrices: $\{-\Gamma V + D\colon V \in \mathcal{Q}(\Gamma_l^t), D \in \mathcal{D}_n\}$ in the case of accordance, and $\{-\Gamma V + D\colon V \in \mathcal{Q}'(\Gamma_l^t), D \in \mathcal{D}_n\}$, in the case of semiaccordance. Interestingly, if the bipartite graph of $\Gamma_l$ includes no cycles, then $\mathcal{Q}(\Gamma_l) = \mathcal{Q}'(\Gamma_l)$ (see Remark~\ref{remsemiclass}) and in this case semiconcordance of a CRN is equivalent to concordance, and semiaccordance is equivalent to accordance. We need some further lemmas in order to be able to state, in Theorem~\ref{gen_powlaw_CRN} below, the connections between $M$-concordance and $M$-accordance on the one hand, and injectivity/ multistationarity of a CRN with $M$-power-law kinetics.

\begin{lemma1}
\label{lem_noninj}
Consider a CRN $\mathcal{R}$ with irreversible stoichiometric matrix $0 \neq \Gamma \in \mathbb{R}^{n \times m}$, and let $M \in \mathbb{R}^{n \times m}$. If $\mathcal{R}$ is not $M$-accordant (namely, $\Gamma \not \Bumpeq -M$), then $\mathcal{R}$ with $M$-power-law kinetics fails condition IC2. In particular, there exist $E' \in \mathcal{D}_m$, $E'' \in \mathcal{D}_n$, and $x, y \gg 0$, $x \neq y$, such that $\Gamma E'\mathrm{exp}(M^t\ln x) - E''x = \Gamma E'\mathrm{exp}(M^t\ln y) - E''y$. 
\end{lemma1}
\begin{proof}
Suppose $\Gamma \not \Bumpeq -M$. Define $\tilde \Gamma = [\Gamma\,|\,{-I}]$, $\tilde M = [M\,|\,I]$, so that, by Lemma~\ref{lembump}, $\tilde \Gamma \nst (-\tilde M) \not > 0$. Also by Equation (\ref{eqbump}) in Lemma~\ref{lembump}, $\tilde \Gamma \nst (-\tilde M) \not < 0$. Observe that $\tilde \Gamma$ has rank $n$, and by Theorem~\ref{gen_powlaw}, the function $f(x) = \tilde{\Gamma} E \mathrm{exp}(\tilde M^t\ln x)$ fails claim IC1 for some $E \in \mathcal{D}_{n+m}$: i.e., there exist $x,y \in \mathbb{R}^n_{\gg 0}$, such that $f(x) = f(y)$, namely
\[
\Gamma E'\mathrm{exp}(M^t\ln x) - E'' x = \Gamma E'\mathrm{exp}(M^t\ln y) - E'' y\,,
\]
where $E' =E(\{1, \ldots, m\}) \in \mathcal{D}_m$ and $E'' =E(\{m\!+\!1, \ldots, m\!+\!n\}) \in \mathcal{D}_n$.
\hfill
\end{proof}
\begin{lemma1}
\label{lem_reverse}
Consider a CRN $\mathcal{R}$ with irreversible stoichiometric matrix $\Gamma \in \mathbb{R}^{n \times m}$, and let $M \in \mathbb{R}^{n \times m}$. If $-M$ is strongly $\Gamma$-incompatible (Definition~\ref{defstrongincompat}), then the fully open extension of $\mathcal{R}$ with $M$-power-law kinetics admits multiple positive equilibria. In particular, there exist $E \in \mathcal{D}_m$, $D \in \mathcal{D}_n$, $c \gg 0$, and $x, y \gg 0$, $x \neq y$, such that $c + \Gamma E\exp(M^t\ln x) - Dx = c + \Gamma E\exp(M^t\ln y) - Dy = 0$.
\end{lemma1}
\begin{proof}
Define $\tilde \Gamma = [\Gamma\,|\,{-I}]$ and $\tilde M = [M\,|\,I]^t$. Recall that $-M$ is strongly $\Gamma$-incompatible if and only if there exists $D_0\in \mathcal{D}_{n+m}$ such that $\mathrm{det}(-\tilde \Gamma D_0 \tilde M) < 0$ and $\tilde \Gamma D_0\mathbf{1} \leq 0$. Assume that $-M$ is strongly $\Gamma$-incompatible and choose such a $D_0$. Defining $D' =D_0(\{1, \ldots, m\}) \in \mathcal{D}_m$, $D'' =D_0(\{m\!+\!1, \ldots, m\!+\!n\}) \in \mathcal{D}_n$, note that 
\[
\tilde \Gamma D_0\mathbf{1} = \Gamma D'\mathbf{1} - D''\mathbf{1} \quad \mbox{and} \quad -\tilde \Gamma D_0\,\tilde M = {-\Gamma} D' M^t + D''\,.
\]
Clearly, by increasing the diagonal elements of $D''$ we can in fact choose $D_1\in \mathcal{D}_{n+m}$ such that $\mathrm{det}\,(\tilde \Gamma D_1\tilde M) = 0$ and $\tilde \Gamma D_1\mathbf{1} \ll 0$. We now choose $0 \neq z \in \mathrm{ker}\,(\tilde \Gamma D_1 \tilde M)$. Let $x = \mathbf{1}$, $y = \exp(z) \gg 0$, and define $\tilde D(z)\in \mathcal{D}_{n+m}$ via
\[
[\tilde D(z)]_{ii} = \left\{\begin{array}{ll}\frac{\exp(\tilde M z)_i - 1}{(\tilde M z)_i} & \mbox{if } (\tilde M z)_i \neq 0,\\1 & \mbox{otherwise.}\end{array}\right.
\]
This gives $\exp(\tilde M\ln y)-\exp(\tilde M\ln x) = \exp(\tilde M z) - \mathbf{1} = \tilde D(z) \tilde M z$. Setting $E(z) = D_1\tilde D^{-1}(z) \in \mathcal{D}_{n+m}$ gives:
\begin{equation}
\label{eqA}
\tilde \Gamma E(z) (\exp(\tilde M\ln y)-\exp(\tilde M\ln x)) = \tilde \Gamma E(z) \tilde D(z) \tilde M z = \tilde \Gamma D_1 \tilde M z = 0\,.
\end{equation}
Observe that as we scale $z$ such that $z \to 0$, $\tilde D(z)$ approaches the identity matrix and thus $E(z) \to D_1$, and so $\tilde \Gamma E(z) \exp(\tilde M\ln x) = \tilde \Gamma E(z) \mathbf{1} \to \tilde \Gamma D_1\mathbf{1} \ll 0$ as $z \to 0$. Thus, by choosing $z \neq 0$ with $|z|$ sufficiently small we can guarantee that $\tilde \Gamma E(z) \mathbf{1} \ll 0$. Choose and fix such a $z$ and set $c(z) = -\tilde \Gamma E(z) \mathbf{1} \gg 0$, so that 
\begin{eqnarray*}
c(z) + \tilde \Gamma E(z) \exp(\tilde M\ln y) & = & c(z) + \tilde \Gamma E(z) \exp(\tilde M\ln x)\quad \mbox{(by Equation~\ref{eqA})}\\
& = & -\tilde \Gamma E(z) \mathbf{1} + \tilde \Gamma E(z) \mathbf{1} = 0\,,
\end{eqnarray*}
and $x,y$ are thus a pair of distinct positive equilibria for the fully open system with $c$ and the rate constants (including outflow rates) chosen appropriately.
\hfill
\end{proof}

The next theorem summarises injectivity and multistationarity results proved above for a system with fixed power-law kinetics. Part (a) tells us that $M$-concordance is necessary and sufficient for a CRN with $M$-power-law kinetics to be injective in the sense of IC1 for all choices of rate constants (and semiconcordance is necessary and sufficient for a mass action system to be injective in the sense of IC1 or IC1a for all choices of rate constants). The remainder of the theorem provides necessary and sufficient conditions for injectivity/the absence of multiple positive equilibria in the fully open system.

\begin{thm}
\label{gen_powlaw_CRN}
Let $M\in \mathbb{R}^{n \times m}$ be fixed. Consider a CRN $\mathcal{R}$ with (irreversible) stoichiometric matrix $0 \neq \Gamma \in \mathbb{R}^{n \times m}$, left stoichiometric matrix $\Gamma_l \in \mathbb{R}^{n \times m}$, and $M$-power-law kinetics. Let $\mathcal{R}_o$ be the fully open extension of $\mathcal{R}$.
\begin{enumerate}[align=left,leftmargin=*]
\item[(a)] Let $r = \mathrm{rank}\,\Gamma$. The following statements are equivalent:
\begin{enumerate}[align=left,leftmargin=*]
\item[(i)] $\mathcal{R}$ is $M$-concordant (i.e., $\Gamma \rst -M > 0$ or $\Gamma \rst -M < 0$).
\item[(ii)] $\mathrm{rank}\,(\Gamma D_1 M^t D_2 \Gamma) = \mathrm{rank}\,\Gamma$ for all $D_1 \in \mathcal{D}_m$ and $D_2 \in \mathcal{D}_n$ (i.e., $\mathcal{Q}'(M^t)$ is $\Gamma$-nonsingular).
\item[(iii)] For all rate constants, $\mathcal{R}$ satisfies claim IC1.
\end{enumerate}
If $M \geq 0$, these are additionally equivalent to:
\begin{enumerate}[align=left,leftmargin=*]
\item[(iv)] For all rate constants, $\mathcal{R}$ satisfies claim IC1a.
\end{enumerate}
\item[(b)] If $\mathcal{R}$ is $M$-accordant (i.e., $\Gamma \Bumpeq -M$) with $M \in \mathbb{R}^{n \times m}$ (resp., $0 \leq M \in \mathbb{R}^{n \times m}$, resp., $M \in \mathbb{R}^{n \times m}$ with $M_{ij} = 0$ or $M_{ij} \geq 1$ for all $i,j$), then for all rate constants, $\mathcal{R}$ satisfies claims IC2 (resp., \apost{IC2}, resp., \apostt{IC2}). 
\item[(c)] If $\mathcal{R}$ is not $M$-accordant (i.e., $\Gamma \not \Bumpeq -M$), then $\mathcal{R}$ fails condition IC2. In particular, there exist $E \in \mathcal{D}_m$, $D \in \mathcal{D}_n$, and $x, y \gg 0$, $x \neq y$, such that $\Gamma E\mathrm{exp}(M^t\ln x) - Dx = \Gamma E\mathrm{exp}(M^t\ln y) - Dy$. 
\item[(d)] If $-M$ is strongly $\Gamma$-incompatible (Definition~\ref{defstrongincompat}), then $\mathcal{R}_o$ admits multiple positive equilibria. In particular, there exist $E \in \mathcal{D}_m$, $D \in \mathcal{D}_n$, $c \gg 0$, and $x, y \gg 0$, $x \neq y$, such that $c + \Gamma E\mathrm{exp}(M^t\ln x) - Dx = c + \Gamma E\mathrm{exp}(M^t\ln y) - Dy = 0$.
\end{enumerate}
\end{thm}
\begin{proof}
(a) This follows immediately from Theorem~\ref{gen_powlaw} and Remark~\ref{remic2dash}. (b) This follows from Lemma~\ref{gen_powlaw1}. (c) This follows from Lemma~\ref{lem_noninj}. (d) This follows from Lemma~\ref{lem_reverse}.
\hfill
\end{proof}

\begin{remark}[Theorem~\ref{gen_powlaw_CRN} for mass action]
\label{remMpowlaw}
If we set $M = \Gamma_l$ in Theorem~\ref{gen_powlaw_CRN}, we immediately get the important special case of mass action kinetics. In this case, note that $0 \leq M = \Gamma_l \in \mathbb{Z}^{n \times m}$, so, for example, the system is semiconcordant if and only if $\Gamma v$ satisfies claim IC1a for all rate constants; similarly the system is semiaccordant (namely $\Gamma \Bumpeq {-\Gamma_l}$) if and only if it satisfies \apostt{IC2} for all rate constants which occurs if and only if it satisfies IC2 for all rate constants. 
\end{remark}

\begin{remark}[Results related to Theorem~\ref{gen_powlaw_CRN} in the case of mass action kinetics]
Theorem~3.1 in Craciun and Feinberg \cite{craciun} states that a fully open CRN (\ref{reacsys1}) with mass action kinetics is injective on $\mathbb{R}^n_{\gg 0}$ if and only if it has nonsingular Jacobian matrix at each $x \in \mathbb{R}^n_{\gg 0}$ and for all rate constants. By similar methods of proof, Corollary 5.9 in Feliu and Wiuf \cite{feliuwiufAMC2012} shows that changing ``Jacobian'' to ``reduced Jacobian'' in the statement above, and restricting attention to stoichiometry classes, yields a result that holds for any CRN, not necessarily fully open. Bearing in mind Remark~\ref{remPLJac}, these are immediate consequences of Theorem~\ref{gen_powlaw_CRN}(a). 
The result in part (d) of Theorem~\ref{gen_powlaw_CRN} giving sufficient conditions for multiple positive equilibria in a fully open system with power-law kinetics, is close to that of Theorem~4.1 in Craciun and Feinberg \cite{craciun}. A related result also appears in Feliu \cite{feliuRoySoc}.
The equivalence of $(a)(i)$, $(a)(iii),$ and the sign condition mentioned in Remark \ref{rem:signCond} is the object of Theorem 3.4. in M\"uller et al. \cite{muellerAll}. 
\end{remark}

{\bf Injectivity of a CRN with power-law kinetics and its fully open extension.} Quite analogously to the situation for general kinetics, it is natural to ask of a CRN $\mathcal{R}$ with fixed power-law kinetics when injectivity of the fully open extension in the sense of IC2 (resp., \apost{IC2}) implies injectivity of $\mathcal{R}$ in the sense of IC1 (resp., IC1a). Where in the case of general kinetics a necessary and sufficient nondegeneracy condition was that $\mathcal{R}$ should not be structurally discordant, for fixed $M$-power-law kinetics (including mass action as a special case), a necessary and sufficient condition is that $\mathcal{R}$ should be $M$-normal. Recall that an irreversible CRN with stoichiometric matrix $\Gamma$ is $M$-normal if $\Gamma \rst M \neq 0$, or equivalently, $\mathcal{Q}'(M^t)$ is not $\Gamma$-singular. By Remark~\ref{remPLJac} the set of all Jacobian matrices of a CRN with $M$-power-law kinetics is $\{\Gamma V\colon V \in \mathcal{Q}'(M^t)\}$, and so $M$-normal CRNs are precisely those which have nonzero reduced determinant somewhere on $\mathbb{R}^n_{\gg 0}$ for $M$-power-law kinetics and some choice of rate constants. Similarly, normal CRNs are those which have nonzero reduced determinant somewhere on $\mathbb{R}^n_{\gg 0}$ for mass action kinetics and some choice of rate constants. We have the following corollary of Theorem~\ref{gen_powlaw_CRN}:
\begin{cor}
\label{cornormal}
Let $\mathcal{R}$ be a CRN with irreversible stoichiometric matrix $\Gamma$ and, let $M$ be any matrix with the dimensions of $\Gamma$. (i) If $\mathcal{R}$ is $M$-accordant, then it is $M$-concordant if and only if it is $M$-normal. (ii) If $\mathcal{R}$ satisfies IC2 (resp., \apost{IC2}) for power-law kinetics (resp., physical power-law kinetics) with matrix of exponents $M^t$, then it satisfies IC1 (resp., IC1a) for this kinetics if and only if it is $M$-normal.
\end{cor}
\begin{proof}
(i) Observe that $M$-accordance ($\Gamma \Bumpeq -M$) rules out $\Gamma \rst -M < 0$, and implies $\Gamma \rst -M > 0$ if and only if $\Gamma \rst M \neq 0$. Thus $M$-accordance implies $M$-concordance if and only if $\mathcal{R}$ is $M$-normal. (ii) By Theorem~\ref{gen_powlaw_CRN}, injectivity of $\mathcal{R}$ in the sense of IC2 or \apost{IC2} (depending on kinetics) is equivalent to $M$-accordance, and injectivity of $\mathcal{R}$ in the sense of IC1 or \apost{IC1} (depending on kinetics) is equivalent to $M$-concordance. The result thus follows from (i).
\hfill \end{proof}

\begin{remark}[Related results: injectivity of a CRN with mass action kinetics from injectivity of its fully open extension]
\label{remnormal}
The particular case of Corollary~\ref{cornormal} for mass action kinetics (namely where $M = \Gamma_l$) is the subject of the main theorem (Theorem 8.2) of \cite{Craciun.2010ac}. 
\end{remark}

\subsection{Injectivity of simple CRNs with mass action kinetics}
\label{secMA1}
Results in Banaji et al. \cite{banajiSIAM} on the special case of a simple CRNs with mass action kinetics motivate the following definitions. 
\begin{def1}[WSD, $r$-strongly WSD, $r$-strongly negatively WSD]
Observe that if $\Gamma$ is the irreversible stoichiometric matrix of a simple CRN, then $\Gamma_r = \Gamma_{+}$ and $\Gamma_l = \Gamma_{-}$. A matrix $\Gamma$ with rank $r \geq 1$ is termed {\em $r$-strongly WSD} if $\Gamma \rst {-(\Gamma_{-})} > 0$ and {\em $r$-strongly negatively WSD} if $\Gamma \rst -(\Gamma_{-}) < 0$. It is WSD if $\Gamma \Bumpeq {-(\Gamma_{-})}$.
\end{def1}

\begin{remark}[WSD matrices]
The acronym WSD was originally an abbreviation of ``weakly sign determined'' in  \cite{banajiSIAM}, where it was shown that every SSD matrix is WSD, but not vice versa. An example of a matrix of rank $r$ which is $r$-strongly negatively WSD is:
\[
\Gamma = \left(\begin{array}{rr}-1&2\\1&-1\end{array}\right)\quad \mbox{so that} \quad -(\Gamma_{-}) = \left(\begin{array}{rr}-1&0\\0&-1\end{array}\right).
\]
We see that $\Gamma$ has rank $2$ and is $2$-strongly negatively WSD as $(\mathrm{det}\,\Gamma)(\mathrm{det}\,(-(\Gamma_{-}))) < 0$. An example of a WSD matrix that is not $r$-strongly WSD is:
\[
\Gamma = \left(\begin{array}{rr}-1&-1\\0&1\\1&0\end{array}\right)\quad \mbox{so that} \quad -(\Gamma_{-}) = \left(\begin{array}{rr}-1&-1\\0&0\\0&0\end{array}\right).
\]
It is easy to see that $\Gamma$ is WSD. However, as $\mathrm{rank}\,\Gamma > \mathrm{rank}\,\Gamma_{-}$ it cannot be $2$-strongly WSD. An example of a matrix which is $r$-strongly WSD, but not WSD is:
\[
\Gamma = \left(\begin{array}{rrr}-1&0&0\\2&-1&0\\-1&1&-1\end{array}\right)\quad \mbox{so that} \quad -(\Gamma_{-}) = \left(\begin{array}{rrr}-1&0&0\\0&-1&0\\-1&0&-1\end{array}\right)\,.
\]
$\Gamma$ has rank $3$ and is $3$-strongly WSD as $(\mathrm{det}\,\Gamma)(\mathrm{det}\,(-(\Gamma_{-}))) > 0$. But 
\[
\Gamma[\{2,3\}|\{1,2\}]\,(-(\Gamma_{-}))[\{2,3\}|\{1,2\}] < 0
\]
and so it is not WSD.
\end{remark}

For reference when discussing examples, we write out in full the following specialisation of Theorem~\ref{gen_powlaw_CRN} to the case of simple CRNs with mass action kinetics.
\begin{thm}
\label{thmWSD}
Consider a simple CRN $\mathcal{R}$ with irreversible stoichiometric matrix $0 \neq \Gamma \in \mathbb{Z}^{n \times m}$ and mass action kinetics. Let $\mathcal{R}_o$ be the fully open extension of $\mathcal{R}$. 
\begin{enumerate}[align=left,leftmargin=*]
\item[(a)] Let $r = \mathrm{rank}\,\Gamma$. The following statements are equivalent:
\begin{enumerate}[align=left,leftmargin=*]
\item[(i)] $\Gamma$ is $r$-strongly WSD or $r$-strongly negatively WSD (namely, $\Gamma \rst {-(\Gamma_{-})} > 0$ or $\Gamma \rst {-(\Gamma_{-})} < 0$).
\item[(ii)] $\mathrm{rank}\,(\Gamma D_1 \Gamma_{-}^t D_2 \Gamma) = \mathrm{rank}\,\Gamma$ for all $D_1 \in \mathcal{D}_m$ and $D_2 \in \mathcal{D}_n$ (i.e., $\mathcal{Q}'(\Gamma_{-}^t)$ is $\Gamma$-nonsingular).
\item[(iii)] For all rate constants $\mathcal{R}$ satisfies claim IC1.
\item[(iv)] For all rate constants $\mathcal{R}$ satisfies claim IC1a.
\end{enumerate}
If $\mathcal{R}$ is weakly reversible, these are additionally equivalent to
\begin{enumerate}[align=left,leftmargin=*]
\item[(v)] $\Gamma$ is $r$-strongly WSD (namely, $\Gamma \rst {-(\Gamma_{-})} > 0$).
\end{enumerate}
\item[(b)] If $\Gamma$ is WSD, then for all rate constants $\mathcal{R}$ satisfies conditions \apostt{IC2}: for arbitrary rate constants and inflows and outflows, $\mathcal{R}_o$ is injective on $\mathbb{R}^n_{\geq 0}$. 
\item[(c)] If $\Gamma$ is not WSD, then for some choice of rate constants $\mathcal{R}$ fails condition IC2. In particular, there exist $E \in \mathcal{D}_m$, $D \in \mathcal{D}_n$, and $x, y \gg 0$, $x \neq y$, such that $\Gamma E\mathrm{exp}(\Gamma_{-}^t\ln x) - Dx = \Gamma E\mathrm{exp}(\Gamma_{-}^t\ln y) - Dy$.
\item[(d)] If $-(\Gamma_{-})$ is strongly $\Gamma$-incompatible (Definition~\ref{defstrongincompat}), then $\mathcal{R}_o$ admits multiple positive equilibria. In particular, there exist $E \in \mathcal{D}_m$, $D \in \mathcal{D}_n$, $c \gg 0$, and $x, y \gg 0$, $x \neq y$, such that $c + \Gamma E\mathrm{exp}(\Gamma_{-}^t\ln x) - Dx = c + \Gamma E\mathrm{exp}(\Gamma_{-}^t\ln y) - Dy = 0$.
\end{enumerate}
\end{thm}
\begin{proof}
(a) Note that by definition the condition that ``$\Gamma$ is $r$-strongly WSD or $r$-strongly negatively WSD'' is equivalent to ``$\mathcal{R}$ is semiconcordant''. Equivalence of (i) to (iv) is immediate from the definitions and Theorem~\ref{gen_powlaw_CRN}(a) with $M = \Gamma_{-}$. Equivalence of (i) and (v) follows once we observe that for simple, weakly reversible CRNs, semiconcordance is equivalent to $\Gamma \rst {-(\Gamma_{-})} > 0$ by Lemma~\ref{lemWRconcord}. (b) By definition, $\Gamma$ is WSD if and only if $\mathcal{R}$ is semiaccordant. The result is now a special case of Theorem~\ref{gen_powlaw_CRN}(b). (c) and (d) follow from Theorem~\ref{gen_powlaw_CRN}(c)~and~(d) with $M = \Gamma_{-}$.
\hfill
\end{proof}

\begin{remark}[Related results]
The result in Theorem~\ref{thmWSD}(b) is a corollary of the results in Section~4 of Banaji et al. \cite{banajiSIAM}. The result in Theorem~\ref{thmWSD}(d) can be inferred from Theorem~4.1 in Craciun and Feinberg \cite{craciun}. 
\end{remark}

Figure~\ref{fig:implications} summarises some of the results on injectivity and the absence of multiple positive equilibria (MPE) for a system of irreversible reactions. Figure~\ref{fig:implications1} summarises some of the results for fully open systems.

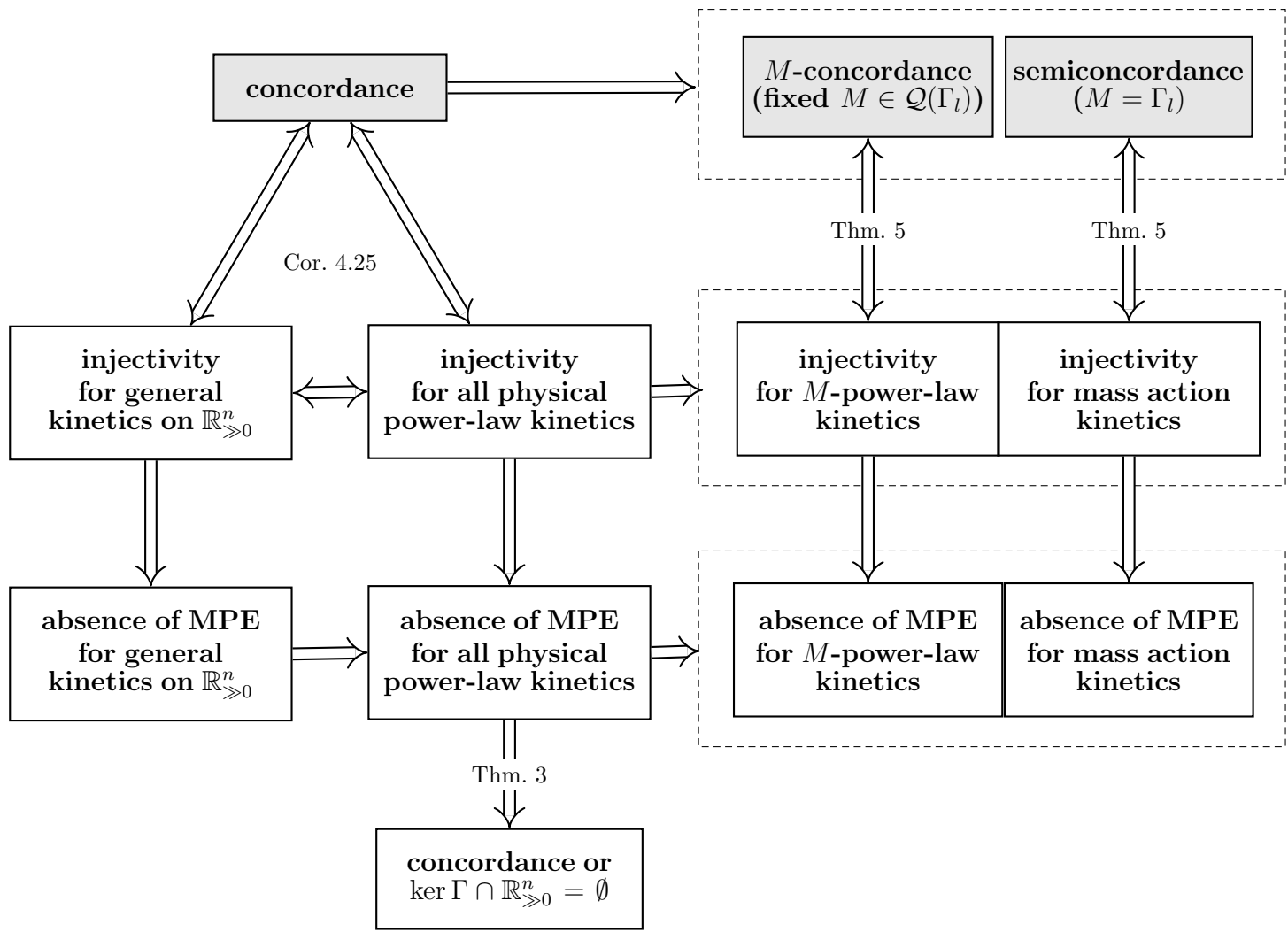
\begin{figure}
\begin{tikzpicture}[rotate=270, scale=0.98] 
\node (conc) [draw, thick, minimum width=2cm,minimum height=1cm, anchor=base, align=center,rotate=90, minimum width=3.5cm,minimum height=1cm, fill=ulightgray] at (12.25,-0.15) {\large\bf concordance};
\node (msemi) [draw, thick, minimum width=2cm,minimum height=1cm, anchor=base, align=center,rotate=90, minimum width=3.5cm,minimum height=1.5cm, fill=ulightgray] at (4,0) {{\large\bf $M$-concordance}\\{\large\bf (fixed $M\in {\cal Q}(\Gamma_l)$)}};
\node (semi) [draw, thick, minimum width=2cm,minimum height=1cm, anchor=base, align=center,rotate=90, minimum width=3.5cm,minimum height=1.5cm, fill=ulightgray] at (0,0) {{\large\bf semiconcordance}\\{\large\bf($M =\Gamma_l)$}};
\node (injgk) [draw, thick, minimum width=2cm,minimum height=1cm, anchor=base, align=center,rotate=90, text width=4cm, minimum width=2cm,minimum height=2cm, ] at (15,4) {\centering\large\bf injectivity\\for general\\ kinetics on $\mathbb{R}^n_{\gg 0}$};
\node (injpl) [draw, thick, minimum width=2cm,minimum height=1cm, anchor=base, align=center,rotate=90, text width=4cm,minimum width=2cm,minimum height=2cm] at (9.5,4) {\centering\large\bf injectivity \\for all physical\\ power-law kinetics};
\node (injmpl) [draw, thick, minimum width=2cm,minimum height=1cm, anchor=base, align=center,rotate=90, text width=3.7cm, minimum width=3.5cm,minimum height=2cm] at (4,4) {\centering\large\bf injectivity \\for $M$-power-law\\ kinetics};
\node (injma) [draw, thick, minimum width=2cm,minimum height=1cm, anchor=base, align=center,rotate=90, text width=3.7cm,minimum width=3.5cm,minimum height=2cm] at (0,4) {\centering\large\bf injectivity \\for mass action \\kinetics};
\node (mpegk) [draw, thick, minimum width=2cm,minimum height=1cm, anchor=base, align=center,rotate=90, text width=4cm,minimum width=2cm,minimum height=2cm] at (15,8) {\centering\large\bf absence of MPE\\ for general\\ kinetics on $\mathbb{R}^n_{\gg 0}$};
\node (mpepl) [draw, thick, minimum width=2cm,minimum height=1cm, anchor=base, align=center,rotate=90, text width=4cm,minimum width=2cm,minimum height=2cm] at (9.5,8) {\centering\large\bf absence of MPE\\ for all physical\\ power-law kinetics};
\node (mpempl) [draw, thick, minimum width=2cm,minimum height=1cm, anchor=base, align=center,rotate=90, text width=3.8cm, minimum width=3.5cm,minimum height=2cm] at (4,8) {\centering\large\bf  absence of MPE \\for $M$-power-law\\ kinetics};
\node (mpema) [draw, thick, minimum width=2cm,minimum height=1cm, anchor=base, align=center,rotate=90, text width=3.5cm,minimum width=3.5cm,minimum height=2cm] at (0,8) {\centering\large\bf absence of MPE\\for mass action\\ kinetics};
\node (extra) [draw, thick, minimum width=2cm,minimum height=1cm, anchor=base, align=center,rotate=90, text width=3.5cm,minimum width=4cm,minimum height=1.5cm] at (9.5,11.7) {\centering\large\bf concordance or\\ $\ker\Gamma\cap {\mathbb R}_{\gg 0}^n=\emptyset$};

\node  [minimum width=2cm,minimum height=1cm, anchor=base, align=center,rotate=90, text width=3.5cm,minimum width=4cm,minimum height=1.5cm] at (12.25,2.5) {Cor.~\ref{corequivs}};

\draw[densely dashed](-2.4,-1.5) rectangle +(9,2.6);
\draw[densely dashed](-2.4,2.8) rectangle +(9,3);
\draw[densely dashed](-2.4,6.8) rectangle +(9,3);

\draw[double,thick,-implies, double distance=4pt, shorten > =.7cm] (conc.east) -- (msemi.west);
\draw[double,thick,-implies, double distance=4pt, shorten > =.5cm] (injpl.east) -- (injmpl.west);
\draw[double,thick,-implies, double distance=4pt, shorten > =.5cm] (mpepl.east) -- (mpempl.west);

\draw[double,thick,implies-implies, double distance=4pt] (semi.south) -- node[pos=.5, rotate=90, fill=white] {Thm.~\ref{gen_powlaw_CRN}} (injma.north);
\draw[double,thick,implies-implies, double distance=4pt] (msemi.south) -- node[pos=.5, rotate=90, fill=white] {Thm.~\ref{gen_powlaw_CRN}} (injmpl.north);
\draw[double,thick,-implies, double distance=4pt] (injma.south) -- (mpema.north);
\draw[double,thick,-implies, double distance=4pt] (injmpl.south) -- (mpempl.north);
\draw[double,thick,-implies, double distance=4pt] (injgk.south) -- (mpegk.north);
\draw[double,thick,-implies, double distance=4pt] (injpl.south) -- (mpepl.north);
\draw[double,thick,-implies, double distance=4pt] (mpepl.south) -- node[pos=.5, rotate=90, fill=white] {Thm.~\ref{thmnoninjgen}} (extra.north);

\draw[double,thick,implies-implies, double distance=4pt] (injgk.east) -- (injpl.west);
\draw[double,thick,-implies, double distance=4pt] (mpegk.east) -- (mpepl.west);
\draw[double,thick,implies-implies, double distance=4pt] (conc) -- (injpl);
\draw[double,thick,implies-implies, double distance=4pt] (conc) -- (injgk);

\node [minimum width=2cm,minimum height=1cm, anchor=base, align=left,rotate=90, minimum width=3.5cm,minimum height=1cm] at (7.5,15) {{\bf MPE}: multiple positive equilibria on a stoichiometry class;\\ {\bf $M$-power-law kinetics}: power-law kinetics with matrix of exponents $M^t$ (here restricted to $M\in {\cal Q}(\Gamma_l)$);\\{\bf concordance}: $\Gamma \circ^r M>0\ \forall M\in {\cal Q}(\Gamma_l)$ or $\Gamma \circ^r M<0\  \forall M\in {\cal Q}(\Gamma_l)$;\\ {\bf $M$-concordance}: $\Gamma \circ^r M>0$\ or $\Gamma \circ^r M<0$; {\bf semiconcordance}: $\Gamma \circ^r \Gamma_l>0$\ or $\Gamma \circ^r \Gamma_l<0$;};

\end{tikzpicture}
\caption{A schematic summarising some results on injectivity and the absence of multiple positive equilibria (MPE) on a stoichiometry class for a CRN with irreversible stoichiometric matrix $\Gamma \in \mathbb{R}^{n \times m}$ and corresponding left stoichiometric matrix $\Gamma_l\in \mathbb{R}^{n \times m}$. Results on fully open systems are gathered in Figure~\ref{fig:implications1}, and specialisations are omitted. The implications without labels follow immediately from other implications or from the definitions.}\label{fig:implications}
\end{figure}

\begin{figure}[p]
\begin{tikzpicture}[rotate=270, scale=0.98] 
\node (comp) [draw, thick, minimum width=5cm,minimum height=1cm, anchor=base, align=center,rotate=90, text width=5cm, minimum width=5cm,minimum height=1cm, fill=ulightgray] at (13,0) {\centering\large\bf Accordance: $\Gamma \Bumpeq \mathcal{Q}({-\Gamma_l})$};
\node (IC2gen) [draw, thick, minimum width=5cm,minimum height=1cm, anchor=base, align=center,rotate=90, text width=5cm, minimum width=5cm,minimum height=1.5cm] at (13,2) {\centering\large\bf injectivity of fully open system for positive general kinetics};
\node (IC2PPL) [draw, thick, minimum width=5cm,minimum height=1cm, anchor=base, align=center,rotate=90, text width=5cm, minimum width=5cm,minimum height=1.5cm] at (13,4.5) {\centering\large\bf injectivity of fully open system for all physical power-law kinetics};
\node (MP3PPL) [draw, thick, minimum width=5cm,minimum height=1cm, anchor=base, align=center,rotate=90, text width=5cm, minimum width=5cm,minimum height=1.5cm] at (13,7.5) {\centering\large\bf absence of MPE in fully open system for all physical power-law kinetics};
\draw[double,thick,implies-implies, double distance=4pt] (comp.south) -- (IC2gen.north);
\draw[double,thick,implies-implies, double distance=4pt] (IC2gen.south) -- (IC2PPL.north);
\draw[double,thick,implies-implies, double distance=4pt] (IC2PPL.south) -- (MP3PPL.north);
\node (legend) [minimum width=5cm,minimum height=1cm, anchor=base, align=center,rotate=90, text width=5cm, minimum width=5cm,minimum height=1.5cm] at (13,13) {\centering\large\bf (Cor.~\ref{corequivs1})};

\node (comp1) [draw, thick, minimum width=5.2cm,minimum height=1cm, anchor=base, align=center,rotate=90, text width=5cm, minimum width=5.2cm,minimum height=1cm, fill=ulightgray] at (6,0) {\centering\large\bf $\Gamma \Bumpeq -M$ ($M \in \mathcal{Q}(\Gamma_l)$)};
\node (IC2gen1) [draw, thick, minimum width=5cm,minimum height=1.5cm, anchor=base, align=center,rotate=90, text width=5cm, minimum width=5cm,minimum height=1.5cm] at (6,4.5) {\centering\large\bf injectivity of fully open system for M-power-law kinetics};
\node (IC2PPL1) [draw, thick, minimum width=5cm,minimum height=1cm, anchor=base, align=center,rotate=90, text width=5cm, minimum width=5cm,minimum height=1.5cm] at (6,7.7) {\centering\large\bf absence of MPE in fully open system for M-power-law kinetics};
\node (MP3PPL1) [draw, thick, minimum width=5cm,minimum height=1cm, anchor=base, align=center,rotate=90, text width=5cm, minimum width=5cm,minimum height=1.5cm] at (6,11) {\centering\large\bf $-M$ is not strongly $\Gamma$-incompatible};
\draw[double,thick,implies-implies, double distance=4pt] (comp1.south) -- (IC2gen1.north);
\draw[double,thick,-implies, double distance=4pt] (IC2gen1.south) -- (IC2PPL1.north);
\draw[double,thick,-implies, double distance=4pt] (IC2PPL1.south) -- (MP3PPL1.north);
\node (legend) [minimum width=5cm,minimum height=1cm, anchor=base, align=center,rotate=90, text width=5cm, minimum width=5cm,minimum height=1.5cm] at (6,13) {\centering\large\bf (Thm.~\ref{gen_powlaw_CRN})};

\node (comp2) [draw, thick, minimum width=5.2cm,minimum height=1cm, anchor=base, align=center,rotate=90, text width=5cm, minimum width=5.2cm,minimum height=1cm, fill=ulightgray] at (0.5,0) {\centering\large\bf $\Gamma \Bumpeq {-\Gamma_l}$};
\node (IC2gen2) [draw, thick, minimum width=5cm,minimum height=1.5cm, anchor=base, align=center,rotate=90, text width=5cm, minimum width=5cm,minimum height=1.5cm] at (0.5,4.55) {\centering\large\bf injectivity of fully open system for mass action kinetics};
\node (IC2PPL2) [draw, thick, minimum width=5cm,minimum height=1cm, anchor=base, align=center,rotate=90, text width=5cm, minimum width=5cm,minimum height=1.5cm] at (0.5,7.75) {\centering\large\bf absence of MPE in fully open system for mass action kinetics};
\node (MP3PPL2) [draw, thick, minimum width=5cm,minimum height=1cm, anchor=base, align=center,rotate=90, text width=5cm, minimum width=5cm,minimum height=1.5cm] at (0.5,11) {\centering\large\bf ${-\Gamma_l}$ is not strongly $\Gamma$-incompatible};
\draw[double,thick,implies-implies, double distance=4pt] (comp2.south) -- (IC2gen2.north);
\draw[double,thick,-implies, double distance=4pt] (IC2gen2.south) -- (IC2PPL2.north);
\draw[double,thick,-implies, double distance=4pt] (IC2PPL2.south) -- (MP3PPL2.north);
\node (legend) [minimum width=5cm,minimum height=1cm, anchor=base, align=center,rotate=90, text width=5cm, minimum width=5cm,minimum height=1.5cm] at (0.5,13) {\centering\large\bf (Thm.~\ref{gen_powlaw_CRN})};
\draw[densely dashed](-2.5,-1) rectangle +(11.5,1.75);
\draw[double,thick,-implies, double distance=4pt, shorten > =.3cm] (comp.east) -- (comp1.west);

\draw[densely dashed](-2.5,3.65) rectangle +(11.5,2.3);
\draw[double,thick,-implies, double distance=4pt, shorten > =.3cm] (IC2PPL.east) -- (IC2gen1.west);

\draw[densely dashed](-2.5,6.9) rectangle +(11.5,2.25);
\draw[double,thick,-implies, double distance=4pt, shorten > =.3cm] (MP3PPL.east) -- (IC2PPL1.west);

\node [minimum width=2cm,minimum height=1cm, anchor=base, align=left,rotate=90, minimum width=3.5cm,minimum height=1cm] at (7.5,15.5) {{\bf MPE}: multiple positive equilibria;\\ {\bf $M$-power-law kinetics}: power-law kinetics with matrix of exponents $M^t$ (here restricted to $M\in {\cal Q}(\Gamma_l)$);\\{\bf $B$ is strongly $A$-incompatible}: there exists $D\in \mathcal{D}_{n+m}$ such that $\mathrm{det}(\tilde AD \tilde B^t) < 0$ and $\tilde AD\mathbf{1} \leq 0$\\ (where $\tilde A = [A|{-I}]$ and $\tilde B = [B|{-I}]$) -- see Definition~\ref{defstrongincompat});};

\end{tikzpicture}
\caption{A schematic summarising some results on injectivity and the absence of multiple positive equilibria (MPE) for a fully open CRN with irreversible stoichiometric matrix $\Gamma \in \mathbb{R}^{n \times m}$ and corresponding left stoichiometric matrix $\Gamma_l\in \mathbb{R}^{n \times m}$. The implications without labels follow immediately from other implications or from the definitons.}\label{fig:implications1}
\end{figure}

\section{Extensions and examples}
\label{secexamples}

We examine some examples chosen to demonstrate the subtleties or limitations of the various results above. In some cases techniques in the literature beyond the scope of this paper augment or clarify or expand the conclusions: particularly worth mentioning are deficiency theory, and applications of the theory of monotone dynamical systems to CRNs. All computations are carried out in {\tt CoNtRol} \cite{control}. Before presenting the examples we list some conditions which may strengthen conclusions about injectivity or multistationarity of a CRN. The first additional condition which may apply is:
\begin{enumerate}[align=left,leftmargin=*]
\item[BC1.] Stoichiometry classes are bounded. 
\end{enumerate}
It is well known that BC1 holds if and only if $\mathrm{ker}\,\Gamma^t \cap \mathbb{R}^n_{\gg 0} \neq \emptyset$ (Lemma~\ref{lembdclass} in Appendix~\ref{appadditional}) and implies that each stoichiometry class is a nonempty compact, convex polyhedron and hence, by the Brouwer fixed point theorem, includes an equilibrium of (\ref{reacsys}). It sometimes occurs that:
\begin{enumerate}[align=left,leftmargin=*]
\item[PC0.] The CRN admits no positive equilibria (Definition~\ref{defposeq}).
\end{enumerate}
If $\overline{\Gamma}\in \mathbb{R}^{n \times m}$ is the irreversible stoichiometric matrix of a CRN and $\mathrm{ker}\,\overline{\Gamma} \cap \mathbb{R}^n_{\gg 0} = \emptyset$, then claim PC0 follows for all classes of kinetics considered in this paper, whereas otherwise the CRN admits a positive equilibrium for mass action kinetics with some choice of rate constants  (Lemma~\ref{lemnoposeq} in Appendix~\ref{appadditional}). So PC0 is equivalent to $\mathrm{ker}\,\overline{\Gamma} \cap \mathbb{R}^n_{\gg 0} = \emptyset$. Perhaps more interesting are:
\begin{enumerate}[align=left,leftmargin=*]
\item[PC1.] No stoichiometry class, other than possibly a stoichiometry class consisting only of $\{0\}$, includes any equilibria on $\partial \mathbb{R}^n_{\geq 0}$.
\item[PC2.] No nontrivial stoichiometry class includes any equilibria on $\partial \mathbb{R}^n_{\geq 0}$.
\end{enumerate}
Observe that (i) PC1 implies that the only possible equilibrium on $\partial \mathbb{R}^n_{\geq 0}$ is $0$, and (ii) PC1 implies PC2. Claims PC1 and PC2 are reached via examination of the so-called ``siphons'' of the system (see \cite{angelipetrinet, shiusturmfels} for example). PC2 holds if the CRN has no critical siphons; PC1 holds if the system has no siphons at all, other than possibly the set of all species, in which case this siphon is non-critical. The details are in Appendix~\ref{appadditional}. We recall that claim PC2 holds automatically if we know that the CRN is concordant and weakly reversible (Remark~\ref{concordtopersist} and Appendix~\ref{appWRconcord}).

\begin{remark}[Implications of \apost{IC1} combined with persistence and boundedness]
\label{remplus}
Note first that \apostt{IC1} $\Rightarrow$ \apost{IC1} and IC1a $\Rightarrow$ \apost{IC1}, so the observations in this remark apply if we replace \apost{IC1} with \apostt{IC1} or IC1a. Claims \apost{IC1} and PC2 (or PC1) together imply that no nontrivial stoichiometry class includes more than one equilibrium. If, additionally, BC1 holds (namely, stoichiometry classes are bounded), then each nontrivial stoichiometry class includes a unique equilibrium, and this equilibrium is positive. Claims \apost{IC1}, PC1 and BC1 together imply, by the Brouwer fixed point theorem, that each stoichiometry class other than $\{0\}$ contains a unique equilibrium, which is positive (an indirect consequence of BC1 and PC1 is that all stoichiometry classes other than $\{0\}$ must in fact be nontrivial). In summary, we have the implications:
\begin{enumerate}[align=left,leftmargin=*]
\item \apost{IC1} + PC2 $-$ BC1: no nontrivial stoichiometry class includes more than one equilibrium (they may have no equilibria). An equilibrium on a nontrivial stoichiometry class, if it exists, must be positive.
\item \apost{IC1} + PC2 + BC1: each nontrivial stoichiometry class includes exactly one equilibrium; this equilibrium is positive.
\item \apost{IC1} + PC1 $-$ BC1: no stoichiometry class includes more than one equilibrium (they may have no equilibria). Equilibria, if any, are positive.
\item \apost{IC1} + PC1 + BC1: each stoichiometry class other than $\{0\}$ is nontrivial, and includes exactly one equilibrium; this equilibrium is positive. 
\end{enumerate}
\end{remark}

\subsection{Examples of simply reversible CRNs}
In the examples to follow, we report mainly conclusions for general kinetics, and for mass action kinetics. However, the reader may easily infer similar conclusions for weak general kinetics, positive general kinetics, power-law kinetics, or physical power-law kinetics, using the theorems and lemmas above.

\begin{example}[The strongest possible claims I]
\label{exbasic}
$A+B\rightleftharpoons C,\,\, 2A\rightleftharpoons B$. The stoichiometric matrix $\Gamma$, $-Dv^t$, and SR graph $G_\Gamma$ are shown:

\begin{tikzpicture}[domain=0:4,scale=0.4]
\node at (-22,0) {$\Gamma = \left(\begin{array}{rr}-1&-2\\-1&1\\1&0\end{array}\right)$};

\node at (-12,0) {$-Dv^t = \left(\begin{array}{rr}-&-\\-&+\\+&0\end{array}\right)$};

\node at (-5,0) {$G_\Gamma = $};
\fill (-1.5,1.5) circle (5pt);
\node at  (1.5,1.5) {$B$};

\node at (-3.3,2.4) {$C$}; \draw[-, line width=0.04cm] (-2.8,2.25)
-- (-1.9, 1.75); 
\fill (1.5,-1.5) circle (5pt); \node at (-1.5,-1.5) {$A$};

\path (135-15: 2.12cm) coordinate (A1end); \draw[-, dashed, line
width=0.04cm] (A1end)  arc (135-15:45+15:2.12cm);

\path (225-15: 2.12cm) coordinate (A2end); \draw[-, dashed, line
width=0.04cm] (A2end)  arc (225-15:135+15:2.12cm);

\path (315-15: 2.12cm) coordinate (A3end); \draw[-, dashed, line
width=0.04cm] (A3end)  arc (315-15:225+15:2.12cm);
\node at (0,-1.75) {$\scriptstyle{2}$};

\path (45-15: 2.12cm) coordinate (A4end); \draw[-, line
width=0.04cm] (A4end)  arc (45-15:-45+15:2.12cm);

\end{tikzpicture}\\
{\tt Report.} General kinetics. $G_\Gamma$ satisfies Condition~($*$). By Theorem~\ref{thmCRN1} both claims \apostt{IC1} and \apostt{IC2} hold. As PC1 and BC1 also hold, each stoichiometry class other than $\{0\}$ contains a unique equilibrium, which is positive (Remark~\ref{remplus}).

{\tt Remark.} In fact, claims \apostt{IC1} and \apostt{IC2} hold if the species participate in these reactions with any stoichiometries, rather than the particular values chosen, and if one or both reactions are set to be irreversible (in either direction); the CRN remains accordant and concordant and \apostt{IC1} and \apostt{IC2} follow by Theorem~\ref{thmnoninjgen}. As with several examples to follow, various other tools allow conclusions about the network beyond questions of injectivity or multistationarity. This network is weakly reversible with deficiency zero and the stoichiometric subspace has dimension $2$: by Theorem 6.3 in Pantea \cite{panteapersistence}, assuming mass action kinetics, the unique equilibrium on each nontrivial stoichiometry class is in fact globally asymptotically stable relative to its stoichiometry class.
\end{example}

\begin{example}[The strongest possible claims II]
$A+B\rightleftharpoons C\rightleftharpoons A+D,\,\,E+B\rightleftharpoons F\rightleftharpoons E+D$. This is the reversible version of the so-called ``futile cycle'' presented in Example~\ref{exfutile} later. The stoichiometric matrix $\Gamma$ and SR graph $G_\Gamma$ are shown:

\begin{tikzpicture}[domain=0:4,scale=0.45]
\node at (-11,4.5) {$\Gamma = \left(\begin{array}{rrrr}-1&1&0&0\\-1&0&-1&0\\1&-1&0&0\\0&1&0&1\\0&0&-1&1\\0&0&1&-1\end{array}\right)$};

\node at (-1,4.5) {$G_\Gamma = $};

\node at (1,1.5) {$B$};
\fill (1,5) circle (4pt);
\fill (5,1.5) circle (4pt);
\node at (1,7.5) {$C$};
\fill (4,7.5) circle (4pt);
\node at (8,7.5) {$D$};
\node at (4,5) {$A$};
\node at (5,4) {$E$};
\node at (8,1.5) {$F$};
\fill (8,4) circle (4pt);
\draw[-, thick] (1.5,1.5) -- (4.6,1.5);
\draw[-, thick, dashed] (5.4,1.5) -- (7.3,1.5);
\draw[-, thick, dashed] (5.5,4) -- (7.6,4);

\draw[-, thick, dashed] (1,2.1) -- (1,4.6);
\draw[-, thick] (1, 5.4) -- (1,7);
\draw[-, thick] (4, 5.5) -- (4,7.1);

\draw[-, thick, dashed] (1.7,7.5) -- (3.6,7.5);
\draw[-, thick] (4.4,7.5) -- (7.4,7.5);

\draw[-, thick, dashed] (1.4,5) -- (3.5,5);

\draw[-, thick] (8,2.1) -- (8,3.6);
\draw[-, thick, dashed] (8, 4.4) -- (8,6.9);

\draw[-, thick] (5,1.9) -- (5,3.55);

\end{tikzpicture}\\
 
{\tt Report.} General kinetics. $G_\Gamma$ satisfies Condition~($*$), and $\Gamma$ is hence SSD and $r$-SSD by Theorem~\ref{thmCRN1}. Thus both claims \apostt{IC1} and \apostt{IC2} hold. As the system is simply reversible, PC2 is automatic (Remark~\ref{concordtopersist}), and as BC1 also holds, each nontrivial stoichiometry class contains a unique equilibrium, which is positive (Remark~\ref{remplus}).

{\tt Remark.} This system also satisfies certain conditions of Theorem~2 in Angeli et al. \cite{angelileenheersontag}, and of Theorem 2.2 in Donnell and Banaji \cite{donnellbanaji}. Either of these theorems can be used to infer that (with general kinetics) all initial conditions on any nontrivial stoichiometry class converge to an equilibrium which is positive and is the unique equilibrium on its stoichiometry class.
\end{example}

\begin{example}[Injectivity on stoichiometry classes, but not of the fully open extension]
\label{exweaker}
Consider he system $A+B\rightleftharpoons C,\,\, 2B\rightleftharpoons C+D,\,\,C\rightleftharpoons\emptyset$ with stoichiometric matrix $\Gamma$ and irreversible stoichiometric matrix $\overline{\Gamma}$ given by
\[
\Gamma = \left(\begin{array}{rrr}-1&0&0\\-1&-2&0\\1&1&-1\\0&1&0\end{array}\right)\,, \quad \overline{\Gamma} = \left(\begin{array}{rrrrrr}-1&1&0&0&0&0\\-1&1&-2&2&0&0\\1&-1&1&-1&-1&1\\0&0&1&-1&0&0\end{array}\right)\,.
\] 
{\tt Report.} (i) General kinetics: $\Gamma$ has rank $3$ and is $3$-SSD, but not SSD (namely, concordant, but not accordant). By Theorem~\ref{thmCRN1}(a) claim \apostt{IC1} holds. As PC1 also holds no nontrivial stoichiometry class includes more than one equilibrium. As stoichiometry classes are unbounded, we cannot actually infer the existence of equilibria on stoichiometry classes. By Theorem~\ref{thmCRN1}(b), the fully open system has multiple positive equilibria for some choice of power-law general kinetics. (ii) Mass action kinetics: as $\overline{\Gamma}$ fails to be WSD, by Theorem~\ref{thmWSD}(c), the system fails condition \apostt{IC2}, namely the fully open system fails to be injective for some choice of rate constants and inflows and outflows.

{\tt Remark.} Interestingly, if the reaction $C\rightleftharpoons\emptyset$ is omitted, then the conclusion about injectivity no longer holds. However, the system $A+B\rightleftharpoons C,\,\, 2B\rightleftharpoons C+D$ is of some interest in its own right: (i) As this is a simply reversible system whose irreversible stoichiometric matrix fails to be $2$-strongly WSD, by Theorem~\ref{thmWSD}, the CRN with mass action kinetics fails condition IC1 for some choice of rate constants. This does not however imply multiple positive equilibria: it is a weakly reversible, deficiency zero network with stoichiometric subspace of dimension $2$; so, with mass action kinetics, each nontrivial stoichiometry class includes exactly one positive equilibrium which attracts all positive initial conditions on its stoichiometry class \cite{panteapersistence}; (ii) $A+B\rightleftharpoons C,\,\, 2B\rightleftharpoons C+D$ defines a monotone dynamical system on each stoichiometry class for general kinetics (Corollary A.7 in \cite{banajidynsys}) and, via Theorem~0.2.2 in \cite{halsmith}, admits no nontrivial attracting periodic orbits.
\end{example}

\begin{example}[Injectivity claims with mass action kinetics only]
\label{exWSDandnotSSD}
Consider the system $A+B\rightleftharpoons C,\,\, 2A+B\rightleftharpoons D$ with stoichiometric matrix $\Gamma$ and irreversible stoichiometric matrix $\overline{\Gamma}$ given by:
\[
\Gamma = \left(\begin{array}{rr}-1 & -2\\-1 & -1\\1 & 0\\0 & 1 \end{array}\right)\,, \quad \overline{\Gamma} = \left(\begin{array}{rrrr}-1 & 1 & -2 & 2\\-1 & 1 & -1 & 1\\1 & -1 & 0 & 0\\0 & 0 & 1 & -1 \end{array}\right)
\]
{\tt Report.} (i) General kinetics. $\Gamma$ has rank $2$, but is neither SSD, nor $2$-SSD (namely, neither accordant nor concordant) and so, by Theorem~\ref{thmCRN1}(a), the system has multiple positive equilibria on a stoichiometry class for some choice of power-law general kinetics and, by Theorem~\ref{thmCRN1}(b), the fully open system has multiple positive equilibria for some choice of power-law general kinetics. (ii) Mass action kinetics. $\overline{\Gamma}$ is both WSD and $2$-strongly WSD, and by Theorem~\ref{thmWSD} both claims IC1a and \apostt{IC2} hold. The fully open system has no more than one equilibrium on $\mathbb{R}^4_{\geq 0}$. Further, PC2 and BC1 hold, so in fact (with mass action kinetics), each nontrivial stoichiometry class includes a unique equilibrium, which is positive. In this example, the assumption of mass action significantly strengthens conclusions for both the CRN and its fully open extension.

{\tt Remark.} This system satisfies certain conditions of Theorem~2 in Angeli et al. \cite{angelileenheersontag} and consequently (with general kinetics) almost all positive initial conditions converge to the set of equilibria (the Lebesgue measure of the set of possibly non-convergent initial conditions is zero). From above, with mass action kinetics, this ``set of equilibria'' intersects each nontrivial stoichiometry class in a unique point. We thus get generic convergence to a unique equilibrium on nontrivial stoichiometry classes for mass action kinetics, without using deficiency theory. 
\end{example}

The next example is only a slight variant on Example~\ref{exWSDandnotSSD}, where some inflows and outflows have been added, but gives different conclusions, illustrating that care is needed in analysing even simple networks.

\begin{example}[Stronger injectivity claims with mass action kinetics] 
\label{exvariant}
Consider the system $A+B\rightleftharpoons C,\,\, 2A+B\rightleftharpoons D,\,\, B\rightleftharpoons \emptyset,\,\, D\rightleftharpoons \emptyset$ with stoichiometric matrix $\Gamma$ and irreversible stoichiometric matrix $\overline{\Gamma}$ given by:
\[
\Gamma = \left(\begin{array}{rrrr}-1 & -2 & 0 & 0\\-1 & -1 & -1 & 0\\1 & 0 & 0 & 0\\0 & 1 & 0 & -1\end{array}\right)\,, \quad \overline{\Gamma} = \left(\begin{array}{rrrrrrrr}-1 & 1 & -2 & 2 & 0 & 0 & 0 & 0\\-1 & 1 & -1 & 1 & -1 & 1 & 0 & 0\\1 & -1 & 0 & 0 & 0 & 0 & 0 & 0\\0 & 0 & 1 & -1 & 0 & 0 & -1 & 1\end{array}\right)\,.
\]
{\tt Report.} (i) General kinetics. $\mathrm{rank}\,\Gamma = 4$ so the only stoichiometry class is $\mathbb{R}^4_{\geq 0}$. $\Gamma$ is $4$-SSD but not SSD, so by Theorem~\ref{thmCRN1}, claim \apostt{IC1} holds. As PC1 also holds, the $\mathbb{R}^4_{\geq 0}$ includes no more than one equilibrium. By Theorem~\ref{thmCRN1}(b), the fully open system has multiple positive equilibria for some choice of power-law general kinetics. Thus, even though the CRN and its fully open extension both have the same stoichiometry class (namely the whole of $\mathbb{R}^4_{\geq 0}$), the conclusions are quite different. (ii) Mass action kinetics. $\overline{\Gamma}$ is both WSD and $4$-strongly WSD (the CRN is both semiaccordant and semiconcordant), so by Theorem~\ref{thmWSD} both claims IC1a and \apostt{IC2} hold. In this example, the assumption of mass action significantly strengthens conclusions for the fully open extension.

{\tt Remark.} This example and the previous one illustrate rather interesting behaviour: adding some, but not all, inflows and outflows to the CRN in Example \ref{exWSDandnotSSD} led to the loss of multistationarity on positive stoichiometry classes, while adding the remaining outflows led to its return. Note that the addition of some inflows and outflows caused a change in the stoichiometric subspace, and this behaviour is thus consistent with the results in Joshi and Shiu \cite{joshishiu}. In fact, this is a weakly reversible deficiency zero network \cite{feinberg} and so, with mass action kinetics, each nontrivial stoichiometry class has exactly one equilibrium, which is positive, and is locally asymptotically stable relative to its stoichiometry class. As the system is in fact complex-balanced \cite{hornjackson} and persistent (since PC1 holds), we can infer that the unique positive equilibrium on each nontrivial stoichiometry class in fact attracts the whole of its stoichiometry class \cite{siegelmaclean}.
\end{example}

\begin{example}[Claims via deficiency theory only]
\label{exdef1}
Consider the system $A\rightleftharpoons 2B,\,\, A\rightleftharpoons 2C,\,\, A\rightleftharpoons B+C$ with stoichiometric matrix $\Gamma$ and irreversible stoichiometric matrix $\overline{\Gamma}$ given by:
\[
\Gamma = \left(\begin{array}{rrr}-1 & -1 & -1\\2 & 0 & 1\\0 & 2 & 1\end{array}\right)\,, \quad \overline{\Gamma} = \left(\begin{array}{rrrrrr}-1&1&-1&1&-1&1\\2&-2&0&0&1&-1\\0&0&2&-2&1&-1\end{array}\right)\,.
\]
{\tt Report.} (i) General kinetics. $\mathrm{rank}\,\Gamma = 2$ and $\Gamma$ is neither $2$-SSD nor SSD (the CRN is neither concordant nor accordant). By Theorem~\ref{thmCRN1}(a), the system has multiple positive equilibria on a stoichiometry class for some choice of power-law general kinetics, and by Theorem~\ref{thmCRN1}(b) the fully open system has multiple positive equilibria for some choice of power-law general kinetics. (ii) Mass action kinetics. $\overline{\Gamma}$ is neither WSD nor $2$-WSD. By Theorem~\ref{thmWSD}(c), the system fails condition \apostt{IC2}, namely the fully open system fails to be injective for some choice of rate constants and inflows and outflows. By Theorem~\ref{thmWSD}, the CRN with mass action kinetics fails condition IC1 for some choice of rate constants. This does not however imply multiple positive equilibria: as a weakly reversible network satisfying the conditions of the deficiency one theorem \cite{feinberg}, it has precisely one positive equilibrium on each nontrivial stoichiometry class (for all choices of rate constants). As stoichiometry classes are bounded and PC1 also holds we can in fact say that with mass action kinetics the CRN has precisely one equilibrium on each stoichiometry class, and this equilibrium is positive provided the stoichiometry class is not $\{0\}$.
\end{example}

\subsection{Examples of CRNs which are not simply reversible}
${}$

\begin{example}[The strongest possible claims III]
\label{exfutile}
The following network is often termed the ``futile cycle'' (\cite{angelileenheersontag} for example): $A+B\rightleftharpoons C\rightarrow A+D,\,\,E+D\rightleftharpoons F\rightarrow E+B$. The stoichiometric matrix $\Gamma$ and DSR graph $G$ are shown:

\begin{center}
\begin{tikzpicture}[domain=0:4,scale=0.45]
\node at (-11,4.5) {$\Gamma = \left(\begin{array}{rrrr}-1&1&0&0\\-1&0&0&1\\1&-1&0&0\\0&1&-1&0\\0&0&-1&1\\0&0&1&-1\end{array}\right)$};
\node at (-1,4.5) {$G_\Gamma = $};
\node at (1,1.5) {$B$};
\fill (1,5) circle (4pt);
\fill (5,1.5) circle (4pt);
\node at (1,7.5) {$C$};
\fill (4,7.5) circle (4pt);
\node at (8,7.5) {$D$};
\node at (4,5) {$A$};
\node at (5,4) {$E$};
\node at (8,1.5) {$F$};
\fill (8,4) circle (4pt);
\draw[<-, thick] (1.5,1.5) -- (4.6,1.5);
\draw[-, thick, dashed] (5.4,1.5) -- (7.3,1.5);
\draw[-, thick, dashed] (5.5,4) -- (7.6,4);

\draw[-, thick, dashed] (1,2.1) -- (1,4.6);
\draw[-, thick] (1, 5.4) -- (1,7);
\draw[<-, thick] (4, 5.5) -- (4,7.1);

\draw[-, thick, dashed] (1.7,7.5) -- (3.6,7.5);
\draw[->, thick] (4.4,7.5) -- (7.4,7.5);

\draw[-, thick, dashed] (1.4,5) -- (3.5,5);
\draw[-, thick] (8,2.1) -- (8,3.6);
\draw[-, thick, dashed] (8, 4.4) -- (8,6.9);
\draw[->, thick] (5,1.9) -- (5,3.55);
\end{tikzpicture}
\end{center}

 {\tt Report.} General kinetics. $G$ satisfies Condition~($*$), and so is concordant and accordant: \apostt{IC1} and \apostt{IC2} hold by Theorem~\ref{thmnoninjgen}. As the CRN is not weakly reversible, PC2 is not automatic from concordance, but PC2 can be computed to hold. As BC1 also holds, each nontrivial stoichiometry class contains a unique equilibrium, which is positive (Remark~\ref{remplus}).

{\tt Remark.} This system also satisfies certain conditions of Theorem~2 in Angeli et al. \cite{angelileenheersontag}, and of Theorem 2.2 in Donnell and Banaji \cite{donnellbanaji}. Either of these theorems can be used to infer that (with general kinetics) all initial conditions on any nontrivial stoichiometry class converge to an equilibrium which is positive and is the unique equilibrium on its stoichiometry class. 
\end{example}

\begin{example}[The strongest possible claims IV]
\label{exWR}
$A+B\rightarrow B+C,\,\,B+C \rightarrow D,\,\,D\rightarrow A+B$. The stoichiometric matrix $\Gamma$, $-Dv^t$, and the DSR graph $G= G_{\Gamma, -Dv}$ are shown:

\begin{tikzpicture}[domain=0:4,scale=0.45]
\node at (-16,4.5) {$\Gamma = \left(\begin{array}{rrr}-1&0&1\\0&-1&1\\1&-1&0\\0&1&-1\end{array}\right)$};
\node at (-7,4.5) {$-Dv^t = \left(\begin{array}{ccc}-&0&0\\-&-&0\\0&-&0\\0&0&-\end{array}\right)$};
\node at (-0.5,4.5) {$G = $};
\node at (1,7) {$A$};
\node at (7,7) {$C$};
\node at (4,4) {$B$};
\node at (4,2) {$D$};
\fill (4,7) circle (4pt);
\fill (1,4) circle (4pt);
\fill (7,4) circle (4pt);

\draw[-, dashed,thick] (1.5,7) -- (3.5,7);
\draw[->, thick] (4.5,7) -- (6.5,7);
\draw[->, thick] (1,4.5) -- (1,6.5);
\draw[->, dashed,thick] (4,4.5) -- (4,6.5);
\draw[-, dashed,thick] (7,4.5) -- (7,6.5);
\draw[->, thick] (1.5,4) -- (3.5,4);
\draw[-, dashed,thick] (4.5,4) -- (6.5,4);
\draw[-, dashed,thick] (1.5,3.7) -- (3.5,2.3);
\draw[<-, thick] (4.5,2.3) -- (6.5,3.7);

\node at (3.5,5.5) {${\scriptstyle \infty}$};
\end{tikzpicture}\\

{\tt Report.} General kinetics. Although the DSR graph fails Condition~($*$), the system is accordant (namely $\Gamma \Bumpeq -Dv^t$), demonstrating that Condition~($*$) is sufficient, but not necessary for accordance. As the system is weakly reversible, accordance implies concordance (Corollary~\ref{corstructdiscord}(iii)). Thus \apostt{IC1} and \apostt{IC2} hold by Theorem~\ref{thmnoninjgen}. Further, PC2 holds automatically as the system is weakly reversible (Remark~\ref{concordtopersist}). As BC1 also holds, each nontrivial stoichiometry class contains a unique equilibrium, which is positive (Remark~\ref{remplus}). 

{\tt Remark.} This example demonstrates that CRNs which are not simple (namely have species occurring on both sides of some reaction) may be accordant and concordant, and hence very well behaved.
\end{example}

The following five examples are all of CRNs which admit no positive equilibria in the sense of Definition \ref{defposeq}. However they nevertheless illustrate various interesting points about injectivity and multistationarity in CRNs.

\begin{example}[Well-behaved on stoichiometry classes and with outflows]
\label{exoutflow}
$A\rightarrow B$, $B+C\rightleftharpoons D$, $2C+A \rightleftharpoons E$. The stoichiometric matrix $\Gamma$ and $-Dv^t$ for this system are:
\[
\Gamma = \left(\begin{array}{rrr}-1 & 0 & -1\\1 & -1 & 0\\0 & -1 & -2\\0 & 1 & 0\\0 & 0 & 1 \end{array}\right)\,, \quad -Dv^t = \left(\begin{array}{ccc}- & 0 & -\\0 & - & 0\\0 & - & -\\0 & + & 0\\0 & 0 & + \end{array}\right)
\]
{\tt Report.} General kinetics: $\Gamma$ has rank $3$, $\Gamma$ and $-Dv^t$ are compatible and $3$-strongly compatible, namely the CRN is accordant and concordant, and \apostt{IC1} and \apostt{IC2} follow by Theorem~\ref{thmnoninjgen}. Stoichiometry classes are bounded, but the system admits no positive equilibria, so stoichiometry classes contain equilibria, but these are all boundary equilibria.  

{\tt Remark.} In this example (and several others to follow), as the CRN admits no positive equilibria, it is conclusion \apostt{IC2} telling us that the fully open system is injective which is likely to be of greatest interest. It is interesting to note that the DSR graph of this CRN satisfies the graph-theoretic condition for concordance in Theorem~2.1 of Shinar and Feinberg \cite{shinarfeinbergconcord2}, although it fails Condition~($*$) in Banaji and Craciun \cite{banajicraciun2}.
\end{example}

\begin{example}[Well-behaved with outflows, but not on stoichiometry classes]
\label{exinjnotmulistat}
Consider the system of two irreversible reactions $A+D\rightarrow B+D,\,\, 2A+D\rightarrow C+D$. The stoichiometric matrix $\Gamma$ and $-Dv^t$ for this system are:
\[
\Gamma = \left(\begin{array}{rr}-1 & -2\\1 & 0\\0 & 1\\0 & 0 \end{array}\right)\,, \quad -Dv^t = \left(\begin{array}{cc}- & -\\0 & 0\\0 & 0\\- & - \end{array}\right)\,.
\]
{\tt Report.} (i) General kinetics. $\Gamma$ has rank $2$ and $\Gamma$ and $-Dv^t$ are compatible but not $2$-strongly compatible, namely the CRN is accordant, but not concordant. By Theorem~\ref{thmnoninjgen}, \apostt{IC2} holds, namely the fully open system forbids multiple equilibria, but the CRN fails IC1 for some choice of power-law general kinetics (Theorem~\ref{thmnoninjgen}). (ii) Mass action kinetics. The CRN fails condition IC1 for some choice of mass action kinetics, namely the vector field is noninjective on the relative interior of some nontrivial stoichiometry class. Again, this clearly does not imply multiple positive equilibria on a stoichiometry class. 

{\tt Remark.} This is an example of a CRN where accordance does not imply concordance as the system is structurally discordant, namely $\Gamma \twost -Dv^t = 0$ for all $Dv$ in the rate pattern (Corollary~\ref{corstructdiscord}(i)). Equivalently, $\mathrm{det}_\Gamma(\Gamma Dv) = 0$ everywhere. 
\end{example}

\begin{example}[Well behaved with mass action but not more generally]
$A+B\rightleftharpoons C,\,\, 2A+2B\rightarrow B + D$. This system has irreversible stoichiometric matrix $\Gamma$, $-Dv^t$, and ${-\Gamma_l}$ as follows:
\[
\Gamma = \left(\begin{array}{rrr}-1 & 1& -2\\-1 &1& -1\\1 & -1 & 0\\0 & 0 & 1 \end{array}\right)\,, \,\, -Dv^t = \left(\begin{array}{ccc}- & 0 & -\\- & 0 & -\\0 & - & 0\\0 & 0 & 0 \end{array}\right)\,, \,\, {-\Gamma_l} = \left(\begin{array}{rrr}-1 & 0& -2\\-1 & 0 &-2\\0 & -1 & 0\\0 & 0 & 0 \end{array}\right)\,.
\]
{\tt Report.} (i) General kinetics. $r = \mathrm{rank}\,\Gamma = 2$, and the system is neither accordant nor concordant, namely none of the following hold: $\Gamma \Bumpeq -Dv^t$, $\Gamma \rst -Dv^t > 0$ or $\Gamma \rst -Dv^t < 0$. By Theorem~\ref{thmnoninjgen}(a) there exists a choice of power-law general kinetics such that the system fails condition IC1. Note however that the system admits no positive equilibria, and hence we cannot claim the existence of multiple positive equilibria on a stoichiometry class for any kinetics. By Theorem~\ref{thmnoninjgen}(b) the fully open system has multiple positive equilibria for some choice of power-law general kinetics. (ii) Mass action kinetics. The CRN is semiaccordant and semiconcordant (namely, $\Gamma \Bumpeq {-\Gamma_l}$ and $\Gamma \rst {-\Gamma_l} > 0$) and so, by Theorem~\ref{gen_powlaw_CRN}(a)~and~(b), IC1a and \apostt{IC2} hold: with mass action kinetics, the CRN is injective on positive stoichiometry classes, and its fully open extension is also injective. Consequently, the CRN forbids multiple positive equilibria on a stoichiometry class, and its fully open extension forbids multiple positive equilibria.

{\tt Remark.} As the system is semiaccordant and semiconcordant, it is normal (Corollary~\ref{cornormal}). This is an example of a normal CRN which is not weakly reversible.
\end{example}

The next system is the same as the previous one, but with the first reaction now irreversible. We see that this change has weakened the claims we are able to make. 

\begin{example}[Setting some reactions to be irreversible can weaken conclusions]
$A+B\rightarrow C,\,\, 2A+2B\rightarrow B + D$. This system has stoichiometric matrix $\Gamma$, $-Dv^t$, and ${-\Gamma_l}$ as follows:
\[
\Gamma = \left(\begin{array}{rr}-1 & -2\\-1 & -1\\1 & 0\\0 & 1 \end{array}\right)\,, \quad -Dv^t = \left(\begin{array}{cc}- & -\\- & -\\0 & 0\\0 & 0 \end{array}\right)\,, \quad {-\Gamma_l} = \left(\begin{array}{rr}-1 & -2\\-1 & -2\\0 & 0\\0 & 0 \end{array}\right)\,.
\]
{\tt Report.} (i) General kinetics. As in the previous example, $r = \mathrm{rank}\,\Gamma = 2$ and the system is neither accordant nor concordant and so, by Theorem~\ref{thmnoninjgen}(b), the fully open system has multiple positive equilibria for some choice of power-law general kinetics. It fails condition IC1 for some choice of power-law general kinetics, but does not admit positive equilibria, so this does not translate into multiple positive equilibria. (ii) Mass action kinetics. The system is semiaccordant ($\Gamma \Bumpeq {-\Gamma_l}$) and so, by Theorem~\ref{gen_powlaw_CRN}(b), \apostt{IC2} holds -- with mass action the fully open system forbids multiple positive equilibria. As neither the system is not semiconcordant (neither of $\Gamma \rst {-\Gamma_l} > 0$ nor $\Gamma \rst {-\Gamma_l} < 0$ holds), Theorem~\ref{gen_powlaw_CRN}(a) tells us that the CRN fails condition IC1 for some choice of rate constants. 

{\tt Remark.} Clearly this CRN fails to be normal (Definition~\ref{defconcord}) as $\Gamma \rst \Gamma_l = 0$ (whenever $\Gamma_l$ has lower rank that $\Gamma$ the failure to be normal is immediate). It is however not structurally discordant (Definition~\ref{defconcord}), illustrating that normal CRNs are a strict subset of those which are not structurally discordant.
\end{example}

\begin{example}[An autocatalytic system]
\label{exautocat}
Consider the simple, autocatalytic system $A\rightarrow B \rightarrow 2A$. Here the stoichiometric matrix $\Gamma$ and $-Dv^t$ are:
\[
\Gamma = \left(\begin{array}{rr}-1 & 2\\1 & -1 \end{array}\right)\,, \quad -Dv^t = \left(\begin{array}{cc}- & 0\\0 & - \end{array}\right)\,.
\]
{\tt Report.} (i) General kinetics. $r = \mathrm{rank}\,\Gamma = 2$, $\Gamma \rst -Dv^t < 0$ (the system is concordant), and PC1 holds, so by Lemma~\ref{thmCRNgen} and Remark~\ref{remgenkin}, claim \apostt{IC1} holds. The system does not however admit any equilibria other than the trivial one. As the system is not accordant, by Theorem~\ref{thmnoninjgen}(b), the fully open system has multiple positive equilibria for some choice of power-law general kinetics. (ii) Mass action kinetics: as $\Gamma$ fails to be WSD, by Theorem~\ref{thmWSD}(c), the system fails condition \apostt{IC2}, namely the fully open system fails to be injective for some choice of rate constants and inflows and outflows.

\end{example}

\section{Concluding remarks}

Results and examples have been presented illustrating a variety of claims about injectivity and multistationarity which can be made about a chemical reaction network, with either mass action or general kinetics, or other related classes of kinetics, primarily using various matrix-related tests. While graph-theoretic approaches have been mentioned only in passing, the practical significance of these approaches becomes particularly important for large systems. Where Condition~($*$) in Appendix~\ref{appstar} implies compatibility of a pair of matrices, and hence accordance of a CRN, an important task for the future is to develop efficient DSR graph conditions for $r$-strong compatibility of a pair of matrices, and hence concordance of a CRN. 

Of the many claims in this paper, we highlight the remarkable parallels between injectivity results for general kinetics and for mass action. For example, given a CRN $\mathcal{R}$ with irreversible stoichiometric matrix $\Gamma$ and corresponding left stoichiometric matrix $\Gamma_l$, and its fully open extension $\mathcal{R}_o$, we have the following parallels:
\begin{enumerate}[align=left,leftmargin=*]
\item {\bf Injectivity on stoichiometry classes}. {\bf Concordance}, namely $\Gamma$-nonsing\-ularity of the qualitative class $\mathcal{Q}(\Gamma_l)$, is equivalent to injectivity of $\mathcal{R}$ in the sense of \apostt{IC1} under the assumption of general kinetics. {\bf Semiconcordance},  namely $\Gamma$-nonsingularity of the semiclass $\mathcal{Q}'(\Gamma_l)$, is equivalent to injectivity of $\mathcal{R}$ in the sense of \apostt{IC1} under the assumption of mass action kinetics.
\item {\bf Injectivity of the fully open system}. {\bf Accordance}, namely $\Gamma \Bumpeq \mathcal{Q}(-\Gamma_l)$, is equivalent to injectivity of $\mathcal{R}_o$ on the nonnegative orthant (i.e., $\mathcal{R}$ satisfies \apostt{IC2}) under the assumption of general kinetics. {\bf Semiaccordance}, namely $\Gamma \Bumpeq \mathcal{Q}'(-\Gamma_l)$, is equivalent to injectivity of $\mathcal{R}_o$ on the nonnegative orthant (i.e., $\mathcal{R}$ satisfies \apostt{IC2}) under the assumption of mass action kinetics.
\item {\bf Nondegeneracy conditions}. Accordance implies concordance if and only if $\mathcal{R}$ is {\bf not structurally discordant}, namely $\mathcal{Q}(\Gamma_l)$ is not $\Gamma$-singular. Semiaccordance implies semiconcordance if and only if $\mathcal{R}$ is {\bf normal}, namely $\mathcal{Q}'(\Gamma_l)$ is not $\Gamma$-singular.
\end{enumerate}
Underlying these parallels is the fact that the derivatives of reaction rates of an irreversible CRN can explore qualitative classes (resp., semiclasses) on $\mathbb{R}^n_{\gg 0}$ under the assumption of general kinetics (resp., mass action). This combines with the fact that whether we assume general kinetics (or some closely related class), or fixed power-law kinetics (with mass action as a special case), {\em collective nonsingularity} of all the allowed systems, namely nonsingularity of each Jacobian, or its restriction to the stoichiometric subspace, is necessary and sufficient for {\em injectivity} of all the associated vector fields, or their restrictions to stoichiometry classes. On the other hand, in all cases there are elegant combinatorial conditions for collective nonsingularity: the ``compatibility'' conditions, relating signs of minors of matrices. 

It is noteworthy that for both general kinetics and mass action, the proof that (collective) nonsingularity implies injectivity on stoichiometry classes uses the fundamental theorem of calculus (Theorems~\ref{thminj}~and~\ref{gen_powlaw}): even though the set of allowed Jacobian matrices of a power-law system is in general not a convex set, by passing to logarithmic coordinates and back again, we can use an essentially convex approach to obtain conclusions about injectivity. The proofs in the other direction, that singularity of some CRN in the class implies the failure of injectivity for some CRN in the class, follow direct constructive approaches where we use the freedom to choose exponents, rate constants, etc., accorded by power-law functions.

Where showing that collective nonsingularity is {\bf equivalent} to collective injectivity is fairly straightforward for the classes of functions encountered here, inferring the existence of multiple equilibria from the failure of collective nonsingularity is trickier. Theorems~\ref{thmnoninjgen}(a),~\ref{thmnoninjgen}(b),~\ref{thmCRN1}(a),~\ref{thmCRN1}(b),~and~\ref{gen_powlaw_CRN}(d) provide conditions for multiple equilibria, but we sometimes need additional conditions beyond the failure of collective nonsingularity (the possibility of positive equilibria in Theorem~\ref{thmnoninjgen}(a); the strong incompatibility condition of Theorem~\ref{gen_powlaw_CRN}(d)). This brings us to the most obvious gap in this work: we do not provide sufficient conditions for the existence of multiple positive equilibria on a stoichiometry class for a non-fully open CRN with $M$-power-law kinetics. Certainly, failure of $M$-concordance is necessary, but may not be sufficient. The question of when, for instance, a mass action CRN which is not fully open is capable of multiple positive equilibria on a stoichiometry class has fundamentally algebraic aspects, beyond the techniques of this paper.  

Another more practical gap in this work involves incomplete algorithmic implementation of the results. For example, analysis of the examples in Section~\ref{secexamples} does not include the results of Theorem~\ref{gen_powlaw_CRN}(d), and so we never in the reports on examples claim definitively the existence of multiple positive equilibria of a fully open system with mass action kinetics: at the time of writing, a check for strong incompatibility of a pair of matrices (Definition~\ref{defstrongincompat}) has not been implemented in {\tt CoNtRol} \cite{control}.

Developments in chemical reaction network theory are occurring rapidly and the intersection of distinct branches of theory has the potential to provide increasingly strong claims about CRNs based on analysis of their structure alone. In the examples above we have already seen hints of this: for instance, in Example~\ref{exWSDandnotSSD} a generic quasiconvergence result based on monotonicity combines with a claim about the existence of a unique equilibrium to allow stronger conclusions. 

\section*{Acknowledgements}
MB's work on this paper was supported by EPSRC grant EP/J008826/1 ``Stability and order preservation in chemical reaction networks''. A large part of CP’s work was completed while at Imperial College London and was supported by Leverhulme grant F/07 058/BU ``Structural conditions for oscillation in chemical reaction networks''. CP was also partially supported by NSF DMS grant 1517577, ``Multistationarity and oscillations in biochemical reaction networks''. We are grateful to Pete Donnell for careful reading and helpful comments on several drafts of this paper, to Stefan M\"uller for useful discussions about necessary and sufficient conditions for injectivity and multistationarity, to Elisenda Feliu for making us aware of an error in an early draft of the paper, to Martin Feinberg for encouraging us to clarify the relationships between some of our results and results on concordance, and to the anonymous referees for helping us improve this paper in several ways.

\appendix

\section{The reduced determinant of a matrix-product}
\label{appreduced}

Let $0 \neq \Gamma \in \mathbb{R}^{n \times m}$ and $V \in \mathrm{R}^{m \times n}$. Let $r = \mathrm{rank}\,\Gamma$. Choose any basis for $\mathrm{im}\,\Gamma$ and write the vectors of this basis as the columns of a matrix $\Gamma_0$. Define $Q$ via $\Gamma = \Gamma_0 Q$, and choose (any) left inverse $\Gamma'$ to $\Gamma_0$ to get $\Gamma'\Gamma = Q$. So $\Gamma = \Gamma_0\Gamma'\Gamma$. 

Given $x \in \mathrm{im}\,\Gamma$, define new coordinates $y$ on $\mathrm{im}\,\Gamma$ via $x = \Gamma_0 y$. We have $\Gamma V \Gamma_0 y = \Gamma_0\Gamma'\Gamma V \Gamma_0 y = \Gamma_0 z$ where $z = \Gamma'\Gamma V \Gamma_0 y$. Thus we have a map $y \mapsto \Gamma'\Gamma V \Gamma_0 y \stackrel{\text{\tiny def}}{=} J_1 y$ which describes the action of $\Gamma V$ in the local coordinates on $\mathrm{im}\,\Gamma$. 

Suppose we choose a different basis for $\mathrm{im}\,\Gamma$, whose vectors are arranged as the columns of a matrix $\Gamma_1$ with left-inverse $\Gamma''$; in a similar way we derive a map $y \mapsto \Gamma''\Gamma V \Gamma_1 y \stackrel{\text{\tiny def}}{=} J_2 y$ which again describes the action of $\Gamma V$ on $\mathrm{im}\,\Gamma$ in the coordinates associated with $\Gamma_1$. It is easy to see that $J_1$ and $J_2$ are similar. Define $R$ via $\Gamma_1 = \Gamma_0 R$; clearly $R$ is (square and) nonsingular since both $\Gamma_0$ and $\Gamma_1$ define bases for $\mathrm{im}\,\Gamma$. Moreover $\Gamma_1R^{-1} = \Gamma_0$ and so $R^{-1} = \Gamma''\Gamma_0$. So
\[
J_2 = \Gamma''\Gamma V \Gamma_1 = \Gamma''\Gamma V \Gamma_0 R = \Gamma''(\Gamma_0 \Gamma'\Gamma) V \Gamma_0 R = R^{-1} \Gamma'\Gamma V \Gamma_0 R = R^{-1}J_1 R
\]
showing that $J_1$ and $J_2$ are similar. Thus although there is no unique choice of matrix describing the action of $\Gamma V$ on $\mathrm{im}\,\Gamma$, since all choices lead to similar matrices their determinant, characteristic polynomial, eigenvalues, etc. are uniquely defined. In particular, given a matrix product $\Gamma V$, we define $\mathrm{det}_{\Gamma}(\Gamma V) = \mathrm{det}(\Gamma'\Gamma V \Gamma_0)$ (with any choice of $\Gamma_0, \Gamma'$ as above) as the ``reduced determinant''of $\Gamma V$.

Clearly if $\mathrm{rank}\,\Gamma = n$, then $\mathrm{det}_{\Gamma}(\Gamma V) = \mathrm{det}(\Gamma V)$. We show that more generally $\mathrm{det}_{\Gamma}(\Gamma V) = \sum_{|\alpha| = r}(\Gamma V)[\alpha]$ where $r = \mathrm{rank}\,\Gamma$, namely $\mathrm{det}_{\Gamma}(\Gamma V)$ is, upto a sign-change, the coefficient of the term of order $n - r$ in the characteristic polynomial of $\Gamma V$. Choose $\alpha' \subseteq \mathbf{n}, \beta' \subseteq \mathbf{m}$ with $|\alpha'| = |\beta'| = r$ such that $\Gamma[\alpha'|\beta'] \neq 0$. Observe that by assumption, $\Gamma_0 \stackrel{\text{\tiny def}}{=} \Gamma(\mathbf{n}|\beta')$ has rank $r$. As above, let $\Gamma'$ be any left-inverse of $\Gamma_0$ so that $\Gamma = \Gamma_0\Gamma'\Gamma$ and define $J_1 = \Gamma' \Gamma V \Gamma_0$ as above. For each $\alpha, \beta$, we have:
\begin{equation}
\label{eqgab}
\Gamma[\alpha|\beta] = (\Gamma_0\Gamma'\Gamma)[\alpha|\beta] = \sum_{|\delta| = r}\Gamma_0[\alpha|\mathbf{r}]\,\Gamma'[\mathbf{r}|\delta]\Gamma[\delta|\beta]\,.
\end{equation}
So:
\begin{eqnarray*}
\sum_{|\alpha| = r}(\Gamma V)[\alpha] & = & \sum_{|\alpha| = |\beta| = r}\Gamma[\alpha|\beta]V[\beta|\alpha]\\
& = & \sum_{|\alpha| = |\beta| = |\delta| = r}\Gamma_0[\alpha|\mathbf{r}]\,\Gamma'[\mathbf{r}|\delta]\Gamma[\delta|\beta]V[\beta|\alpha] \qquad \mbox{(using (\ref{eqgab}))}\\
 & = & \sum_{|\delta| = r}\Gamma'[\mathbf{r}|\delta]\sum_{|\alpha| = |\beta| = r}\Gamma[\delta|\beta]V[\beta|\alpha]\Gamma_0[\alpha|\mathbf{r}]\\
& = & \sum_{|\delta| = r}\Gamma'[\mathbf{r}|\delta]\,(\Gamma V \Gamma_0)[\delta|\mathbf{r}]\\
& = & (\Gamma' \Gamma V \Gamma_0)[\mathbf{r}|\mathbf{r}] =  \mathrm{det}(J_1) = \mathrm{det}_{\Gamma}(\Gamma V)\,.
\end{eqnarray*}

\begin{lemma1}
\label{lemrank}
$\mathrm{det}_\Gamma(\Gamma V) \neq 0$ if and only if $\mathrm{rank}(\Gamma V \Gamma) = r$.
\end{lemma1}
\begin{proof}
Observe that (trivially) $\mathrm{rank}(\Gamma V \Gamma) \leq r$, and $\mathrm{rank}(\Gamma V \Gamma) < r$ if and only if there exists $0 \neq y \in \mathrm{im}\,\Gamma$ such that $\Gamma V y = 0$. On the other hand, choosing $\Gamma_0$ and $\Gamma'$ as above, $\mathrm{det}_\Gamma(\Gamma V) = 0$ if and only if there exists $z \neq 0$ such that $(\Gamma' \Gamma V \Gamma_0)z = 0$. 

Suppose $\mathrm{rank}(\Gamma V \Gamma) < r$, choose nonzero $y \in \mathrm{im}\,\Gamma$ such that $\Gamma V y = 0$, and write $y = \Gamma_0z$ ($z \neq 0$). Immediately, $\Gamma' \Gamma V \Gamma_0z = 0$ so $\mathrm{det}_\Gamma(\Gamma V) = 0$.

Conversely, suppose $\mathrm{det}_\Gamma(\Gamma V) = 0$ and choose $z \neq 0$ such that $\Gamma' \Gamma V \Gamma_0z = 0$. This implies that $\Gamma V \Gamma_0z = 0$ since by definition $\mathrm{im}\,\Gamma \cap \mathrm{ker}\,\Gamma' = \{0\}$. But $0 \neq \Gamma_0z\in \mathrm{im}\,\Gamma$. So $\mathrm{rank}(\Gamma V \Gamma) < r$. 
\hfill
\end{proof}

\begin{lemma1}
\label{lemposdefred}
Let $\Gamma \in \mathbb{R}^{n \times m}$, $V \in \mathrm{R}^{m \times n}$, and $\Gamma V$ be positive definite on $\mathrm{im}\,\Gamma$ in the sense that $0 \neq z \in \mathrm{im}\,\Gamma \Rightarrow z^t\Gamma V z > 0$. Then $\mathrm{det}_\Gamma \Gamma V > 0$.
\end{lemma1}
\begin{proof}
Fix some basis for $\mathrm{im}\,\Gamma$ and, as in the preceding discussion, let $J$ be the matrix describing the action of $\Gamma V$ in this basis, so that $\mathrm{det}_\Gamma\Gamma V = \mathrm{det}J$. By Lemma~\ref{lemrank}, $z^t\Gamma V z > 0$ for all $0 \neq z \in \mathrm{im}\,\Gamma$ implies that $\mathrm{det}_\Gamma\Gamma V \neq 0$, namely $\mathrm{det}J \neq 0$. Consider the spectrum of $J$, namely the list of eigenvalues of $\Gamma V$ associated with $\mathrm{im}\,\Gamma$, say $(\lambda_1, \lambda_2, \ldots, \lambda_r)$. As $\mathrm{det}J \neq 0$, none of these eigenvalues is $0$. If one, say $\lambda_1$, is real and negative, then choosing a corresponding eigenvector $z \in \mathrm{im}\,\Gamma$, we get the contradiction $0 < z^t\Gamma V z = \lambda_1 |z|^2 < 0$. As (i) any real eigenvalues of $J$ are positive and (ii) any nonreal eigenvalues of $J$ come in complex conjugate pairs, the product $\lambda_1\lambda_2\cdots\lambda_r > 0$, namely $\mathrm{det}J = \mathrm{det}_\Gamma\Gamma V > 0$. 
\hfill \end{proof}

\section{General kinetics, weak general kinetics, positive general kinetics}
\label{appkin}
Given a CRN, let $\mathcal{I}_{j, l}$ be the set of indices of species occurring on the left of reaction $j$ and $\mathcal{I}_{j, r}$ be the set of indices of the species occurring on the right of reaction $j$. The following assumptions about the function $v(x)$ in (\ref{reacsys}), apply in the case of ``general kinetics'', where $v$ is assumed to be $C^1$ on $\mathbb{R}^n_{\geq 0}$. They are collectively termed ``Assumption K'':
\begin{enumerate}[align=left,leftmargin=*]
\item[(A)] If reaction $j$ is irreversible then
\begin{enumerate}[align=left,leftmargin=*]
\item[(i)]
$v_j \geq 0$ with $v_j = 0$ if and only if $x_i = 0$ for some $i \in \mathcal{I}_{j, l}$.
\item[(ii)]
$\partial v_j/\partial x_i \geq 0$ for each $i \in \mathcal{I}_{j,l}$. If $x_i > 0$ for all $i \in \mathcal{I}_{j,l}$, then $\partial v_j/\partial x_i > 0$ for each $i \in \mathcal{I}_{j,l}$.
\end{enumerate}
\item[(B)] If reaction $j$ is reversible then
\begin{enumerate}[align=left,leftmargin=*]
\item[(i)]
If $x_{i} = 0$ for some $i \in \mathcal{I}_{j,l}$ (resp., for some $i \in \mathcal{I}_{j,r}$) then $v_j \leq 0$ (resp., $v_j \geq 0$).
\item[(ii)]
If $x_{i} = 0$ for some $i \in \mathcal{I}_{j,l}$ (resp., for some $i \in \mathcal{I}_{j,r}$), then $v_j < 0$ (resp., $v_j > 0$) if and only if $x_{i'} > 0$ for each $i' \in \mathcal{I}_{j,r}$ (resp., for each $i' \in \mathcal{I}_{j,l}$).
\item[(iii)]
If $k \in \mathcal{I}_{j,l}, k \not \in \mathcal{I}_{j,r}$, and $x_i > 0$ for all $i \in \mathcal{I}_{j,l}$ then $\partial v_j(x)/\partial x_k > 0$ (resp., if $k \in \mathcal{I}_{j,r}, k \not \in \mathcal{I}_{j,l}$, and $x_i > 0$ for all $i \in \mathcal{I}_{j,r}$ then $\partial v_j(x)/\partial x_k < 0$). 
\end{enumerate}
\end{enumerate}
These assumptions are similar to the assumptions made in \cite{banajimierczynski}, although there the case where species may occur on both sides of the same reaction was excluded. The reader may confirm that the assumptions here imply the ones in \cite{banajimierczynski} in that case. Note that the assumptions for a reversible reaction are presented for completeness, but can actually be inferred from the assumptions for irreversible reactions.

For ``weak general kinetics'' (Definition~\ref{defgenkin}), where we assume that $v$ is defined and continuous on $\mathbb{R}^n_{\geq 0}$, and $C^1$ on $\mathbb{R}^n_{\gg 0}$, we replace A(ii) with ``$\partial v_j/\partial x_i > 0$ on $\mathbb{R}^n_{\gg 0}$ for each $i \in \mathcal{I}_{j,l}$.''

For ``positive general kinetics'' (Definition~\ref{defgenkin}), where we assume only that $v$ is defined and $C^1$ on $\mathbb{R}^n_{\gg 0}$, Assumption K reduces to Assumption K$_\mathrm{o}$ which consists of:
\begin{enumerate}[align=left,leftmargin=*]
\item[(A$_\mathrm{o}$)] If reaction $j$ is irreversible then (i) $v_j > 0$, (ii) $\partial v_j/\partial x_i > 0$ for each $i \in \mathcal{I}_{j,l}$.
\item[(B$_\mathrm{o}$)] If reaction $j$ is reversible then:  $k \in \mathcal{I}_{j,l}, k \not \in \mathcal{I}_{j,r}$, and $x_i > 0$ for all $i \in \mathcal{I}_{j,l}$ then $\partial v_j(x)/\partial x_k > 0$ (resp., if $k \in \mathcal{I}_{j,r}, k \not \in \mathcal{I}_{j,l}$, and $x_i > 0$ for all $i \in \mathcal{I}_{j,r}$ then $\partial v_j(x)/\partial x_k < 0$). 
\end{enumerate}

The following lemma is a straightforward result. Versions of it have appeared in previous literature, with slightly different technical assumptions (see for example Appendix I of Feinberg \cite{feinberg}).

\begin{lemma1}
\label{lemma:invariance}
Let the system (\ref{reacsys}) satisfy Assumption K. Then for any $x \in \mathbb{R}^{n}_{\geq 0}$, any $j$, and any $i$ such that $x_i = 0$ there holds $\dot x_i = \Gamma_{ij} v_j(x) \ge 0$. Consequently, for such a system $\mathbb{R}^{n}_{\geq 0}$ is forward invariant.
\end{lemma1}
\begin{proof}
The result in fact requires only Assumptions (A)(i) and (B)(i). Let $C_i$ refer to the $i$th species and $R_j$ to the $j$th reaction. Let $x \in \mathbb{R}^{n}_{\geq 0}$ be such that $x_i = 0$.
\begin{itemize}[align=left,leftmargin=*]
\item If $C_i$ does not participate in $R_j$ then $\Gamma_{ij} = 0$, and so $\Gamma_{ij} v_j(x) = 0$.
\item Suppose $R_j$ is irreversible. If $C_i$ occurs on the left of $R_j$ then, by (A)(i), $v_j(x) = 0$, and consequently $\Gamma_{ij} v_j(x) = 0$. If $C_i$ occurs only on the right of $R_j$ then $\Gamma_{ij} > 0$ and, by (A)(i), $v_j \geq 0$, so $\Gamma_{ij} v_j(x) \ge 0$.  
\item Suppose $R_j$ is reversible. If $C_i$ occurs on both sides of $R_j$ then,  by (B)(i), $v_j(x) = 0$. If $C_i$ occurs only on the left of $R_j$, then $\Gamma_{ij} < 0$ and, by B(i), $v_j(x) \leq 0$; consequently $\Gamma_{ij} v_j(x) \ge 0$. If $C_i$ occurs only on the right of $R_j$, then $\Gamma_{ij} > 0$ and, by B(i), $v_j(x) \geq 0$; again $\Gamma_{ij} v_j(x) \ge 0$.
\end{itemize}
Thus $x_i = 0$ implies $\dot x_i = \sum_j \Gamma_{ij} v_j(x) \geq 0$, and so $\mathbb{R}^{n}_{\geq 0}$ is forward invariant. \hfill
\end{proof}

\section{Concordance}
\label{appconcord}
Consider a CRN with irreversible stoichiometric matrix $\Gamma \in \mathbb{R}^{n \times m}$. Let $\Gamma_l \geq 0$ be the left stoichiometric matrix so that, by Assumption K$_o$, $Dv(x) \in \mathcal{Q}(\Gamma_l^t)$ for $x \gg 0$. Let $r = \mathrm{rank}\,\Gamma$. We show that concordance of a system of irreversible reactions as defined by Shinar and Feinberg \cite{shinarfeinbergconcord1} is equivalent to the condition $\Gamma \rst \mathcal{Q}(\Gamma_l) > 0$ or $\Gamma \rst \mathcal{Q}(\Gamma_l) < 0$, which is the form taken by concordance as defined here for such a system (Lemma~\ref{lemconcordbasic}). First observe that
\[
\begin{array}{rcl} && \Gamma \rst \mathcal{Q}(\Gamma_l) > 0\,\,\mbox{ or } \,\, \Gamma \rst \mathcal{Q}(\Gamma_l) < 0\\
&\Leftrightarrow  & [\Gamma \rst M > 0 \,\,\mbox{ or } \,\, \Gamma \rst M < 0] \,\,\forall M \in \mathcal{Q}(\Gamma_l) \quad \mbox{(as $\mathcal{Q}(\Gamma_l^t)$ is path connected)}\\
& \Leftrightarrow & \mathrm{rank}(\Gamma V \Gamma) = r \quad \forall V \in \mathcal{Q}(\Gamma_l^t) \quad \mbox{(Lemma~\ref{lemmain0})}\\
& \Leftrightarrow &[\Gamma V\Gamma y = 0 \,\,\Leftrightarrow\,\, \Gamma y = 0 \quad \forall V \in \mathcal{Q}(\Gamma_l^t)] \qquad (*)\\
\end{array}
\]
We now show that $(*)$ is equivalent to concordance. Consider the negation of $(*)$, namely, ``there exists $y \in \mathbb{R}^m$ and $V \in \mathcal{Q}(\Gamma_l^t)$ such that $\sigma \stackrel{\text{\tiny def}}{=} \Gamma y \neq 0$, but $\Gamma V\sigma = 0$''. In other words, ``there is some $V\in \mathcal{Q}(\Gamma_l^t)$ which can map a nonzero vector in the image of $\Gamma$ to the kernel of $\Gamma$''. 

(i) {\bf If $(*)$ fails, the system is discordant}. Suppose ($*$) fails so there exist $V \in \mathcal{Q}(\Gamma_l^t)$, $\alpha \in \mathrm{ker}\,\Gamma$, and $0 \neq \sigma \in \mathrm{im}\,\Gamma$ such that $\alpha = V\sigma$. Note that $\alpha_i = \sum_jV_{ij}\sigma_j$ and that $V_{ij} \in \mathcal{Q}((\Gamma_l)_{ji})$. Fix $i$. Since $V \geq 0$ and $ V \in \mathcal{Q}(\Gamma_l^t)$:
\begin{enumerate}[align=left,leftmargin=*]
\item If $\alpha_i = (V\sigma)_i = 0$, then either $(\Gamma_l)_{ji} > 0 \Rightarrow \sigma_j = 0$, or there exist $j_1 \neq j_2$ such that $\sigma_{j_1}\sigma_{j_2} < 0$ and $(\Gamma_l)_{j_1i}, (\Gamma_l)_{j_2i} > 0$. 
\item If $\alpha_i = (V\sigma)_i > 0$, there exists $j$ s.t. $(\Gamma_l)_{ji}\sigma_j > 0$; if $\alpha_i = (V\sigma)_i < 0$, there exists $j$ s.t. $(\Gamma_l)_{ji}\sigma_j < 0$. 
\end{enumerate}
The existence of $\alpha \in \mathrm{ker}\,\Gamma$, $0 \neq \sigma \in \mathrm{im}\,\Gamma$ satisfying (1) and (2) above means (by definition) that the system is discordant. 

(ii) {\bf If the system is discordant, then $(*)$ fails.} Suppose the system is discordant, namely there is a pair $0 \neq \sigma \in \mathrm{im}\,\Gamma$, $\alpha \in \mathrm{ker}\,\Gamma$ such that
\begin{enumerate}[align=left,leftmargin=*]
\item Whenever $\alpha_i = 0$, then either $(\Gamma_l)_{ji} > 0 \Rightarrow \sigma_j = 0$, or there exist $j_1 \neq j_2$ such that $\sigma_{j_1}\sigma_{j_2} < 0$ and $(\Gamma_l)_{j_1i}, (\Gamma_l)_{j_2i} > 0$. [discordance condition ii.]
\item Whenever $\alpha_i > 0$, there exists $j$ s.t. $(\Gamma_l)_{ji}\sigma_j > 0$; whenever $\alpha_i < 0$, there exists $j$ s.t. $(\Gamma_l)_{ji}\sigma_j < 0$. [discordance condition i.]
\end{enumerate}
Then there exists $V \in \mathcal{Q}(\Gamma_l^t)$ such that $V\sigma = \alpha$ and hence $\Gamma y \neq 0$, but $\Gamma V\Gamma y = 0$, namely $(*)$ fails. This follows straightforwardly:
\begin{itemize}[align=left,leftmargin=*]
\item If $\alpha_i = 0$ and $(\Gamma_l)_{ji} > 0 \Rightarrow \sigma_j = 0$, then trivially $0 = \alpha_i = (V\sigma)_i$ for any $V \in \mathcal{Q}(\Gamma_l^t)$;
\item If $\alpha_i = 0$ and there exist $j_1 \neq j_2$ such that $\sigma_{j_1}\sigma_{j_2} < 0$ and $(\Gamma_l)_{j_1i}, (\Gamma_l)_{j_2i} > 0$, then we can clearly choose $V^i \in \mathcal{Q}((\Gamma_l)_i)$ (the $i$th row of $V$ in qualitative class of $i$th column of $\Gamma_l$) such that $V^i\sigma = 0 = \alpha_i$. 
\item If $\alpha_i > 0$ and there exists $j$ s.t. $(\Gamma_l)_{ji}\sigma_j > 0$, , then we can clearly choose $V^i \in \mathcal{Q}((\Gamma_l)_i)$ such that $V^i\sigma = \alpha_i$.
\item If $\alpha_i < 0$ and there exists $j$ s.t. $(\Gamma_l)_{ji}\sigma_j < 0$, , then we can clearly choose $V^i \in \mathcal{Q}((\Gamma_l)_i)$ such that $V^i\sigma = \alpha_i$.
\end{itemize}

\section{Additional information}
\label{appadditional}

{\bf Claim BC1 (bounded stoichiometry classes).} The following is well known (Appendix 1 of Horn and Jackson \cite{hornjackson} for example):
\begin{lemma1}
\label{lembdclass}
Stoichiometry classes of a CRN with stoichiometric matrix $\Gamma \in \mathbb{R}^{n \times m}$ are bounded if and only if there exists $0 \ll p \in \mathrm{ker}\,\Gamma^t$.
\end{lemma1}
\begin{proof}
Recall that stoichiometry classes are simply the intersections of cosets of $\mathrm{im}\,\Gamma$ with $\mathbb{R}^n_{\geq 0}$, and that $\mathrm{ker}\,\Gamma^t$ is the orthogonal complement of $\mathrm{im}\,\Gamma$. The proof is now easily inferred from Theorem~4 in Ben-Israel \cite{benisrael}. 
\hfill \end{proof}

Note that if BC1 holds then all nonempty stoichiometry classes are bounded polyhedra. As they are also forward invariant under the local semiflow generated by (\ref{reacsys}) (Lemma~\ref{lemma:invariance} in Appendix~\ref{appkin}), they contain equilibria as a consequence of the Brouwer fixed point theorem (\cite{spanier} for example).

{\bf Claim PC0 (nonexistence of positive equilibria).} In many situations a very simple result on the nonexistence of positive equilibria can be applied. The following lemma is an amalgamation of easy and well-known facts (see for example Section 5.3. of Feinberg \cite{feinberg} for related results).

\begin{lemma1}
\label{lemnoposeq}
Let $\Gamma\in \mathbb{R}^{n \times m}$ be the stoichiometric matrix of a CRN, $\overline{\Gamma} \in \mathbb{R}^{n \times m'}$ the corresponding irreversible stoichiometric matrix. (i) If $\mathrm{ker}\,\overline{\Gamma} \cap \mathbb{R}^n_{\gg 0} = \emptyset$ then, with general kinetics on $\mathbb{R}^n_{\gg 0}$, or any power-law kinetics, the CRN has no positive equilibria. (ii) If $\mathrm{ker}\,\overline{\Gamma} \cap \mathbb{R}^n_{\gg 0} \not = \emptyset$, then the CRN with mass action kinetics has a positive equilibrium for some choice of rate constants.
\end{lemma1}
\begin{proof}
(i) We prove the contrapositive. Let $v\colon \mathbb{R}^n \to \mathbb{R}^m$ be the rate function of the original CRN (not necessarily irreversible). Without loss of generality let reactions $1,\ldots,r$ be reversible, and reactions $r+1, \ldots, m$ irreversible. Suppose the CRN has an equilibrium $x_{eq} \gg 0$ and let $w = v(x_{eq})$ so that $\Gamma w=0$. Assuming only that reaction rates of irreversible reactions on $\mathbb{R}^n_{\gg 0}$ are positive (certainly true for positive general kinetics, or any power-law kinetics), $w_{r+1}, \ldots, w_m$ are all positive. Let 
\[
\overline{\Gamma} = (\Gamma_1\,|\,-\!\Gamma_1\,|\,\cdots \,|\, \Gamma_r\,|\,-\!\Gamma_r\,|\,\Gamma_{r+1}\,|\,\Gamma_{r+2}\,|\,\cdots\,|\, \Gamma_m)
\]
be the irreversible stoichiometric matrix of the system. For $k = 1, \ldots, r$, define $w_{k+} = 1+\max\{w_{k},0\}$, $w_{k-} = 1-\min\{w_{k},0\}$, so that $w_{k+},w_{k-} > 0$ and $w_k = w_{k+} - w_{k-}$. Define
\[
\overline{w} = (w_{1+}, w_{1-}, \ldots, w_{r+}, w_{r-}, w_{r+1}, \ldots, w_m)^t\,.
\]
By construction $\overline{w} \gg 0$, and clearly $\overline{\Gamma}\overline{w} = \Gamma w = 0$, and thus $\mathrm{ker}\,\overline{\Gamma} \cap \mathbb{R}^n_{\gg 0} \not = \emptyset$. 

(ii) Let $0 \ll z \in \mathrm{ker}\,\overline{\Gamma}$. Define $x = \mathbf{1}$ and choose $E \in \mathcal{D}_m$ via $E_{ii} = z_i$. Then for any matrix of exponents $M \in \mathbb{R}^{m \times n}$ (including, in particular, $M = \Gamma_l^t$, where $\Gamma_l$ is the left stoichiometric matrix corresponding to $\overline{\Gamma}$), $\overline{\Gamma} E \exp(M \ln x) = \overline{\Gamma} E \mathbf{1} = \overline{\Gamma} z = 0$, and thus $x$ is a positive equilibrium of the system.
\hfill \end{proof}

\begin{remark}
A variety of conditions on a network with mass action kinetics are known to guarantee that it has a positive equilibrium for {\em all} choices of rate constants. However, $\mathrm{ker}\,\overline{\Gamma} \cap \mathbb{R}^n_{\gg 0} \not = \emptyset$ is not sufficient -- see Remark 5.3B in Feinberg \cite{feinberg}. 
\end{remark}

{\bf Claims PC1 and PC2 (persistence of solutions).} A {\em siphon} of a CRN is a nonempty subset $\Sigma$ of the chemical species such that (under the assumption of general kinetics) if all species in $\Sigma$ are absent, then no reaction is able to produce any species of $\Sigma$. Corresponding to siphon $\Sigma$ is a subset $F_\Sigma$ of $\partial \mathbb{R}^n_{\geq 0}$ where all concentrations of species from $\Sigma$ are zero, and all others nonzero, termed a ``siphon face'' in \cite{donnellbanaji}; it is easy to show (for a CRN with general kinetics, and in fact under considerably weaker assumptions) that all nonzero $\omega$-limit points of the system on $\partial \mathbb{R}^n_{\geq 0}$ must in fact lie on siphon faces (see Banaji and Mierczy\'nski \cite{banajimierczynski} for example). A siphon is termed ``critical''  if a nontrivial stoichiometry class intersects the corresponding siphon face. Note that in the literature siphons have also been called ``semilocking sets'', and critical siphons have also been called ``relevant'', while non-critical siphons have been termed ``stoichiometrically infeasible'', and ``structurally nonemptiable''. 

{\bf PC2} occurs if the system has no critical siphons, in which case no positive initial condition can have an $\omega$-limit point on $\partial \mathbb{R}^n_{\geq 0}$. The absence of critical siphons, implying ``structural persistence'' of the CRN, occurs if for each siphon $\Sigma$, there exists a nonnegative and nonzero vector in $\mathrm{ker}\,\Gamma^t$ orthogonal to $F_\Sigma$, or in the terminology of \cite{angelipetrinet}, each siphon contains the ``support of a P-semiflow''. This condition for the absence of critical siphons is also stated without proof in Remark 6.1.E of Feinberg \cite{feinberg}. Verification of this condition involves checking whether certain linear equalities and inequalities are satisfiable and is easily implemented computationally. Details and an example of the calculations are provided in Donnell and Banaji \cite{donnellbanaji} while the computations are implemented in {\tt CoNtRol} \cite{control}. Where the calculations implemented in \cite{control} involve linear programming, an algebraic algorithm for verifying the absence of critical siphons is  given in Shiu and Sturmfels \cite{shiusturmfels}. Note also that if a network is concordant and weakly reversible then the absence of critical siphons is automatic by Theorem~6.3 in Shinar and Feinberg \cite{shinarfeinbergconcord1}, reproved by elementary means in the next appendix.

{\bf PC1} is satisfied if the system has no siphons, other than possibly the set of all species, corresponding to siphon face $\{0\}$. If the set of all species is a siphon, then it is non-critical (an indirect consequence in this case is that stoichiometry classes must be bounded as the stoichiometric subspace has trivial intersection with the nonnegative orthant). Since a CRN satisfying PC1 either has no siphons, or a single non-critical siphon at the origin, it also satisfies PC2.

\section{Elementary proof that a concordant, weakly reversible CRN has no critical siphons}
\label{appWRconcord}
Shinar and Feinberg \cite{shinarfeinbergconcord1} provide two proofs of the claim reproved by direct means in this appendix. One uses unpublished results of Deng et al. \cite{dengWR}, who prove that nontrivial stoichiometry classes of weakly reversible CRNs with mass action kinetics contain positive equilibria. The second proof relies on classical results of Horn and Jackson on the existence and uniqueness of complex balanced equilibria for CRNs with mass action kinetics \cite{hornjackson}. As both hypotheses and conclusions of the theorem are fundamentally linear algebraic/combinatorial, and have little to do with chemical kinetics at all (as remarked by the authors themselves), we provide a proof which does not rely on results for mass action systems. We require the following fact:
\begin{lemma1}[Theorem 7.2 in Craciun and Feinberg \cite{Craciun.2010ac}]
\label{WRnormal}
Every weakly reversible CRN is normal.
\end{lemma1}
\begin{proof}
Consider a CRN with irreversible stoichiometric matrix $\Gamma \in \mathbb{R}^{n \times m}$ and left stoichiometric matrix $\Gamma_l$ and whose complex digraph $G$ is weakly reversible. Assume that $G$ has no loops: this assumption is without loss of generality as $\mathrm{det}_\Gamma(\Gamma V) = \sum_{|\alpha| = |\beta| =\mathrm{rank}\,\Gamma}\Gamma[\alpha|\beta]V[\beta|\alpha]$ is unchanged (for arbitrary $V$ of appropriate dimension) by the addition of a column of zeros to $\Gamma$.

To prove the lemma we will show that $\Gamma D\Gamma_l^t$ is negative definite on $\mathrm{Im}\,\Gamma$ for some $D \in \mathcal{D}_m$, namely $z^t\Gamma D \Gamma_l^tz < 0$ for all $0 \neq z \in \mathrm{Im}\,\Gamma$. Then $\mathrm{rank}\,\Gamma D\Gamma_l^t\Gamma =\mathrm{rank}\,\Gamma$, implying normality since $D\Gamma_l^t \in \mathcal{Q}'(\Gamma_l^t)$. Specifying some order on the complexes, let $Y$ be the matrix of complexes, and $\Theta$ the (signed) incidence matrix of $G$ (namely $\Theta_{ik} = -1, \Theta_{jk}=1$ iff edge $k$ is $(i,j)$), so that $\Gamma = Y\Theta$. Define $\Theta_l = \Theta_{-}$ and $\Theta_r = \Theta_{+}$, so that $\Gamma_l = Y\Theta_l$. For arbitrary $Y$, we aim to construct $D$ such that 
\[
z^tY\Theta D \Theta_l^tY^tz < 0\quad \mbox{or equivalently} \quad -z^tY(\Theta D \Theta_l^t + \Theta_lD\Theta^t)Y^tz > 0 
\]
for all $0 \neq z \in \mathrm{Im}\,Y\Theta$. Each column of $\Theta_l$ (resp., $\theta_r$) has exactly one nonzero entry and so its nonzero rows form an orthogonal basis for $\mathrm{Im}\,\Theta_l^t$. Define:
\[
L(G) \stackrel{\text{\tiny def}}{=} (-\Theta\Theta_l^t)_{ij} \!=\! \sum_k\Theta_{ik}(-\Theta_l)_{jk} = \left\{\begin{array}{rl}q&\mbox{if $i=j$ and vertex $i$ has $q$ out-edges}\\-r& \mbox{if there are $r$ edges $(i,j)$ in $G$}\\0 &\mbox{otherwise}\end{array}\right.
\]
We see that $L(G)$ is the transpose of a digraph analogue of the Laplacian matrix of a graph (where the diagonal entries count the outdegree of a vertex, and the off-diagonal $ij$ entry counts the number of edges $(i,j)$). 

Let $C$ be a cycle in $G$ and consider the subgraph $G_C$ which has all the vertices of $G$ but edges only from $C$. The incidence matrix $\Theta_C$ of $G_C$ is simply $\Theta$ with all entries corresponding to edges (namely columns of $\Theta$) not in $C$ set to zero. We can confirm easily that $L(G_C) \stackrel{\text{\tiny def}}{=} \Theta_C\Theta_{C,l}^t = \Theta_C\Theta_l^t$. Consider the matrix $\overline \Theta \stackrel{\text{\tiny def}}{=} \sum_{i}\Theta_{C_i}$ where the sum is over all cycles. If edge $j$ occurs in $k_j \in \mathbb{N}_0$ cycles, then column $j$ of $\Theta$ appears as column $j$ of $\Theta_{C_i}$ for $k_j$ values of $i$, and so column $j$ of $\overline \Theta$ is simply $k_j$ times column $j$ of $\Theta$. If $G$ is weakly reversible then each edge occurs in some cycle, namely $k_j \geq 1$ for each $j$, and consequently $\overline \Theta = \Theta D$ where $D = \mathrm{diag}(k_1, k_2, \ldots) \in \mathcal{D}_m$.

Now fix a cycle $C$ and observe that $\Theta_C\Theta_C^t = -(\Theta_C\Theta_{C,l}^t + \Theta_{C,l}\Theta_C^t)$: this is a straightforward computation, where we need only note that $\Theta_{C,r}\Theta^t_{C,r} = \Theta_{C,l}\Theta^t_{C,l}$ as each nonzero row of $\Theta_C$ contains exactly one $+1$ and one $-1$, and all rows of $\Theta_{C,l}$ (resp., $\Theta_{C,r}$) are orthogonal. Thus
\[
-(\Theta D \Theta_l^t + \Theta_l D \Theta^t) = \sum_i-(\Theta_{C_i}\Theta_l^t + \Theta_l\Theta_{C_i}^t) = \sum_i\Theta_{C_i}\Theta_{C_i}^t = \tilde\Theta\Tilde\Theta^t
\]
where $\tilde\Theta = [\Theta_{C_1}|\Theta_{C_2}|\cdots|\Theta_{C_k}]$, and hence, for any $z \in \mathbb{R}^n$, 
\[
-z^tY(\Theta D \Theta_l^t + \Theta_lD\Theta^t)Y^tz = z^tY\tilde\Theta{\tilde\Theta}^tY^tz  = |{\tilde\Theta}^tY^tz|^2 \geq 0\,.
\]
with $z^tY\tilde\Theta{\tilde\Theta}^tY^tz=0$ iff ${\tilde\Theta}^tY^tz = 0$. On the other hand, $\mathrm{ker}\,\tilde\Theta^t = \mathrm{ker}\,\Theta^t$ as $\tilde\Theta^t$ includes precisely the rows of $\Theta_t$, perhaps reordered and with some repetition, and no others. So ${\tilde\Theta}^tY^tz = 0$ iff $\Theta^tY^tz = 0$. Fixing $0 \neq z \in \mathrm{Im}\,Y\Theta$, $\Theta^tY^tz \neq 0$, and consequently $-z^tY(\Theta D \Theta_l^t + \Theta_lD\Theta^t)Y^tz > 0$ as desired. \hfill \end{proof}

It is useful to note the following immediate corollary:
\begin{cor}
\label{corWRconcord}
Given a weakly reversible CRN with irreversible stoichiometric matrix $\Gamma$ and corresponding left-stoichiometric matric $\Gamma_l$, there exists positive diagonal matrix $D$ such that $\mathrm{det}_\Gamma(-\Gamma D \Gamma_l^t) > 0$. 
\end{cor}
\begin{proof}
By the proof of Lemma~\ref{WRnormal}, there exists positive diagonal $D$ such that $-\Gamma D \Gamma_l^t$ is positive definite on $\mathrm{im}\,\Gamma$. By Lemma~\ref{lemposdefred}, this implies that $\mathrm{det}_\Gamma(-\Gamma D \Gamma_l^t) > 0$. 
\hfill \end{proof}

\begin{thm}[Theorem~6.3 in Shinar and Feinberg \cite{shinarfeinbergconcord1}]
\label{structpersistthm}
A weakly reversible CRN with a critical siphon is discordant. 
\end{thm}
\begin{proof}
Consider a weakly reversible CRN on $n$ species with a critical siphon $\emptyset \neq \Sigma \subseteq \mathbf{n}$. Let $\Gamma, \Gamma_l \in \mathbb{R}^{n \times m}$ be the irreversible stoichiometric matrix and left stoichiometric matrix of the system respectively, and $G$ its complex digraph. The strategy will be to find $0 \neq y' \in \mathrm{im}\,\Gamma$, and $M \in \mathcal{Q}(\Gamma_l^t)$, such that $\Gamma My' = 0$, thus showing that the system is discordant.  

Let $S = \mathbf{n}\backslash \Sigma$ be the set of $0 \leq k < n$ species not in the siphon. Define
\[
F = \{x \in \mathbb{R}^n\,|\, x_i > 0, i \in S\,\,\, \mbox{and}\,\,\, x_i = 0, i \not \in S\}.
\]
to be the corresponding face of $\mathbb{R}^n_{\geq 0}$. Order the species and reactions so that:
\[
\Gamma = \left[\begin{array}{cc}\left.\begin{array}{c}\Gamma^S\\\hline 0\end{array}\right| & \Gamma^0\end{array}\right], \quad \Gamma_l = \left[\begin{array}{cc}\left.\begin{array}{c}\Gamma_l^S\\\hline 0\end{array}\right| & \Gamma_l^0\end{array}\right]\,.
\]
Here the superscript $S$ refers to reactions involving species {\em only} from $S$, while the superscript $0$ refers to the remaining reactions. We allow $\Gamma^S$ and $\Gamma_l^S$ to be empty matrices in the case that $S$ is empty (corresponding to $\mathbf{n}$ being the critical siphon): all arguments remain valid in this case. Let $\Gamma_S$ have $s$ columns, where $0 \leq s < m$: whether or not $S$ is empty, the set of reactions supported on $S$ may be empty, but cannot be all of the reactions as $\Sigma$ is critical -- if $\Gamma^0$ were empty, then $\mathrm{im}\,\Gamma \subseteq \mathrm{span}\,F$, and $\Sigma$ would then clearly be noncritical. We refer to the reactions of $\Gamma^S$ as the $S$-reactions, and the others as the non-$S$-reactions. 

Observe that a non-$S$-reaction can share no complexes with an $S$-reaction: otherwise, as $\Sigma$ is a siphon, any shared complexes must figure only as product complexes for non-$S$-reactions, violating weak reversibility. Thus the $S$-reactions and non-$S$-reactions are made up from distinct sets of connected components of $G$, namely each forms a weakly reversible subsystem of the CRN (this is also observed in the original proof of Shinar and Feinberg \cite{shinarfeinbergconcord1}). Consequently, as any union of cycles in $G$ corresponds to a nonnegative vector in $\mathrm{ker}\,\Gamma$ with support precisely equal to the reactions in these cycles, there exists a strictly positive vector $q^0 \in \mathbb{R}^{m-s}$ in $\mathrm{ker}\,\Gamma^0$. 

Choose $p \in \mathbb{R}^n_{\gg 0}$ and a vector $y \in \mathrm{im}\,\Gamma$ such that $p - y \in F$. Geometrically, we begin at a point on the critical siphon face $F$ and move along its stoichiometry class into $\mathbb{R}^n_{\gg 0}$, possible by definition of a critical siphon: $y$ is the vector travelled. Let 
\[
y = \left[\begin{array}{c}y^S\\y^0\end{array}\right]
\]
where $y^S$ and $y^0$ have their natural meanings. By construction, entries in $y^0$ are all positive while the signs of entries in $y^S$ are unknown. In the case $k=0$, $y^S$ is empty, and $y$ is a strictly positive vector. 

We now consider the subsystem of $S$-reactions, assuming for the time being that it is nonempty, namely $s > 0$ (and hence certainly $k > 0$). As it is a weakly reversible system of reactions it is normal (Lemma~\ref{WRnormal}) and hence not structurally discordant, and so there exists {\em some} $M^S \in \mathcal{Q}((\Gamma_l^S)^t)$ such that $\mathrm{im}\,M^S\Gamma^S \oplus \mathrm{\ker}\,\Gamma^S = \mathbb{R}^s$. As a consequence, we can write $M^Sy^S = \alpha + \beta$ where $\alpha \in \mathrm{im}\,M^S\Gamma^S$ and $\beta \in \mathrm{ker}\,\Gamma^S$. Choosing $y^{S,1} \in \mathrm{im}\,\Gamma^S$ such that $M^Sy^{S,1} = -\alpha$, and hence $M^S(y^S + y^{S,1}) = \beta$, we see that $\Gamma^SM^S(y^S + y^{S,1}) = 0$. Now set
\[
y' = \left[\begin{array}{c}y^S + y^{S,1}\\y^0\end{array}\right] \quad \mbox{and write $M \in \mathcal{Q}((\Gamma_l)^t)$ as} \quad M = \left[\begin{array}{c}M^S\,\,|\,\, 0\\\hline M^0\end{array}\right].
\]
Choose the entries of $M^0$ (corresponding to non-$S$-reactions) as follows. For $i > s$, reaction $i$ must have a species on the left in $\Sigma$, for otherwise (since $\Sigma$ is a siphon) it would produce only species in $S$, namely there exists $j > k$ such that $(\Gamma_l)_{ji} > 0$. Since $y'_j > 0$, we can choose entries in row $i$ of $M \in \mathcal{Q}(\Gamma_l^t)$ such that $(My')_i = q^0_{i-s} > 0$, namely $\Gamma^0M^0y' = \Gamma^0q^0 = 0$. We now have
\begin{eqnarray*}
\Gamma M y' & = & \left[\begin{array}{cc}\left.\begin{array}{c}\Gamma^S\\\hline 0\end{array}\right| & \Gamma^0\end{array}\right]\,\left[\begin{array}{c}M^S\,\,|\,\, 0\\\hline M^0\end{array}\right]\left[\begin{array}{c}y^S + y^{S,1}\\y^0\end{array}\right]\\
& = & \left[\begin{array}{cc}\Gamma^SM^S&0\\0&0\end{array}\right]\left[\begin{array}{c}y^S + y^{S,1}\\y^0\end{array}\right] + \Gamma^0M^0\,\left[\begin{array}{c}y^S + y^{S,1}\\y^0\end{array}\right]\\
& = & \left[\begin{array}{c}\Gamma^SM^S(y^S + y^{S,1})\\0\end{array}\right] + \Gamma^0M^0y' = 0 + 0 = 0\,.
\end{eqnarray*}
As $0 \neq y' \in \mathrm{im}\,\Gamma$, this concludes the proof that the system is discordant in the case that the $S$-subsystem is nonempty. In case the $S$-subsystem is empty, $\Gamma = \Gamma^0$, $M = M^0$ and either: 
\begin{enumerate}[align=left,leftmargin=*]
\item[(i)] $k=0$ in which case $y = y^0$ (i.e., $y^S$ is empty). Then $\Gamma M y = \Gamma^0M^0y^0$; as above we can construct $M^0$ such that $M^0y^0 = q^0$, giving $\Gamma^0M^0y^0 = 0$. 
\item[(ii)] Otherwise $y^S$ is not empty, but we choose $y^{S,1} = 0$. We apply the construction above to get $M^0$ such that $M^0y = q^0$, so $\Gamma M y = \Gamma^0M^0y = \Gamma^0q^0 = 0$.
\end{enumerate}
This completes the proof.\hfill \end{proof}

\section{DSR graphs and Condition~($*$)}
\label{appstar}
We follow the constructions in Angeli et al. \cite{angelibanajipantea} based on those in Banaji and Craciun \cite{banajicraciun2} though with minor technical differences, the most important of which is that what is here termed $G_{A,B}$ was termed $G_{A, B^t}$ in \cite{banajicraciun2}. A great deal more explanation and justification for the construction of the DSR graph, and explanation of Condition~($*$), are given in \cite{banajicraciun2}.

{\bf DSR graph of a matrix pair.} Given $A \in \mathbb{R}^{n \times m}$, $B \in \mathbb{R}^{m \times n}$, construct a signed, labelled, bipartite, generalised graph $G_{A, B}$ on $n + m$ vertices as follows: beginning with $n$ vertices $X_1, \ldots, X_n$ and $m$ vertices $Y_1, \ldots, Y_m$, add the directed edge $Y_jX_i$ iff $A_{ij} \not = 0$, and give this edge the sign of $A_{ij}$ and label $|A_{ij}|$; add the directed edge $X_iY_j$ iff $B_{ji} \not = 0$, and give this edge the sign of $B_{ji}$ (and no label). If a pair of edges $X_iY_j$ and $Y_jX_i$ both exist and have the same sign, merge these into a single undirected edge with the label inherited from $X_iY_j$. Note that edges of $G_{A, B}$ may be directed from an $X$-vertex to a $Y$-vertex, or from a $Y$-vertex to an $X$-vertex, or undirected. Any edges which remain unlabelled at the end of the construction must be directed from $X$ vertex to $Y$ vertex and can be given the formal label $\infty$. We follow the convention that edge-labels of $1$ are omitted.

{\bf SR graph of a matrix.} $G_A$, the SR graph of a matrix $A$ is the DSR graph $G_{A, A^t}$. By construction, all edges in $G_{A, A^t}$ are undirected and have finite labels, namely the magnitudes of entries of $A$. $G_A$ can thus be seen as a representation of the matrix. 

{\bf DSR graph of a CRN.} Given a CRN $\mathcal{R}$ with stoichiometric matrix $\Gamma$ and rate vector $v(x)$ in (\ref{reacsys}), for each $x \in \mathbb{R}^n_{\gg 0}$, define $G_{\Gamma, -Dv(x)}$ to be the DSR graph of the matrix pair $\Gamma, -Dv(x)$. The DSR graph $G$ of $\mathcal{R}$ is a formal union of the DSR graphs $G_{\Gamma, -Dv(x)}$ as follows: all $G_{\Gamma, -Dv(x)}$ have the same vertex set and this is the vertex set of $G$; the edge set of $G$ is the union of edges of all the $G_{\Gamma, -Dv(x)}$, where two edges are considered to be equal if they have the same direction, sign and label. Assumption K$_o$ allows the DSR graph of {\em any} CRN to be constructed from knowledge of the reactions alone (see the DSR graphs shown in Section~\ref{secexamples}). Note, however, that the DSR graph of a CRN differs depending on whether a reversible reaction is treated as a single reaction or a pair of irreversible reactions, and the choice to treat a reversible reaction as a pair of irreversible reactions, or not, can affect whether Condition~($*$) below holds for the DSR graph of the CRN.

{\bf SR graph of a CRN.} The SR graph of a CRN is just the DSR graph with all direction removed from edges. 

{\bf Cycles in DSR graphs.} Consider a DSR graph $G$ some of whose edges may be directed: each edge $e \in E(G)$ has a sign $\pm 1$, and a numerical label $l(e) \in \mathbb{N} \cup \{\infty\}$. A cycle $c$ in a DSR graph $G$ is a path from some vertex to itself which repeats no other vertices, and which respects the orientation of any directed edges traversed. Its length $|c|$ is the number of edges (or vertices) in $c$, and the sign of $c$ is the product of the edge-signs in $c$. As $G$ is bipartite, any cycle $c$ has even length and we can define:
\[
P(c) = (-1)^{|c|/2}\mathrm{sign}(c).
\]
$c$ is termed an {\bf e-cycle} if $P(c) = 1$, and an {\bf o-cycle} otherwise. If $c = (e_1, e_2, \ldots, e_{2r})$, then $c$ is an {\bf s-cycle} if all edges in $c$ have finite labels and
\[
\prod_{i = 1}^{r}l(e_{2i-1}) = \prod_{i = 1}^{r}l(e_{2i}).
\]
Two oriented cycles in $G$ are compatibly oriented if each induces the same orientation on every edge in their intersection. Two cycles (possibly unoriented) in $G$ are compatibly oriented if there is an orientation for each so that this requirement is fulfilled. Two cycles of $G$ have {\bf S-to-R intersection} if they are compatibly oriented and each component of their intersection contains an odd number of edges (this is trivially fulfilled if their intersection includes no edges).

{\bf Condition~($*$).} A DSR graph $G$ satisfies Condition~($*$) if all its e-cycles are s-cycles, and no two e-cycles have S-to-R intersection.

Note that Condition~($*$) is sufficient but not necessary to ensure that a CRN is accordant, namely that $\Gamma \Bumpeq -V^t$ for each $V$ in the rate pattern of the CRN. Construction of the DSR graph of a CRN, and calculation of whether it satisfies Condition~($*$), are automated in {\tt CoNtRol} \cite{control}.

\bibliographystyle{siam}

\end{document}